\documentclass[leqno,11pt]{amsart}
\usepackage{amssymb, amsmath}
\usepackage{amssymb,amsfonts}
\usepackage{amsmath,latexsym}
\usepackage{pdfsync}
\usepackage{bbm}
\usepackage{verbatim}
\usepackage{color}
\usepackage{appendix}
\usepackage{amsmath}\allowdisplaybreaks[3]
\newcommand{\andf}{\quad\hbox{and}\quad}
\newcommand{\with}{\quad\hbox{with}\quad}

\newcommand{\newcom}{\newcommand}

\def\inte#1{
\displaystyle\mathop{#1\kern0pt}^\circ }

\newcom{\al}{\alpha}
\newcom{\de}{\delta}
\newcom{\Th}{\Theta}
\newcom{\be}{\beta}
\newcom{\s}{\sigma}
\newcom{\eps}{\epsilon }
\newcom{\ve}{\varepsilon}
\newcom{\ga}{\gamma}
\newcom{\Ga}{\Gamma}
\newcom{\ka}{\kappa}
\newcom{\Lam}{\Lambda}
\newcom{\lam}{\lambda}
\newcom{\bp}{\Phi}
\newcom{\om}{\omega}
\newcom{\Sig}{\Sigma}
\newcom{\sig}{\sigma}
\newcom{\tht}{\theta}
\newcom{\tri}{\triangle}
\newcom{\oo}{\infty}
\newcom{\h}{{\rm h}}
\newcom{\rmv}{{\rm v}}
\newcom{\hs}{\hslash}
\newcom{\vphi}{\varphi}
\def\dive{\mathop{\rm div}\nolimits}
\newcom{\grad}{\nabla}
\newcom{\lap}{\Delta}
\newcom{\cB}{{\mathcal B}}
\newcom{\cC}{{\mathcal C}}
\newcom{\cD}{{\mathcal D}}
\newcom{\cF}{{\mathcal F}}
\newcom{\cL}{{\mathcal L}}
\newcom{\cM}{{\mathcal M}}
\newcom{\cP}{{\mathcal P}}
\newcom{\cS}{{\mathcal S}}
\newcom{\cQ}{{\mathcal Q}}
\newcom{\cT}{{\mathcal T}}
\newcom{\cY}{{\mathcal Y}}
\newcom{\cZ}{{\mathcal Z}}
\newcom{\R}{\Bbb R}
\newcom{\T}{\Bbb T}
\newcom{\N}{\Bbb N}
\newcom{\Z}{\Bbb Z}
\newcom{\C}{\Bbb C}
\newcom{\E}{\Bbb E}

\newcom{\f}{\frac}
\newcom{\dint}{\displaystyle\int}
\newcom{\dsum}{\displaystyle\sum}
\newcom{\dlim}{\displaystyle\lim}
\newcom{\ov}{\overline}
\newcom{\wt}{\widetilde}
\newcom{\pt}{\partial_t}
\newcom{\p}{\partial}
\newcom\na{\nabla}
\newcom\rto{\rightarrow}
\newcom\lto{\leftarrow}
\newcom\mto{\mapsto}
\newcom{\disp}{\displaystyle}
\newcom{\non}{\nonumber}
\newcom{\no}{\noindent}
\newcom{\QED}{$\square$}

\def\eqdefa{\buildrel\hbox{\footnotesize def}\over =}

\newcommand{\beq}{\begin{equation}}
\newcommand{\eeq}{\end{equation}}
\newcommand{\beqo}{\begin{equation*}}
\newcommand{\eeqo}{\end{equation*}}
\newcommand{\ben}{\begin{eqnarray}}
\newcommand{\een}{\end{eqnarray}}
\newcommand{\beno}{\begin{eqnarray*}}
\newcommand{\eeno}{\end{eqnarray*}}

\def\gpe{\nabla \psi ^\varepsilon}
\def\cle{c^\ve_1}
\def\cre{c^\ve_2}
\def\cei{c^\ve_i}
\def\ue{u^\ve}
\def\lc{c_1}
\def\rc{c_2}
\def\dl{D_1}
\def\dr{D_2}
\def\zl{z_1}
\def\zr{z_2}
\def\re{\rho^\ve}
\def\Gl{\Gamma_1}
\def\Gr{\Gamma_2}
\def\pw{\Phi_W}
\def\pe{\psi^\varepsilon}

\def\cirs{c_i^R}
\def\prs{\Phi^R}
\def\urs{u^R}
\def\rrs{\rho^R}

\def\cisum{c_i^S}
\def\psum{\Phi^S}
\def\usum{u^S}
\def\rsum{\rho^S}
\def\cik{c_i^{\scriptscriptstyle( k)}}
\def\ciik{c_{\scriptscriptstyle i,I}^{\scriptscriptstyle(k)}}
\def\cilbk{c_{\scriptscriptstyle i,LB}^{\scriptscriptstyle(k)}}
\def\cirbk{c_{\scriptscriptstyle i,RB}^{\scriptscriptstyle(k)}}
\def\cilmk{c_{\scriptscriptstyle i,LM}^{\scriptscriptstyle(k)}}
\def\cirmk{c_{\scriptscriptstyle i,RM}^{\scriptscriptstyle(k)}}
\def\cio{c_i^{\scriptscriptstyle(0)}}
\def\ciio{c_{\scriptscriptstyle i,I}^{\scriptscriptstyle(0)}}
\def\cilbo{c_{\scriptscriptstyle i,LB}^{\scriptscriptstyle(0)}}
\def\cirbo{c_{\scriptscriptstyle i,RB}^{\scriptscriptstyle(0)}}
\def\cilmo{c_{\scriptscriptstyle i,LM}^{\scriptscriptstyle(0)}}
\def\cirmo{c_{\scriptscriptstyle i,RM}^{\scriptscriptstyle(0)}}
\def\cil{c_i^{\scriptscriptstyle(1)}}
\def\ciil{c_{\scriptscriptstyle i,I}^{\scriptscriptstyle(1)}}
\def\cilbl{c_{\scriptscriptstyle i,LB}^{\scriptscriptstyle(1)}}
\def\cirbl{c_{\scriptscriptstyle i,RB}^{\scriptscriptstyle(1)}}
\def\cilml{c_{\scriptscriptstyle i,LM}^{\scriptscriptstyle(1)}}
\def\cirml{c_{\scriptscriptstyle i,RM}^{\scriptscriptstyle(1)}}
\def\cir{c_i^{\scriptscriptstyle(2)}}
\def\ciir{c_{\scriptscriptstyle i,I}^{\scriptscriptstyle(2)}}
\def\cilbr{c_{\scriptscriptstyle i,LB}^{\scriptscriptstyle(2)}}
\def\cirbr{c_{\scriptscriptstyle i,RB}^{\scriptscriptstyle(2)}}
\def\cilmr{c_{\scriptscriptstyle i,LM}^{\scriptscriptstyle(2)}}
\def\cirmr{c_{\scriptscriptstyle i,RM}^{\scriptscriptstyle(2)}}
\def\ciis{c_{\scriptscriptstyle i,I}^{\scriptscriptstyle(3)}}

\def\rio{\rho_{\scriptscriptstyle I}^{\scriptscriptstyle(0)}}
\def\rlbo{\rho_{\scriptscriptstyle LB}^{\scriptscriptstyle(0)}}

\def\rlmo{\rho_{\scriptscriptstyle LM}^{\scriptscriptstyle(0)}}

\def\rl{\rho^{\scriptscriptstyle(1)}}
\def\ril{\rho_{\scriptscriptstyle I}^{\scriptscriptstyle(1)}}
\def\rlbl{\rho_{\scriptscriptstyle LB}^{\scriptscriptstyle(1)}}

\def\rlml{\rho_{\scriptscriptstyle LM}^{\scriptscriptstyle(1)}}

\def\rr{\rho^{\scriptscriptstyle(2)}}
\def\rir{\rho_{\scriptscriptstyle I}^{\scriptscriptstyle(2)}}
\def\rlbr{\rho_{\scriptscriptstyle LB}^{\scriptscriptstyle(2)}}
\def\rrbr{\rho_{\scriptscriptstyle RB}^{\scriptscriptstyle(2)}}
\def\rlmr{\rho_{\scriptscriptstyle LM}^{\scriptscriptstyle(2)}}
\def\rrmr{\rho_{\scriptscriptstyle RM}^{\scriptscriptstyle(2)}}
\def\ris{\rho_{\scriptscriptstyle I}^{\scriptscriptstyle(3)}}

\def\uo{u^{\scriptscriptstyle(0)}}
\def\uio{u_{\scriptscriptstyle I}^{\scriptscriptstyle(0)}}

\def\ul{u^{\scriptscriptstyle(1)}}
\def\uil{u_{\scriptscriptstyle I}^{\scriptscriptstyle(1)}}

\def\ur{u^{\scriptscriptstyle(2)}}
\def\uir{u_{\scriptscriptstyle I}^{\scriptscriptstyle(2)}}

\def\uis{u_{\scriptscriptstyle I}^{\scriptscriptstyle(3)}}

\def\vk{v^{\scriptscriptstyle(k)}}
\def\vik{v_{\scriptscriptstyle I}^{\scriptscriptstyle(k)}}
\def\vlbk{v_{\scriptscriptstyle LB}^{\scriptscriptstyle(k)}}
\def\vrbk{v_{\scriptscriptstyle RB}^{\scriptscriptstyle(k)}}
\def\vlmk{v_{\scriptscriptstyle LM}^{\scriptscriptstyle(k)}}
\def\vrmk{v_{\scriptscriptstyle RM}^{\scriptscriptstyle(k)}}

\def\vio{v_{\scriptscriptstyle I}^{\scriptscriptstyle (0)}}
\def\vlbo{v_{\scriptscriptstyle LB}^{\scriptscriptstyle (0)}}
\def\vrbo{v_{\scriptscriptstyle RB}^{\scriptscriptstyle (0)}}
\def\vlmo{v_{\scriptscriptstyle LM}^{\scriptscriptstyle (0)}}
\def\vrmo{v_{\scriptscriptstyle RM}^{\scriptscriptstyle (0)}}

\def\vil{v_{\scriptscriptstyle I}^{\scriptscriptstyle (1)}}
\def\vlbl{v_{\scriptscriptstyle LB}^{\scriptscriptstyle (1)}}
\def\vrbl{v_{\scriptscriptstyle RB}^{\scriptscriptstyle (1)}}
\def\vlml{v_{\scriptscriptstyle LM}^{\scriptscriptstyle (1)}}
\def\vrml{v_{\scriptscriptstyle RM}^{\scriptscriptstyle(1)}}

\def\vir{v_{\scriptscriptstyle I}^{\scriptscriptstyle(2)}}
\def\vlbr{v_{\scriptscriptstyle LB}^{\scriptscriptstyle(2)}}
\def\vrbr{v_{\scriptscriptstyle RB}^{\scriptscriptstyle(2)}}
\def\vlmr{v_{\scriptscriptstyle LM}^{\scriptscriptstyle(2)}}
\def\vrmr{v_{\scriptscriptstyle RM}^{\scriptscriptstyle(2)}}

\def\wk{w^{\scriptscriptstyle (k)}}
\def\wik{w_{\scriptscriptstyle I}^{\scriptscriptstyle (k)}}
\def\wlbk{w_{\scriptscriptstyle LB}^{\scriptscriptstyle (k)}}
\def\wrbk{w_{\scriptscriptstyle RB}^{\scriptscriptstyle (k)}}
\def\wlmk{w_{\scriptscriptstyle LM}^{\scriptscriptstyle (k)}}
\def\wrmk{w_{\scriptscriptstyle RM}^{\scriptscriptstyle (k)}}

\def\wio{w_{\scriptscriptstyle I}^{\scriptscriptstyle (0)}}
\def\wlbo{w_{\scriptscriptstyle LB}^{\scriptscriptstyle (0)}}
\def\wrbo{w_{\scriptscriptstyle RB}^{\scriptscriptstyle (0)}}
\def\wlmo{w_{\scriptscriptstyle LM}^{\scriptscriptstyle (0)}}
\def\wrmo{w_{\scriptscriptstyle RM}^{\scriptscriptstyle (0)}}

\def\wil{w_{\scriptscriptstyle I}^{\scriptscriptstyle(1)}}
\def\wlbl{w_{\scriptscriptstyle LB}^{\scriptscriptstyle(1)}}
\def\wrbl{w_{\scriptscriptstyle RB}^{\scriptscriptstyle(1)}}
\def\wlml{w_{\scriptscriptstyle LM}^{\scriptscriptstyle(1)}}
\def\wrml{w_{\scriptscriptstyle RM}^{\scriptscriptstyle(1)}}

\def\wir{w_{\scriptscriptstyle I}^{\scriptscriptstyle(2)}}
\def\wlbr{w_{\scriptscriptstyle LB}^{\scriptscriptstyle(2)}}
\def\wrbr{w_{\scriptscriptstyle RB}^{\scriptscriptstyle(2)}}
\def\wlmr{w_{\scriptscriptstyle LM}^{\scriptscriptstyle(2)}}
\def\wrmr{w_{\scriptscriptstyle RM}^{\scriptscriptstyle(2)}}

\def\pk{\Phi^{\scriptscriptstyle (k)}}
\def\pik{\Phi_{\scriptscriptstyle I}^{\scriptscriptstyle(k)}}
\def\plbk{\Phi_{\scriptscriptstyle LB}^{\scriptscriptstyle(k)}}
\def\prbk{\Phi_{\scriptscriptstyle RB}^{\scriptscriptstyle (k)}}
\def\plmk{\Phi_{\scriptscriptstyle LM}^{\scriptscriptstyle (k)}}
\def\prmk{\Phi_{\scriptscriptstyle RM}^{\scriptscriptstyle (k)}}
\def\po{\Phi^{\scriptscriptstyle(0)}}
\def\pio{\Phi_{\scriptscriptstyle I}^{\scriptscriptstyle (0)}}
\def\plbo{\Phi_{\scriptscriptstyle LB}^{\scriptscriptstyle(0)}}
\def\prbo{\Phi_{\scriptscriptstyle RB}^{\scriptscriptstyle(0)}}
\def\plmo{\Phi_{\scriptscriptstyle LM}^{\scriptscriptstyle(0)}}
\def\prmo{\Phi_{\scriptscriptstyle RM}^{\scriptscriptstyle(0)}}
\def\pl{\Phi^{\scriptscriptstyle(1)}}
\def\pil{\Phi_{\scriptscriptstyle I}^{\scriptscriptstyle(1)}}
\def\plbl{\Phi_{\scriptscriptstyle LB}^{\scriptscriptstyle(1)}}
\def\prbl{\Phi_{\scriptscriptstyle RB}^{\scriptscriptstyle(1)}}
\def\plml{\Phi_{\scriptscriptstyle LM}^{\scriptscriptstyle(1)}}
\def\prml{\Phi_{\scriptscriptstyle RM}^{\scriptscriptstyle(1)}}
\def\pr{\Phi^{\scriptscriptstyle(2)}}
\def\pir{\Phi_{\scriptscriptstyle I}^{\scriptscriptstyle(2)}}
\def\plbr{\Phi_{\scriptscriptstyle LB}^{\scriptscriptstyle(2)}}
\def\prbr{\Phi_{\scriptscriptstyle RB}^{\scriptscriptstyle(2)}}
\def\plmr{\Phi_{\scriptscriptstyle LM}^{\scriptscriptstyle(2)}}
\def\prmr{\Phi_{\scriptscriptstyle RM}^{\scriptscriptstyle(2)}}
\def\prek{p^{\scriptscriptstyle(k)}}
\def\preik{p_{\scriptscriptstyle I}^{\scriptscriptstyle(k)}}
\def\prelbk{p_{\scriptscriptstyle LB}^{\scriptscriptstyle(k)}}
\def\prerbk{p_{\scriptscriptstyle RB}^{\scriptscriptstyle(k)}}
\def\prelmk{p_{\scriptscriptstyle LM}^{\scriptscriptstyle(k)}}
\def\prermk{p_{\scriptscriptstyle RM}^{\scriptscriptstyle(k)}}
\def\prefr{p^{\scriptscriptstyle(-2)}}
\def\preifr{p_{\scriptscriptstyle I}^{\scriptscriptstyle(-2)}}
\def\prelbfr{p_{\scriptscriptstyle LB}^{\scriptscriptstyle(-2)}}

\def\prelmfr{p_{\scriptscriptstyle LM}^{\scriptscriptstyle(-2)}}

\def\prefl{p^{\scriptscriptstyle(-1)}}
\def\preifl{p_{\scriptscriptstyle I}^{\scriptscriptstyle(-1)}}
\def\prelbfl{p_{\scriptscriptstyle LB}^{\scriptscriptstyle(-1)}}
\def\prerbfl{p_{\scriptscriptstyle RB}^{\scriptscriptstyle(-1)}}
\def\prelmfl{p_{\scriptscriptstyle LM}^{\scriptscriptstyle(-1)}}
\def\prermfl{p_{\scriptscriptstyle RM}^{\scriptscriptstyle(-1)}}
\def\preo{p^{\scriptscriptstyle (0)}}
\def\preio{p_{\scriptscriptstyle I}^{\scriptscriptstyle(0)}}
\def\prelbo{p_{\scriptscriptstyle LB}^{\scriptscriptstyle(0)}}
\def\prerbo{p_{\scriptscriptstyle RB}^{\scriptscriptstyle(0)}}
\def\prelmo{p_{\scriptscriptstyle LM}^{\scriptscriptstyle(0)}}
\def\prermo{p_{\scriptscriptstyle RM}^{\scriptscriptstyle(0)}}
\def\prel{p^{\scriptscriptstyle(1)}}
\def\preil{p_{\scriptscriptstyle I}^{\scriptscriptstyle(1)}}
\def\prelbl{p_{\scriptscriptstyle LB}^{\scriptscriptstyle(1)}}
\def\prerbl{p_{\scriptscriptstyle RB}^{\scriptscriptstyle(1)}}
\def\prelml{p_{\scriptscriptstyle LM}^{\scriptscriptstyle(1)}}
\def\prerml{p_{\scriptscriptstyle RM}^{\scriptscriptstyle(1)}}
\def\prer{p^{\scriptscriptstyle(2)}}

\newtheorem{thm}{Theorem}[section]
\newtheorem{lem}{Lemma}[section]
\newtheorem{rmk}{Remark}[section]
\newtheorem{cor}{Corollary}[section]
\newtheorem{prop}{Proposition}[section]

\renewcommand{\theequation}{\thesection.\arabic{equation}}

\newtheorem{Theorem}{Theorem}[section]

\newtheorem{Proposition}[Theorem]{Proposition}

\begin{document}
\title[Quasi-neutral limit of NPNS system]
{Quasi-neutral limit of Nernst-Planck-Navier-Stokes  system }

\author[P. Zhang]{Ping Zhang}%
\address[P. Zhang]
 {Academy of Mathematics $\&$ Systems Science
and  Hua Loo-Keng Key Laboratory of Mathematics, The Chinese Academy of
Sciences, Beijing 100190, CHINA, and School of Mathematical Sciences, University of Chinese Academy of Sciences, Beijing 100049, CHINA. }
\email{zp@amss.ac.cn}

\author[Y. Zhang]{Yibin Zhang}
\address[Y. Zhang]
 {Academy of Mathematics $\&$ Systems Science, The Chinese Academy of
Sciences, Beijing 100190, CHINA.  } \email{zhangyibin22@mails.ucas.ac.cn}

\date{\today}

\begin{abstract}
In this paper, we investigate the quasi-neutral limit of Nernst-Planck-Navier-Stokes  system
in a smooth bounded domain $\Omega$ of $\R^d$ for $d=2,3,$ with
``electroneutral boundary conditions" and well-prepared data. We first prove by using modulated energy estimate
that the solution sequence converges to the limit system in the norm of $L^\infty((0,T);L^2(\Omega))$ for some positive time $T.$
In order to justify the limit in a stronger norm, we need to construct both the  initial layers and weak boundary layers in the approximate
solutions.
\end{abstract}

\maketitle

{\bf \normalsize Keywords:}  Nernst-Planck-Navier-Stokes  system;\, quasi-neutral limit;\, modulated energy estimate.

\vskip 0.2cm
\noindent {\sl AMS Subject Classification (2000):}
{ 82C21; 82D15; 35Q92}

\renewcommand{\theequation}{\thesection.\arabic{equation}}
\setcounter{equation}{0}

\section{Introduction}\label{section1}
The Nernst-Planck-Navier-Stokes  system,  in short NPNS system,  describes the time evolution of ionic concentration in solvents,
which are transported by viscous incompressible fluid and diffuse under an electric potential and their own concentration gradients.
 The main purpose of this paper is to study the quasi-electroneutral limit ($\ve \rightarrow $0) of the following binary  NPNS system on a smooth bounded domain $\Omega \subset  \mathbb{R}^d$ for $d = 2,3,$
\begin{equation}\label{NPNS}
 \quad\left\{\begin{array}{l}
\pt c_1 + u\cdot \grad c_1 = D_1 \dive (\grad c_1 + z_1 c_1 \grad \bp),\\
\pt c_2 + u\cdot \grad c_2 = D_2 \dive (\grad c_2 + z_2 c_2 \grad \bp),\\
- \ve^2 \lap \bp = \rho = z_1 c_1+z_2 c_2,\\
\pt u+ u\cdot \grad u -\nu \lap u=-\grad p - K \rho \grad \bp, \\
\dive u=0, \\
c_1|_{t=0}=c_1(0),\quad c_2|_{t=0}=c_2(0),\quad u|_{t=0}=u_0.
\end{array}\right.
\end{equation}
The function $c_i=c_i(x,t)$ represents the concentration of the $i$-th species ($i=1,2$),
and $\bp$ is the electrical potential generated by the charge density $\rho.$
The divergence-free vector field $u$ is the velocity of the fluid, and $p$ is a scalar  pressure function.
$\zl>0>\zr $, $D_i >0$ are  constants (which may differ from each other), and designate the valency and diffusivity of the $i$-th species, respectively. In what follows, we denote $D^* \eqdefa \min\{D_1,D_2\}$, and without loss of generality,  we assume that $D_1 \geq D_2$.
The positive constant $\ve $ is a rescaled dielectric permittivity of the solvent and is proportional to  the  Debye length. The kinematic viscosity of the fluid  $\nu >0 $, and $K>0 $ is a constant related to the Boltzmann's constant and the absolute temperature (for simplicity,
 here we take $K=1$). One may check \cite{Ru90} for an introduction of the
basic physical and mathematical issues about the system \eqref{NPNS}.

We consider the electroneutrality (or the vanishing of electronic charge) of  NPNS system with ``well-prepared" initial data, i.e., $\rho(0)=0$. In addition, we consider the following ``electroneutral boundary conditions" for $c_i$
\begin{equation}\label{EN}
    (\mathbf{EN})\quad c_1|_{\partial \Omega}=\ga_1(x)>0, \quad c_2|_{\partial \Omega}=\ga_2(x)>0, \quad z_1 \ga_1+z_2 \ga_2=0.
\end{equation}
We also implement the system \eqref{NPNS} with  Dirichlet boundary conditions for $u$ and $\Phi:$
\begin{equation}
    u |_{\p\Omega}=0 \andf \bp |_{\p \Omega}=W(x).
    \nonumber
\end{equation}
Since here we consider  the quasineutral limit of the system \eqref{NPNS} in finite time, the boundary data $\ga_i(x)$ and $W(x)$
are allowed to depend on time. For simplicity, we just take the time independent data.

We denote by $\Ga_i$ and $\pw$ to be determined  respectively by
\begin{equation}\label{1.3}
    \quad\left\{ \begin{array}{l}
         \lap \Gamma_i =0 ,\quad x\in \Omega,\\
         \Gamma_i |_{\p \Omega}=\gamma_i(x),
    \end{array} \right.
\andf \
    \left\{ \begin{array}{l}
         \lap \pw =0 ,\quad x\in \Omega,\\
         \pw |_{\p \Omega}=W(x).
    \end{array} \right.
\end{equation}
Then due to the uniqueness of Poisson equation and (\ref{EN}), we have $\zl \Gl +\zr \Gr =0$.

In the case $d=2$, the global well-posedness of  binary NPNS system has been established in \cite{CIL NPNS far from eq}. In the case  $d=3$,
the problem of global well-posedness of binary NPNS system has only partial results, especially due to the part of Navier-Stokes equations.
In fact, even the global existence of pure 3D Nernst-Planck system or the system coupled to Stokes flow is in general open. One may check \cite{FS17, JS09, LW20, S09}
for the global existence of weak solutions of the 3-D system \eqref{NPNS}.
Thus, we can only expect that the  lifespan to the  strong solution of 3-D NPNS system has a positive lower bound, which is independent of the Debye length $\ve$, and which will be established in Section \ref{section2}.

 Constantin and Ignatova etal \cite{CI On the npns system, CIL NPNS far from eq} first proved  the global existence and stability
 of solutions to 2-D NPNS system under the blocking boundary conditions  or uniform selective boundary conditions (we also mention that the same result was obtained in \cite{BFA14} with blocking boundary conditions for the ions and a Robin boundary condition for the electric potential). They further proved in \cite{CIL Interior electroneutrlity} that
 $$
 \lim_{\ve\to 0}\lim_{t\to\infty}\sup_{x\in K}|\rho(t,x)|=0,
 $$
 for any fixed initial conditions and any compact subset $K$ of $\Omega.$ They also established the same result in 3D with the same boundary conditions and small perturbations of steady states data. In \cite{WJ21}, the authors investigated the initial layer for the 3-D system \eqref{NPNS} in $\T^3$ in the quasi-neutral regime.

Formally, by setting $\ve=0$ in the Poisson equation of \eqref{NPNS}, one has $\rho = 0 $ and the Nernst-Planck equations and Navier-Stokes equations will be decoupled. To make this scenario mathematically rigorous, we introduce $\psi^\ve \eqdefa \bp^\ve - \pw $, and  rewrite the NPNS system as
\begin{subequations} \label{S1eq1}
\begin{gather}
\label{e1}
    \partial_t \cle + \ue \cdot\grad \cle =\dl \dive (\grad \cle + \zl \cle \grad \pe + \zl \cle \grad \pw),\\
\label{e2}
    \partial_t \cre + \ue \cdot\grad \cre =\dl \dive (\grad \cre + \zr \cre \gpe+ \zr \cre \grad \pw),\\
\label{e3}
    -\ve^2 \lap \psi^\ve =\re=\zl \cle +\zr\cre,\\
\label{e4}
     \p_t  \ue + \ue \cdot \grad \ue -\nu \lap \ue =-\grad p^\ve - \re\gpe - \re \grad \pw \\
    \qquad\quad = -\grad p^\ve +\ve ^2 \dive (\grad \psi^\ve \otimes \grad \psi^\ve)-\frac{\ve ^2}{2}\grad \lvert\grad \psi^\ve \rvert^2 -\re \grad \pw, \nonumber\\
\label{e5}
    \dive \ue=0.
    \end{gather}
\end{subequations}
Here we have used the relation
\begin{equation}
    \lap \pe \grad \pe =\dive (\grad \psi^\ve \otimes \grad \psi^\ve)-\frac{1}{2}\grad \lvert\grad \psi^\ve \rvert^2.
    \nonumber
\end{equation}
And we implement the following boundary conditions for the system \eqref{S1eq1}:
\begin{equation}\label{ebdry}
   \cle|_{\p \Omega}=\gamma_i(x), \quad \ue|_{\p \Omega}=0, \quad \pe |_{\p \Omega}=0,\quad z_1 \ga_1+z_2 \ga_2=0.
\end{equation}

By using modulated energy estimate, we shall prove in   Section \ref{section3} that as $\ve \rightarrow 0, $
the solutions of the   system \eqref{S1eq1} with ``well-prepared" initial data converge to the following decoupled system:
\begin{subequations} \label{S1eq2}
\begin{gather}
\label{limit1}
    \partial_t \lc + u \cdot\grad \lc =\dl \dive (\grad \lc + \zl \lc \grad \psi + \zl \lc \grad \pw),\\
\label{limit2}
    \partial_t \rc + u \cdot\grad \rc =\dr \dive (\grad \rc + \zr \rc \grad \psi + z_2 \rc \grad \pw),\\
\label{limit3}
    \rho =\zl \lc+\zr\rc=0,\\
\label{limit4}
    \pt u+ u\cdot \grad u -\nu \lap u=-\grad p,\\
\label{limit5}
    \dive u=0,\\
  \label{limit-bdry}
    c_i|_{\p \Omega}=\gamma_i(x),\quad \psi|_{\p \Omega}=0,\quad u|_{\p \Omega} =0 ,\quad z_1 \ga_1+z_2 \ga_2=0.
    \end{gather}
\end{subequations}

By combining equations (\ref{limit1}), (\ref{limit2}), (\ref{limit3}) and boundary conditions(\ref{limit-bdry}), $\psi$ is the unique solution of the following non-degenerate second order elliptic equation
\begin{equation}\label{psi}
\quad \left\{  \begin{array}{l}
 \dive\big( (\zl \dl \grad \lc +\zr\dr \grad \rc)+(\zl^2\dl \lc +\zr^2\dr\rc)(\grad \psi +\grad \pw)\big)=0,\\
\psi |_{\p \Omega}=0.
\end{array}\right.
\end{equation}

In view of (\ref{limit3}), we can equivalently rewrite (\ref{psi}) as
\begin{equation}\label{equivalent psi}
\quad \left\{  \begin{array}{l}
 \dive\big( (\dl-\dr) \grad \lc +(\zl\dl -\zr\dr) \lc (\grad \psi +\grad \pw) \big)=0, \\
\psi |_{\p \Omega}=0.
\end{array}\right.
\end{equation}

\begin{rmk} In Section \ref{section2}, we shall prove by
 maximum principle that $\zl^2\dl \lc +\zr^2\dr\rc$ has a  positive lower bound as long as $\zl^2\dl \lc(0) +\zr^2\dr\rc(0)$ is strictly positive.
So that the equation (\ref{psi}) is strictly elliptic.

Furthermore, by plugging equations (\ref{limit3}) and (\ref{psi}) into equations (\ref{limit1}) and (\ref{limit2}),
we obtain the following   equation for $c_i$ ($i=1,2$):
\begin{equation}\label{ trans + diff}
    \p_t c_i +u\cdot \grad c_i = \frac{(\zl-\zr)\dl\dr}{\zl\dl-\zr\dr}\lap c_i.
\end{equation}
\end{rmk}

In what follows, we always denote
\begin{subequations} \label{S1eq3}
\begin{gather}
    \lam^\ve_i \eqdefa \min\bigl\{ \inf_{\p \Omega} \ga_i, \inf_\Omega c^\ve_i(0)\bigr\},\qquad \Lam^\ve_i \eqdefa \max\bigl\{\sup_{\p \Omega} \ga_i, \sup_\Omega c^\ve_i(0)\bigr\},\label{S1eq3a}\\
\lam_i \eqdefa \min\bigl\{ \inf_{\p \Omega} \ga_i, \inf_\Omega c_i(0) \bigr\}  ,\qquad \Lam_i \eqdefa \max\bigl\{\sup_{\p \Omega} \ga_i, \sup_\Omega c_i(0)\bigr\}, \label{S1eq3b}
 \end{gather}
\end{subequations}
and we denote by  $A$ to be the Stokes operator with domain $D(A) \eqdefa H^2(\Omega)\cap H^1_{0,\sigma}(\Omega)$, and denote by $V$ to be the space $D(A^\frac{1}{2})$.
We always assume that there exist $0<\lam \leq \Lam < \oo$ such that
\begin{equation}\label{uniform bound for ci}
    0<\lam \leq \min \{ \lam^\ve_i, \lam_i \} \leq \max \{ \Lam^\ve_i, \Lam_i \} \leq \Lam <\oo.
\end{equation}

Our first main result of this paper states as follows:

\begin{thm}\label{thm1}
{\sl
Let $d=2$ or $3$,  and the initial data be ``well-prepared" i.e. $\re(0)=0.$ We assume $(\cei(0),\ue(0), c_i(0),u(0)) \in H^5$.
If there exists a positive constant $C>0 $ so that
\begin{subequations}
    \begin{gather}
         \|\cei(0) - c_i(0)\|_{L^2} + \|\ue(0) - u_0\|_ {L^2}  \leq C\ve, \label{assumption 1}\\
          \|\ue(0)\|_{V} \leq C \label{assumtion 2},
    \end{gather}
\end{subequations}
then there exist positive constants $M, T>0,$ which  depend  only on initial data, $\nu,\lam, \Lam, W(x)$, $ \ga_i(x), z_i$ and $D_i$ for $i=1,2,$  so that
\begin{subequations} \label{S1eq4}
\begin{gather}
 \|\cei -c_i\|_{L^\oo_T(L^2)}+\|\ue - u\|_{L^\oo_T(L^2)} +\ve \|\grad \pe\|_{L^\oo_T(L^2)}\leq M \ve,
   \label{result 1 in thm 1}\\
    \|\grad \cei - \grad c_i\|_{L^2_T(L^2)} + \|\grad \ue - \grad u\|_{L^2_T(L^2)}
    + \|\gpe - \grad \psi\|_{L^2_T(L^2)} \leq M \ve,  \label{result 2 in thm 1} \\
    \ve^{-1}\|\re \|_{L^2_T(L^2)} = \ve \|\lap\pe\|_{L^2_T(L^2)} \leq M \ve. \label{result 3 in thm 1}
    \end{gather}
\end{subequations}
}
\end{thm}

\begin{rmk}\begin{enumerate}
  \item[(1)]  Due to the ``well-prepared" assumption, $\pe(0)$ equals to zero and doesn't appear in our assumption (\ref{assumption 1}).

  \item[(2)]   (\ref{assumtion 2}) is not required in the case d=2.

   \item[(3)]  The main idea used to prove Theorem \ref{thm1} is to use modulated energy estimate,
  which was first introduced by Brenier \cite{B00}  to study the quasi-neutral limit of
  Vlasov-Poisson system to the incompressible Euler equations, which motivates the first author of this
  paper to  investigate the semi-classical limit of Schr\"odinger-Poisson equations through Wigner transform (\cite{Z1}).
  The ``modulated energy functional" in \cite{Z1} was simplified by the authors in \cite{lz3} in order to deal with
  the semi-classical limit of cubic Sch\"odinger equation in the exterior domain.
    \end{enumerate}
\end{rmk}

In order to consider the convergence in \eqref{S1eq4} in stronger norm, we need to investigate both the initial layers
and boundary layers. For simplicity, we just take $\Omega=\T^{d-1}\times(0,1)$ for $d=2,3.$

\begin{thm}\label{main thm}
    {\sl
Let $d=2$ or $3$,  and the initial data be ``well-prepared" i.e. $\re(0)=0.$ We assume $(\cei(0),\ue(0), c_i(0),u(0)) \in H^5$.
If there exists a positive constant $C>0 $ so that
\begin{subequations}
    \begin{gather}
        \|\cei(0) - c_i(0)\|_{L^2}  + \|\ue(0) - u_0\|_ {L^2} \leq C\ve^3 ,\label{main thm,assumption1}\\
        \|\grad\cei(0) - \grad c_i(0)\|_{L^2} + \|\ue(0)- u_0\|_V   \leq C\ve^\frac{3}{2},\label{main thm,assumption2} \\
        \|\lap\cei(0)-\lap c_i(0)\|_{L^2} + \|A\ue(0)-Au_0\|_{L^2} \leq \ve^\frac{1}{2},\label{main thm,assumption3} \\
         \psi(0) = 0, \label{main thm,assumption4}
    \end{gather}
\end{subequations}
then there exist positive constants $\ve_0, M, T>0,$ which  depend only on initial data, $\nu, \lam,\Lam,$ $W(x), \ga_i(x), z_i$ and $D_i$ for $i=1,2,$  so that for $\ve\leq \ve_0$,
\begin{equation*}
     \|\cei- c_i\|_{L^\oo_T(H^2(\T^{d-1}\times(0,1)))} +  \|\re\|_{L^\oo_T(H^2(\T^{d-1}\times(0,1)))} \leq M\ve^\frac{1}{2}.
\end{equation*}
    }
\end{thm}

\begin{rmk} In fact, we shall present more precise expansions for $\cei$ and $\re$ in {Theorem \ref{thm4}} below.
    One may check \eqref{S5eq1} and \eqref{S5eq2} for details.
\end{rmk}

We end this section with the structure of this paper.

In Section \ref{section2}, we shall prove that the $\ve$-dependent NPNS system has a strong solution on some $\ve$-independent time interval.

In Section \ref{section3}, we present the proof of Theorem \ref{thm1}.

In Section \ref{analysis of layers}, we investigate both the initial and boundary layers of the $\ve$-dependent NPNS system.

Finally in Section \ref{convergence of Loo}, we present the proof of Theorem \ref{main thm}. In fact, we shall present
more detailed approximation, see Theorems \ref{thm4} and \ref{thm3}.

\renewcommand{\theequation}{\thesection.\arabic{equation}}
\setcounter{equation}{0}
\section{Preliminaries and technical lemmas}\label{section2}

In this section, we present some relevant results and some basic lemmas.  We first define the electrochemical potentials as follows:
\begin{equation}\label{electrochemical potential}
    \mu^\ve_i\eqdefa \log c^\ve_i +z_i(\pe + \pw), \quad \mu^*_i\eqdefa\log \Gamma_i + z_i\pw \andf \mu_i\eqdefa\log c_i +z_i(\psi+\pw).
\end{equation}
Then  we may rewrite the mass conservation equations (\ref{e1}-\ref{e2}) and (\ref{limit1}-\ref{limit2}) of $c^\ve_i$ and $c_i$ as
\begin{equation}\label{mass conservtion of divergence form}
    \pt c^\ve_i= \dive (D_ic^\ve_i \grad \mu^\ve_i-\ue c^\ve_i),
\end{equation}
and
\begin{equation}\label{ci mass conservtion of divergence form}
    \pt c_i= \dive (D_ic_i \grad \mu_i-u c_i).
\end{equation}

In the case of $\zl=-\zr,$ the positive lower bound of $c^\ve_i$ for the system \eqref{NPNS}
with Dirichlet boundary condition can be obtained by using maximum principle (see\cite{CIL Interior electroneutrlity,CIL Physica D}).
 Along the same lines to \cite{CIL Interior electroneutrlity},  we shall prove similar result for  the system \eqref{NPNS} with
  ``electro-neutral boundary conditions" \eqref{EN} and ``well-prepared" initial data. However, $\zl$ and $-\zr$ are relaxed to arbitrary positive constants.

\begin{lem}\label{maximal principle}
 {\sl   Let $(\cle,\cre,\pe,\ue)$ be smooth enough solution of the system (\ref{e1}-\ref{e5}) with boundary conditions (\ref{ebdry}) on $[0,T].$
 We assume that the initial data is ``well-prepared" (i.e. $\re(0)=0$). Then $c^\ve_i (i=1, 2)$ verifies
    \begin{equation*}
       0<\lam^\ve_i  \leq c^\ve_i(t,x) \leq \Lam^\ve_i
    \end{equation*} for $\lam^\ve_i$ and $\Lam^\ve_i$ being defined by \eqref{S1eq3a}.}
\end{lem}
\begin{proof}
Notice that $z_1>0$ and $z_2<0,$ we deduce from \eqref{EN} and $\rho^\ve(0)=0$ that
\begin{equation*} \label{S2eq1}
\zl \lam_1^\ve+\zr \lam_2^\ve=0 \andf \zl \Lam_1^\ve+\zr \Lam_2^\ve=0.
\end{equation*}

In what follows, we just prove $\lc^\ve(t,x) \leq \Lam_1^\ve$.
 Otherwise,  we denote $m_i^\ve(t)\eqdefa\sup_{x\in\Omega} c_i^\ve(x,t)$ and fix a positive constant $\kappa > \Lam_1^\ve,$ then
 there exists some $t\in (0,T)$ so that
$m_1^\ve(t) \geq \kappa$. We denote $t_0>0$ to be the first time when $m_1^\ve(t_0)=\kappa$ is attained.
Since $\kappa > \Lam_1^\ve \geq \sup_{\p \Omega} \ga_1^\ve$, there exists an interior point $x_0 \in \Omega$ so that
$c_1^\ve(x_0,t_0)=\kappa$. Hence using the the equation (\ref{e1}), (\ref{e3}) and maximality at point $(x_0, t_0)$, we find
    \begin{equation*}
        0\leq \pt c_1^\ve(x_0,t_0) \leq -\frac{\zl \dl}{\ve^2}c_1^\ve(x_0,t_0)\big( \zl \lc^\ve(x_0,t_0)+\zr \rc^\ve(x_0,t_0)\big),
    \end{equation*}
where we used the facts that the gradient of $c_1^\ve$ vanishes at $(x_0,t_0),$
 and the Laplacian is non-positive at an interior maximal point. As a result, it comes out
    \begin{equation*}
    \zl c_1^\ve(x_0,t_0)+\zr \rc^\ve(x_0,t_0) \leq 0,
    \end{equation*}
    which implies
    \begin{equation}\label{S2eq2}
        \rc^\ve(x_0,t_0) \geq -\frac{\zl}{\zr}\lc^\ve(x_0,t_0) \geq -\frac{\zl}{\zr} \kappa > -\frac{\zl}{\zr}\Lam_1^\ve = \Lam_2^\ve.
    \end{equation}
    So that
    we can choose the first time $0<t_1 \leq t_0$ when $m_2^\ve(t_1)=-\frac{\zl}{\zr} \kappa.$  We claim that $t_1=t_0$.
    Otherwise by repeating the proof of \eqref{S2eq2}, we can choose the first time  $0<t_2 \leq t_1< t_0$ when $m_1^\ve(t_2)=\kappa,$ which contradicts with the definition of $t_0$.

    Since $m_i^\ve(t_0)>\Lam_i^\ve$ ($i=1,2$), by continuity, there exists $\delta>0$ such that for all $s\in [t_0-\de ,t_0]$, there exist interior points $x_i(s) \in \Omega$ so that $c_i^\ve(x_i(s),s)=m_i^\ve(s)>\Lam_i^\ve$. Then for any $t_0-\delta < r< s \leq t_0$, we have
    \beq\label{m1}
    \begin{split}
        \limsup_{r \rto s^-} \frac{m_1^\ve(s)-m_1^\ve(r)}{s-r} \leq \limsup_{r \rto s^-} \frac{c_1^\ve(x_1(s),s)-c_1^\ve(x_1(s),r)}{s-r} =\pt c_1^\ve(x_1(s),s) \\
        \leq -\frac{\zl \dl}{\ve^2} m_1^\ve(s)(\zl m_1^\ve(s)+\zr m_2^\ve(s)) \leq -\frac{\zl^2 \dl }{2\ve ^2} (m^\ve_1)^2(s)
        + \frac{\zr^2 \dl }{2\ve ^2} (m^\ve_2)^2(s).
       \end{split}\eeq
       Along the same line, we obtain
       \begin{equation}\label{m2}
           \limsup_{r \rto s^-} \frac{m_2^\ve(s)-m^\ve_2(r)}{s-r} \leq \frac{\zl^2 \dr }{2\ve ^2} (m^\ve_1)^2(s)
           -\frac{\zr^2 \dr }{2\ve^2 } (m^\ve_2)^2(s).
       \end{equation}

    By   multiplying (\ref{m1}) by $\dr$ and (\ref{m2}) by $\dl,$ and then summing up the resulting inequalities, we achieve
       \begin{equation*}
           \limsup_{r \rto s^-} \Bigl( \dr \frac{m_1^\ve(s)-m_1^\ve(r)}{s-r} +\dl \frac{m_2^\ve(s)-m_2^\ve(r)}{s-r} \Bigr) \leq 0.
       \end{equation*}
 Then we deduce from Lemma \ref{f nonincreasing} below that $ \dr m_1^\ve(s) +\dl {m_2^\ve(s)} $ is a non-increasing function on $[t_0-\delta,t_0].$
     As a consequence, we obtain
       \begin{equation*}
           \dr m_1^\ve(t_0-\de)+\dl m_2^\ve(t_0-\de) \geq \dr m_1^\ve(t_0)+\dl m_2^\ve(t_0) = \dr \kappa -\frac{\zl}{\zr}\dl \kappa.
       \end{equation*}
       Hence we have $m_1^\ve(t_0-\de) \geq \kappa$, or $m_2^\ve(t_0-\de) \geq -\frac{\zl}{\zr} \kappa,$ either of which contradicts with our choice of $t_0$. This completes the proof of Lemma \ref{maximal principle}.
\end{proof}

\begin{lem}\label{f nonincreasing}
 {\sl   Let $f: (a,b] \rightarrow \mathbb{R}$ satisfies
    \begin{equation}\label{limsup condition}
        \limsup_{r \rto s^-} \frac{f(s)-f(r)}{s-r} \leq 0\quad \mbox{for any}\ \ s \in (a,b].
    \end{equation}
    Then f is non-increasing on $(a,b]$.}
\end{lem}
\begin{proof}
    Without loss of generality, it suffices to fix some $c\in (a,b)$ and prove $f(c) \geq f(b)$. In view of
    (\ref{limsup condition}), for any $\epsilon >0$ and $ s\in (a,b]$, there exists $\delta_s>0 $ such that for any $r \in (s-\delta_s,s)$, we have  $\frac{f(s)-f(r)}{s-r} \leq \epsilon $. Thus $\{(s-\delta_s,s) | s \in (a,b)\}$ constitutes an open covering of closed interval $[c,b-\delta_b]$, and by finite covering theorem, there exists a set of finite open intervals $\{(s_k-\delta_k,s_k)\}_{k=1}^n$ such that
    $$s_1-\delta_1 < c < s_2- \delta_2 < s_1 <s_3-\delta_3 <s_2 < \dots s_{n}-\delta_n <s_{n-1}< b-\delta_b< s_n < b, $$
    and
    \begin{equation*}
        \quad \left\{\begin{array}{ll}
            f(s_1)-f(c) \leq \epsilon (s_1 -c ), \quad &  k=1,  \\
             f(s_k)-f(s_{k-1}) \leq \epsilon (s_k - s_{k-1}),\quad  &2\leq k \leq n, \\
             f(b)-f(s_{n}) \leq \epsilon (b-s_{n}). \quad &
        \end{array} \right.
    \end{equation*}
    By summing up the above inequalities for $k$ from $1$ to $n,$ we achieve  $f(b)-f(c) \leq \epsilon(b-c)$, and then letting $\epsilon \rightarrow 0 $ leads to $f(b)\leq f(c).$ This completes the proof of the lemma.
\end{proof}

\begin{rmk}
  Since $c_i$ satisfies a transport-diffusion equation (\ref{ trans + diff}), it's easy to deduce the same result for $c_i$, i.e. $\lam_i \leq  c_i \leq \Lam_i$ with $\lam_i$ and $\Lam_i$ being defined by \eqref{S1eq3b}.
\end{rmk}

In the subsequent lemma, we shall use the dissipative structure of the system \eqref{S1eq1} to derive the estimate for the energy functional $E^\ve(t).$

\begin{lem}\label{lem:energy}
{\sl Let $(\cle,\cre,\ue)$ be a smooth enough solution of the system \eqref{S1eq1} supplemented with the boundary conditions (\ref{ebdry}) on interval $[0,T]$, we define energy functional $E^\ve(t)$ via
\begin{equation}
E^\ve(t) \eqdefa \sum^2_{i=1} \int_\Omega \Gamma_i \varphi\bigl(\frac{c^\ve_i}{\Gamma_i}\bigr)\,dx +\frac{\ve^2}{2} \| \grad \pe(t) \|^2_{L^2} +\frac{1}{2}\| \ue(t)\|^2_{L^2},
\label{S2eq3}
\end{equation}
where $\vphi(s)=s\log s -s+1 \geq 0 $.
Then there hold
\begin{enumerate}
\item
\beq\label{dissipation equality}
    \begin{split}
  &\frac{d}{dt}E^\ve(t) +\nu \| \grad \ue \|^2_{L^2} + \sum^2_{i=1}D_i \int_{\Omega} c^\ve_i | \grad \mu^\ve_i |^2\,dx  \\
  &= -\sum^2_{i=1} \int_\Omega \ue c^\ve_i \grad \log \Gamma_i\,dx -\int_\Omega \ue \re \grad \pw \,dx + \sum^2_{i=1}D_i \int_\Omega c^\ve_i \grad \mu^\ve_i \grad \mu^*_i  \,dx.
    \end{split}\eeq
\item There exists a positive constant $M$ depending only on $\Lam,W(x),\ga_i(x),z_i$ and $D_i$ for $i=1,2$, such that for any $t\in [0,T],$
    \beq
    \label{dissipation inequality}
        E^\ve(t)+\frac{\nu}{2} \| \grad \ue \|^2_{L^2_t(L^2)} + \frac{1}{2} \sum^2_{i=1}D_i \int^t_0 \int_{\Omega} c^\ve_i | \grad \mu^\ve_i |^2 dxdt  \leq (E^\ve(0)+Mt)e^t.
    \eeq
\item  There exists a positive constant $M$ depending only on $\lam,\Lam,z_i$ and $D_i$ for $i=1,2$, such that
 \begin{equation}\label{dissipation controls rho and electric field}
     \|\grad \cei\|_{L^2}^2 + \| \gpe +\grad \pw \|^2_{L^2} + \| \frac{\re}{\ve} \|^2_{L^2} \leq M \sum^2_{i=1}D_i \int_{\Omega} c^\ve_i | \grad \mu^\ve_i |^2\,dx.
 \end{equation}
\end{enumerate}
}
\end{lem}

\begin{proof} Notice that $\frac{c^\ve_i}{\Gamma_i}|_{\partial\Omega}=1,$ by multiplying (\ref{mass conservtion of divergence form})  by $\log \frac{c^\ve_i}{\Gamma_i}$ and
integrating over $\Omega$ and then summing up the resulting inequalities for $i=1,2,$ we find
\begin{align*} \label{dissipation 1}
    &  \sum^2_{i=1}\frac{d}{dt} \int_\Omega \Gamma_i \varphi(\frac{c^\ve_i}{\Gamma_i})\,dx = \sum^2_{i=1} \int_\Omega \pt c^\ve_i \log \frac{c^\ve_i}{\Gamma_i} \,\,dx \\
    & =- \sum^2_{i=1} \int_\Omega  ( D_i c^\ve_i \grad \mu^\ve_i-\ue c^\ve_i) (\grad \log c^\ve_i -\grad \log \Gamma_i ) \,dx \nonumber \\
    & = - \sum^2_{i=1} D_i \int_\Omega  c^\ve_i \grad \mu^\ve_i \grad \log c^\ve_i \,dx + \sum^2_{i=1} D_i \int_\Omega  c^\ve_i \grad \mu^\ve_i \grad \log \Gamma_i \,dx -\sum^2_{i=1} \int_\Omega \ue c^\ve_i \grad \log \Gamma_i \,dx,
\end{align*}
where we used  $\dive \ue =0$ and $u^\ve|_{\p\Omega}=0$ in the third equality.

While due to $\psi^\ve|_{\partial\Omega}=0,$ we get, by using integration by parts and the equations (\ref{e3}), (\ref{mass conservtion of divergence form}), that
\begin{align*}
     \frac{\ve^2}{2}\frac{d}{dt} & \| \grad \pe(t)\|^2_{L^2} = \ve^2 \int_\Omega \pt\grad \pe \grad \pe \,dx = \int_\Omega \pt \re \pe \,dx \\
    & = \sum^2_{i=1} \int_\Omega \pt c^\ve_i (z_i \pe ) \,dx = - \sum^2_{i=1} \int_\Omega  ( D_i c^\ve_i \grad \mu^\ve_i-\ue c^\ve_i) (z_i \gpe  ) \,dx \nonumber \\
    &=- \sum^2_{i=1} z_i D_i \int_\Omega  c^\ve_i \grad \mu^\ve_i \gpe  \,dx +\int_\Omega \ue \re \gpe \,dx.
\end{align*}

Whereas by taking $L^2$ inner product of the equation \eqref{e4} with $u^\ve$ and using integration by parts and \eqref{e5}, we obtain
\beno
      \frac{1}{2}\frac{d}{dt} \| \ue(t) \|^2_{L^2} = -\nu \| \grad \ue \|^2_{L^2} - \int_\Omega \ue \re (\gpe + \grad \pw) \,dx.
\eeno

By
summing  up the above equalities  and recalling the definition of  electrochemical potentials (\ref{electrochemical potential}), we obtain (\ref{dissipation equality}).

To prove (\ref{dissipation inequality}), it suffices to estimate the right hand side of (\ref{dissipation equality}). Indeed
we deduce from {Lemma \ref{maximal principle}} that
\begin{align*}
    -&\sum^2_{i=1}  \int_\Omega \ue c^\ve_i \grad \log \Gamma_i\,dx -\int_\Omega \ue \re \grad \pw \,dx + \sum^2_{i=1}D_i \int_\Omega c^\ve_i \grad \mu^\ve_i \grad \mu^*_i  \,dx \\
   & \leq \frac{1}{2}\| \ue \|^2_{L^2} + \frac{1}{2}\sum^2_{i=1}D_i \int_{\Omega} c^\ve_i | \grad \mu^\ve_i |^2\,dx +M(\Lam,W,\gamma_i,z_i,D_i)  \nonumber \\
   & \leq E^\ve  + \frac{1}{2}\sum^2_{i=1}D_i \int_{\Omega} c^\ve_i | \grad \mu^\ve_i |^2\,dx +M(\Lam,W,\gamma_i,z_i,D_i).
\end{align*}
Then (\ref{dissipation inequality}) follows from  Gronwall's inequality.

Recalling $D^* \eqdefa \min \{\dl,\dr\}$, we deduce
from  \eqref{EN}, \eqref{1.3}, (\ref{e3}), {Lemma \ref{maximal principle}}  that
\begin{align*}
    \sum^2_{i=1}D_i & \int_{\Omega} c^\ve_i | \grad \mu^\ve_i |^2\,dx \geq D^* \sum^2_{i=1} \int_{\Omega} c^\ve_i | \grad \mu^\ve_i |^2\,dx  \\
    &= D^*  \sum^2_{i=1} \int_{\Omega} c^\ve_i \bigl|\frac{\grad c^\ve_i}{c^\ve_i} + z_i(\gpe +\grad \pw)\bigr|^2 \,dx \nonumber \\
    &=D^* \sum^2_{i=1} \int_{\Omega} \frac{|\grad c^\ve_i|^2}{c^\ve_i} \,dx +2 D^*  \int_{\Omega} \grad \re (\gpe +\grad \pw) \,dx \nonumber \\
    &\quad + D^*\int_{\Omega} (\zl^2 \lc^\ve +\zr^2 \rc^\ve)|\gpe +\grad \pw|^2 \,dx \nonumber \\
    &\geq 2 D^*  \| \frac{\re}{\ve} \|^2_{L^2} + D^* (\zl^2  +\zr^2 )\lam  \| \gpe +\grad \pw \|^2_{L^2}  +\frac{D^*}{\Lam}\sum_{i=1}^2\|\grad \cei\|_{L^2}^2,
\end{align*}
which leads to (\ref{dissipation controls rho and electric field}). This completes the proof of Lemma \ref{lem:energy}.
\end{proof}

As we mentioned in the introduction, the global wellposedness of 2-D NPNS system with Dirichlet boundary conditions (\ref{ebdry})
 was settled in \cite{CI On the npns system} (see Theorem 9). While it follows from Theorem 3 of \cite{CIL NPNS far from eq} that
  the 3-D binary NPNS system with Dirichlet boundary conditions (\ref{ebdry}) has a unique  strong solution on $[0,T]$
   as long as  $U(T)=\int^T_0 \| \ue \|^4_{V} dt < \oo $. Based on this criterion,  we are able to derive an $\ve$-independent lower bound
    for the lifespan of the $\ve$-depending NPNS system (\ref{S1eq1})-(\ref{ebdry}). Precisely, we have the following Theorem.

\begin{thm}\label{coexistence of npns}
  {\sl Under the assumption  (\ref{uniform bound for ci}), one has
  \begin{itemize}

  \item 2-D system \eqref{S1eq1} with boundary conditions (\ref{ebdry}) has a unique global strong solution and there exists a positive constant $M$ so that for any $T\geq 0$,
    \begin{equation}\label{estimate of ue in H^1, 2d}
        \|\ue\|_V^2 +\int_0^T \|A\ue\|_H^2 dt \leq M(\|\ue(0)\|_V^2 + E^\ve(0))e^{ME^\ve(0)}.
    \end{equation}

\item If we assume in addition that $\|\ue(0)\|_V^2+E^\ve(0) \leq M_{in} <+\oo $. Then there exists an $\ve$-independent $T>0,$  so that the 3-D NPNS system \eqref{S1eq1} with boundary conditions (\ref{ebdry}) has a unique strong solution on $[0,T]$, and there exists an increasing function $h(a,b):\mathbb{R^+} \times \mathbb{R^+} \to \mathbb{R^+}$, so that
    \begin{equation}\label{estimate of ue in H^1, 3d}
        \|\ue\|_V^2 +\int_0^{T} \|A\ue\|_H^2 dt \leq h(\|\ue(0)\|_V^2, E^\ve(0)).
    \end{equation} \end{itemize}}
\end{thm}

\begin{proof}
    Based on  Theorem 3 of \cite{CIL NPNS far from eq}, it suffice to deal with the {\it a priori}
     estimate of $\| \ue \|_{V}$. By  taking $L^2$ inner product of (\ref{e3}) with $Au^\ve$ and  using standard  estimates for the Navier-Stokes equations in \cite{Constantin & Foias}, Lemma \ref{maximal principle}  and the  assumption (\ref{uniform bound for ci}), we find
     \begin{equation}\label{growth of ue,2d}
        \frac{d}{dt}\| \ue\|^2_{V}+\nu\|Au^\ve\|^2_H \leq M(\Lam,z_i) (\| \ue\|^4_{V}+ \|\gpe + \grad \pw\|^2_{L^2}), \qquad d=2,
    \end{equation}
    and
    \begin{equation}\label{growth of ue,3d}
        \frac{d}{dt}\| \ue\|^2_{V}+\nu\|Au^\ve\|^2_H \leq M(\Lam,z_i) (\| \ue\|^6_{V}+ \|\gpe + \grad \pw\|^2_{L^2}), \qquad d=3.
    \end{equation}
   Thus, together with estimate (\ref{dissipation inequality}) and (\ref{dissipation controls rho and electric field}), (\ref{growth of ue,2d}) gives rise to (\ref{estimate of ue in H^1, 2d}), (\ref{growth of ue,3d}) implies the existence of $T$ so that  (\ref{estimate of ue in H^1, 3d}) holds.
\end{proof}

\begin{rmk}
    It is easy to deduce the local wellposedness of the limit  system (\ref{limit1})-(\ref{limit5}) with boundary condition (\ref{limit-bdry}).
    As a consequence, we can choose $T>0$ so that both the $\ve$-depending NPNS systems and the limit system  have a unique strong solution on the time interval $[0,T]$.
\end{rmk}

Next Lemma is about the convexity of the energy function.

\begin{lem}\label{convexity of vphi}
  {\sl  For arbitrary $0<m\leq s \leq M$, we have
    \begin{equation*}
       \frac{1}{2M} (s-1)^2 \leq \vphi(s) \leq \frac{1}{2m} (s-1)^2.
    \end{equation*}}
\end{lem}
\begin{proof}
  By  using Taylor expansion and the fact that  $\vphi{''}(s)=\frac{1}{s} $, for arbitrary $0<s<M$, we have
    \begin{equation*}
        \vphi(s)=\vphi(1)+\vphi{'}(1)(s-1)+\frac{\vphi{''}(\theta)}{2}(s-1)^2 \geq \frac{1}{2M} (s-1)^2,
    \end{equation*}
    where $\theta$ is a number between $1$ and $s$. This completes the proof of the second inequality. The first inequality follows
    along the same line.
\end{proof}

\begin{rmk}\label{l2 eq varphi}
In particular, by taking $s=\frac{\cei}{c_i}$, we  deduce from Lemmas \ref{convexity of vphi}  and  \ref{maximal principle}  that $$ \int_\Omega c_i \varphi\bigl(\frac{\cei}{c_i}\bigr) \,dx \approx \|\cei -c_i \|^2_{L^2}. $$
\end{rmk}

\renewcommand{\theequation}{\thesection.\arabic{equation}}
\setcounter{equation}{0}

\section{The Convergence in $L_T^\oo(L^2)$}\label{section3}

In this section, we shall use modulated energy estimate to prove Theorem \ref{thm1}, which will be
 based on the following modulated energy functional:
\begin{equation}\label{defintion of H}
    H^\ve(t) \eqdefa \sum^2_{i=1} \int_\Omega c_i \varphi\bigl(\frac{c^\ve_i}{c_i}\bigr)(t)\,dx +\frac{\ve^2}{2} \| \grad \pe(t) \|^2_{L^2}  +\frac{1}{2}\| (\ue - u)(t) \|^2_{L^2},
\end{equation}
where $(\cle,\cre,\pe,\ue)$, $(\lc,\rc,\psi,u)$ are smooth enough solutions of NPNS system (\ref{S1eq1}-\ref{ebdry}) and the limit  system \eqref{S1eq2} on $[0,T]$ respectively.
We construct the energy dissipation functional $\Theta^\ve(t)$ as follows:
\beq\label{difinition of Theta}
\begin{split}
    \Theta^\ve  \eqdefa & \sum^2_{i=1} D_i \int_\Omega \frac{|\grad c^\ve_i-\grad c_i|^2}{c^\ve_i}\, dx+\sum^2_{i=1} z^2_i D_i \int_\Omega c^\ve_i |\gpe - \grad \psi|^2\, dx \\
    & +{D^*}\bigl\| \frac{\re}{\ve}\bigr\|^2_{L^2} + \nu \|(\grad\ue- \grad u)\|_{L^2}^2.
\end{split}\eeq

The main result states as follows:

\begin{prop}\label{S3prop1}
{\sl Let $(\cle,\cre,\pe,\ue)$, $(\lc,\rc,\psi,u)$ be smooth enough solutions of NPNS system (\ref{S1eq1}-\ref{ebdry}) and the limiting system \eqref{S1eq2} on $[0,T]$ respectively. Then for any $t\in
[0,T]$: one has
\begin{equation}\label{main goal}
\begin{split}
    \frac{d}{dt}& H^\ve(t) + \f{D_2}{4D_1} \Theta^\ve(t) \\
    \leq & M\bigl(1+\|\na u\|_{L^\infty}+\sum_{i=1}^2\|\na c_i\|_{L^\infty}^2
    +\|\na\psi\|_{L^\infty}^2
    +\|\na\Phi_W\|_{L^\infty}^2\bigr)H^\ve(t)\\
    & +M\ve^2\|\lap\psi\|_{L^2}^2
    +\int_\Omega \psi(\pt \re + \ue \cdot \grad \re ) \,dx.
   \end{split}
\end{equation}
Here and in the rest of this paper, we always denote $M$ to be a   positive constant which depends on
$\nu,\lambda, \Lambda, z_i$ and $D_i$ for $i=1,2$, unless otherwise stated, and which may vary from line to line.
}\end{prop}

\begin{proof} We shall divide the proof of Proposition \ref{S3prop1} into the following steps:

\no{\bf Step 1.} The derivation of the differential equality.

Observing that
\begin{align*}
c_i \varphi\bigl(\frac{c^\ve_i}{c_i}\bigr) &= c_i\bigl(\frac{c^\ve_i}{c_i}\log\frac{c^\ve_i}{c_i} -\frac{c^\ve_i}{c_i}+1\bigr) =c^\ve_i \log c^\ve_i -c^\ve_i\log c_i -c^\ve_i+c_i \\
&= \Gamma_i \varphi(\frac{c^\ve_i}{\Gamma_i}) - c^\ve_i(\log c_i -\log \Gamma_i) +c_i -\Gamma_i,
\end{align*}
in view of \eqref{S2eq3} and \eqref{defintion of H}, we have
\begin{equation}\label{S3eq1}
    H^\ve(t)=E^\ve(t) - \sum^2_{i=1} \int_\Omega \cei \bigl(\log c_i - \log \Gamma_i\bigr) \,dx +\sum^2_ {i=1} \int_\Omega\bigl( c_i - \Gamma_i\bigr)\, dx +\frac{1}{2}\| u \|^2_{L^2}-\int_\Omega \ue \cdot u\,dx.
\end{equation}

Below let us calculate the time derivative of the right-hand side of \eqref{S3eq1}.
Indeed it follows from  (\ref{dissipation equality}) that
 \begin{align*}
  \frac{d}{dt}E^\ve(t) = &-\nu \| \grad \ue \|^2_{L^2} - \sum^2_{i=1}D_i \int_{\Omega} c^\ve_i | \grad \mu^\ve_i |^2\,dx  -\sum^2_{i=1} \int_\Omega \ue c^\ve_i \grad \log \Gamma_i \,dx\\
  &-\int_\Omega \ue \re \grad \pw \,dx + \sum^2_{i=1}D_i \int_\Omega c^\ve_i \grad \mu^\ve_i \grad (\log \Gamma_i + z_i \pw)  \,dx.   \nonumber
    \end{align*}

While by virtue of (\ref{mass conservtion of divergence form}), \eqref{ci mass conservtion of divergence form} and $\frac{\cei}{c_i}|_{\p \Omega}=1,$  we get,
by using integration by parts, that
\begin{align*}\label{dt cei logci-loggi}
-\frac{d}{dt} & \int_\Omega \cei  (\log c_i - \log \Gamma_i) \,dx  =-\int_\Omega  \pt\cei(\log c_i-\log \Gamma_i) \,\,dx -\int_\Omega \cei \frac{\pt c_i}{c_i} \,dx \\
=&-\int_\Omega \dive (D_i\cei \grad \mu^\ve_i - \ue \cei)(\log c_i - \log \Gamma_i) \,dx
   - \int_\Omega \frac{\cei}{c_i} \dive (D_i c_i \grad \mu_i - u  c_i) \,dx\nonumber \\
=&D_i \int_\Omega  \cei \grad \mu^\ve_i \bigl(\frac{\grad c_i}{c_i} - \grad \log \Gamma_i\bigr) \,dx -\int_\Omega  \ue \cei \bigl(\frac{\grad c_i}{c_i} - \grad \log \Gamma_i\bigr) \,dx \nonumber \\
&+D_i \int_\Omega \grad \bigl(\frac{\cei}{c_i}\bigr) c_i \grad \mu_i \,dx -D_i \int_\Omega \dive (c_i \grad \mu_i)\,dx +\int_\Omega \frac{\cei}{c_i} u\cdot \grad c_i \,dx \\
=& D_i \int_\Omega  \cei \grad \mu^\ve_i \frac{\grad c_i}{c_i} \,dx
+D_i \int_\Omega \grad \bigl(\frac{\cei}{c_i}\bigr) c_i \grad \mu_i \,dx
- \int_\Omega (\ue-u)\frac{\cei}{c_i}\grad c_i  \,dx \\
&  -D_i \int_\Omega  \cei \grad \mu^\ve_i  \grad \log \Gamma_i \,dx +\int_\Omega  \ue \cei \grad \log \Gamma_i \,dx -D_i \int_\Omega \dive (c_i \grad \mu_i)\,dx.
\end{align*}

Thanks to (\ref{ci mass conservtion of divergence form}), one has
\begin{equation*}\label{dt ci}
    \frac{d}{dt} \int_\Omega c_i \,dx = D_i \int_\Omega \dive (c_i \grad \mu_i)\,dx.
\end{equation*}

Whereas due to $\dive u=0$ and $u|_{\p \Omega}=0,$ we get, by using energy method to equation (\ref{limit4}), that
\begin{equation*}\label{dt u}
    \frac{1}{2}\frac{d}{dt} \|u\|^2_{L^2}=-\nu\|\grad u \|^2_{L^2}.
\end{equation*}

Finally thanks to  (\ref{e5}) and (\ref{limit5}), we get, by using integration by parts, that
\begin{align*}\label{dt ue u}
    -\frac{d}{dt} &\int_\Omega \ue u \,dx = -\int_\Omega\pt \ue u \,dx -\int_\Omega\ue \pt u \,dx \\
    =&-\int_\Omega \bigl(\nu \lap \ue -\ue\cdot \grad \ue +\ve^2 \dive (\grad \psi^\ve \otimes \grad \psi^\ve) -\re \grad \pw\bigr)\cdot u \,dx \nonumber \\
    & -\int_\Omega \ue \cdot(\nu \lap u -u\cdot \grad u ) \,dx \nonumber \\
    =&2\nu \int_\Omega\grad \ue : \grad u \,dx +\int_\Omega (\ue\cdot\grad \ue) \cdot u \,dx + \int_\Omega (u \cdot \grad u)\cdot \ue \,dx \nonumber\\
    &+ \ve^2 \int_\Omega\grad u: \grad \psi^\ve \otimes \grad \psi^\ve \,dx +\int_\Omega u \re \grad\pw \,dx.
    \end{align*}

Thanks to \eqref{S3eq1}, we get,
by summarizing the above equalities, that
\beq \label{S3eq2}
    \frac{d}{dt} H^\ve(t)=-\nu \| \grad (\ue-u) \|^2_{L^2} +I_1+I_2+I_3,
\eeq
where
\beq\label{S3eq3}
\begin{split}
    I_1 \eqdefa& \int_\Omega (\ue\cdot\grad \ue) \cdot u \,dx + \int_\Omega (u \cdot \grad u)\cdot \ue \,dx,\\
    I_2 \eqdefa & \ve^2 \int_\Omega\grad u: \grad \psi^\ve \otimes \grad \psi^\ve \,dx -\sum ^2_{i=1} \int_\Omega (\ue-u)\bigl(\frac{\cei}{c_i}-1\bigr)\grad c_i  \,dx\\
    &  - \int_\Omega (\ue -u ) \re \grad \pw \,dx,\\
        I_3 \eqdefa&  - \sum^2_{i=1}D_i \int_{\Omega} c^\ve_i | \grad \mu^\ve_i |^2\,dx +\sum^2_{i=1} D_i\int_\Omega  \cei \grad \mu^\ve_i \frac{\grad c_i}{c_i} \,dx \\
    & +\sum^2_{i=1} D_i \int_\Omega \grad \Bigl(\frac{\cei}{c_i}\Bigr) c_i \grad \mu_i \,dx +\sum^2_{i=1} z_i D_i \int_\Omega c^\ve_i \grad \mu^\ve_i \cdot \grad \pw   \,dx.
\end{split}\eeq

\no{\bf Step 2.} The estimates of $I_1$ to $I_2.$

 We first get, by  using  $\dive\ue=\dive u=0$ and the homogeneous
boundary conditions of $\ue$ and $u,$ that
\beno
\begin{split}
    I_1&=\int_\Omega (\ue\cdot\grad \ue) \cdot u \,dx + \int_\Omega (u \cdot \grad u)\cdot \ue \,dx = \int_\Omega \big((\ue-u)\cdot\grad u \big) \cdot (u-\ue) \,dx,
\end{split} \eeno
from which, we infer
\beq\label{estimate of i1}
|I_1|\leq \|\grad u\|_{L^\oo}\|\ue-u\|^2_{L^2}.
\eeq

Similarly, in view of $\re=\sum_{i=1}^2 z_i(\cei-c_i)$,  one has
\beq\label{estimate of i2}
\begin{split}
|I_2|
    \leq & \ve^2  \|\grad u\|_{L^\oo}\|\gpe\|^2_{L^2} +  \bigl(\sum_{i=1}^2\|\grad c_i\|_{L^\oo}\|\frac{\cei}{c_i} -1  \|_{L^2} +  \|\grad\pw\|_{L^\oo}\|\re\|_{L^2}\bigr)\|\ue - u \|_{L^2} \\
     \leq & M\bigl(\|\grad u\|_{L^\oo}+\sum_{i=1}^2\|\grad c_i\|_{L^\oo}+  \|\grad\pw\|_{L^\oo}\bigr)H^\ve.
\end{split} \eeq

\no{\bf Step 3.} The estimates of  $I_3.$

In view of (\ref{electrochemical potential}), we write
\beq\label{S3eq4}
\begin{split}
I_3
= -\sum ^2_{i=1} D_i \int_\Omega \frac{|\grad \cei - \grad c_i|^2}{\cei}\,dx -\sum ^2_{i=1}  z_i^2 D_i\int_\Omega \cei |\gpe - \grad \psi |^2 \,dx +I_{31} +I_{32} +I_{33}, \
\end{split}\eeq
where we denote $I_{31}$, $I_{32}$ and $I_{33}$ respectively by
\begin{align*}
    I_{31}\eqdefa &-\sum ^2_{i=1}  D_i \int_\Omega \frac{2\grad \cei \cdot \grad c_i - |\grad c_i|^2}{\cei} \,dx\\
     &+\sum ^2_{i=1} D_i \bigl(\int_\Omega \frac{\grad \cei \cdot \grad c_i }{c_i}\,dx +\int_\Omega (\grad \cei-\frac{\cei}{c_i}\grad c_i)\frac{\grad c_i}{c_i} \,dx\bigr),\\
     I_{32}\eqdefa &  \sum ^2_{i=1}z_i D_i \int_\Omega \frac{\cei}{c_i}\grad c_i (\gpe - \grad \psi )\,dx -\sum ^2_{i=1}z_i D_i \int_\Omega \grad \cei(\gpe - \grad \psi)\,dx \nonumber \\
    &-\sum ^2_{i=1}z_i D_i \int_\Omega \bigl(\grad \cei+z_i \cei(\grad \psi + \grad \pw)\bigr)(\gpe +\grad \pw)\,dx  \\
    &  -\sum ^2_{i=1}  z_i^2 D_i \int_\Omega \cei   (\gpe- \grad \psi)(\grad \psi +  \grad \pw )  \,dx, \\
    I_{33}\eqdefa & \sum ^2_{i=1}z_i D_i \int_\Omega \cei \grad \mu^\ve_i \cdot \grad\pw \,dx.
\end{align*}

\no{\bf Step 3.1} The estimates of  $I_{31}.$

The estimate of $I_{31}$ is straightforward. Notice that
 $$\frac{1}{c_i}- \frac{1}{\cei}=\bigl(\frac{\cei }{c_i}-1\bigr)\frac{1}{\cei}   \andf
 \frac{\cei}{c_i^2} -\frac{1}{\cei} =\bigl(\frac{\cei }{c_i}+1\bigr)\bigl(\frac{\cei }{c_i}-1\bigr)\frac{1}{\cei},$$
   we write
\begin{align*}
    I_{31}
   & =2\sum ^2_{i=1} D_i \int_\Omega \grad \cei \cdot \grad c_i \bigl(\frac{\cei }{c_i}-1\bigr)\frac{1}{\cei}\,dx -\sum ^2_{i=1} D_i \int_\Omega\bigl(\frac{\cei }{c_i}+1\bigr) |\grad c_i|^2 \bigl(\frac{\cei }{c_i}-1\bigr)\frac{1}{\cei}\,dx   \\
   &=2 \sum ^2_{i=1} D_i \int_\Omega  (\grad \cei- \grad c_i)\cdot \grad c_i \bigl(\frac{\cei }{c_i}-1\bigr)\frac{1}{\cei}\,dx  -\sum ^2_{i=1} D_i \int_\Omega |\grad c_i|^2 \bigl(\frac{\cei }{c_i}-1\bigr)^2 \frac{1}{\cei} \,dx.
  \end{align*}
 Then we get, by using Lemma \ref{maximal principle}, Lemma \ref{convexity of vphi}  and H{\"o}lder's inequality,  that for any $\delta>0,$
 \begin{align*}
   |I_{31}|
   & \leq \delta \sum_{i=1}^2D_i \int_\Omega \frac{ |\grad \cei- \grad c_i|^2}{c_i^\ve}\,dx+\f{2}\delta \sum ^2_{i=1} D_i \int_\Omega |\grad c_i|^2 \bigl(\frac{\cei }{c_i}-1\bigr)^2 \frac{1}{\cei} \,dx\\
   &\leq \delta
   \Theta^\ve +\f{M}\delta \bigl(\sum_{i=1}^2\|\grad c_i\|_{L^\infty}^2\bigr) H^\ve.
\end{align*}

\no{\bf Step 3.2} The estimates of  $I_{32}.$

We split $I_{32}$ further into the following two parts:
\begin{align*}
I_{32}=&J_1+J_2\with\\
J_1 \eqdefa &\sum ^2_{i=1}z_i D_i \int_\Omega \frac{\cei}{c_i}\grad c_i (\gpe - \grad \psi )\,dx -\sum ^2_{i=1}z_i D_i \int_\Omega \grad \cei(\gpe - \grad \psi)\,dx,\\
  J_2 \eqdefa & -\sum ^2_{i=1}z_i D_i \int_\Omega \big(\grad \cei+z_i \cei(\grad \psi + \grad \pw)\big)(\gpe +\grad \pw)\,dx  \\
    & -\sum ^2_{i=1}  z_i^2 D_i \int_\Omega \cei   ( \gpe- \grad \psi)(\grad \psi +  \grad \pw )  \,dx.
    \end{align*}
 It is easy to observe that
 \beno
 J_1
    =\sum ^2_{i=1}z_i D_i \int_\Omega \bigl(\frac{\cei}{c_i}-1\bigr) \grad c_i (\gpe - \grad \psi )\,dx - \sum ^2_{i=1}z_i D_i \int_\Omega (\grad \cei-\grad c_i)(\gpe - \grad \psi)\,dx,
    \eeno
    so that one has
\begin{align*}
    |J_1|
    &\leq \bigl(\delta+\frac12\bigr)\sum_{i=1}^2D_iz_i^2\int_\Omega \cei|\gpe - \grad \psi|^2\,dx +\frac12 \sum_{i=1}^2D_i\int_\Omega\frac{ |\grad \cei-\grad c_i|^2}{\cei}\,dx\\
    &\quad+\f{M}{\delta} \sum_{i=1}^2
    \int_\Omega\frac{|\grad c_i|^2}{\cei}  \bigl(\frac{\cei}{c_i}-1\bigr)^2\,dx \\
    &\leq \bigl(\delta+\f12\bigr) \Theta^\ve +\frac{M}{\delta}\bigl(\sum_{i=1}^2\|\nabla c_i\|_{L^\infty}^2 \bigr) H^\ve.
\end{align*}

The estimate of the rest part in $I_{32}$  relies on the evolution equation of $\rho^\ve$. We first write
\begin{align*}
    J_2
    &=-\int_\Omega (\gpe-\grad \psi) \Big(\zl \dl \grad \cle + \zr \dr \grad \cre +(\zl^2 \dl \cle + \zr^2\dr\cre)(\grad \psi +\grad \pw) \Big)\,dx \nonumber\\
    &\quad-\int_\Omega (\grad \psi +\grad \pw) \Big(\zl \dl \grad \cle + \zr \dr \grad \cre +(\zl^2 \dl \cle + \zr^2\dr\cre)(\grad \psi +\grad \pw) \Big)\,dx \\
    &\quad -\int_\Omega (\grad \psi +\grad \pw)(\gpe -\grad \psi)(\zl^2 \dl \cle + \zr^2\dr\cre)\,dx.
    \end{align*}
    As $\rho^\ve=z_1c_1^\ve+z_2c_2^\ve,$ we have
    \begin{align*}
    J_2
    &= -\int_\Omega (\gpe-\grad \psi)  \Big( \zl (\dl-\dr) \grad \cle + \dr \grad \re\\
     &\qquad\qquad\qquad\qquad\qquad+\big(\zl (\zl \dl-\zr \dr) \cle + \zr\dr\re \big)                       (\grad \psi +\grad \pw) \Big)\,dx\\
    &\quad - \int_\Omega (\grad \psi +\grad \pw) \Big(\zl \dl \grad \cle + \zr \dr \grad \cre +(\zl^2 \dl \cle + \zr^2\dr\cre)(\gpe +\grad \pw) \Big)\,dx,\end{align*}
    which together with the  $\psi$ equation  (\ref{equivalent psi}) ensures that
     \begin{align*}
    J_2
    &= -\int_\Omega (\gpe-\grad \psi)  \Big( \zl (\dl-\dr) (\grad \cle -\grad c_1)\\
    &\qquad\qquad\qquad\qquad\qquad  +\big(\zl (\zl \dl-\zr \dr) c_1 (\frac{\cle}{\lc}-1) \big)                    (\grad \psi +\grad \pw) \Big)\,dx \nonumber\\
    &\quad -D_2\int_\Omega (\gpe-\grad \psi)\grad \re \,dx -z_2 D_2\int_\Omega (\gpe-\grad \psi)\re (\grad \psi +\grad \pw) \,dx \\
    & \quad - \int_\Omega (\grad \psi +\grad \pw) \big(\zl \dl \cle \grad \mu^\ve_1 + \zr \dr \cre \grad \mu^\ve_2 \big)\,dx.
\end{align*}

For the first line in $J_2$, we get, by applying H\"older's inequality, that
\begin{align*}
    & \bigl|\int_\Omega (\gpe-\grad \psi)  \big( \zl (\dl-\dr) (\grad \cle -\grad c_1)\\
     &\qquad\qquad\qquad\quad+\big(\zl (\zl \dl-\zr \dr) c_1 (\frac{\cle}{\lc}-1) \big)(\grad \psi +\grad \pw) \big)\,dx \bigr|\\
    & \leq \bigl(1-\f{D_2}{D_1}\bigr)D_1\Bigl(\bigl(\delta+\f{1}2\bigr)z_1^2\int_\Omega c_1^\ve |\gpe-\grad \psi|^2\,dx+\int_\Omega \f{|\grad \cle -\grad c_1|^2}{c_1^\ve}\,dx\Bigr)\\
    &\quad+\f{M}\delta\bigl(\|\na\psi\|_{L^\infty}^2+\|\na\Phi_W\|_{L^\infty}^2\bigr)\|c_1^\ve-c_1\|_{L^2}^2\\
    &\leq  \bigl(\frac{1}{2}-\frac{D_2}{2D_1}\bigr)\Theta^\ve +  \delta \Theta^\ve +\f{M}\delta\bigl(\|\na\psi\|_{L^\infty}^2+\|\na\Phi_W\|_{L^\infty}^2\bigr)H^\ve.
\end{align*}

For the second line in $J_2$, by using integration by parts and the equation (\ref{e3}), we obtain
\begin{align*}
     &-D_2\int_\Omega (\gpe-\grad \psi)\grad \re \,dx -z_2 D_2\int_\Omega (\gpe-\grad \psi)\re (\grad \psi +\grad \pw) \,dx  \\
     &\leq  -D_2\|\frac{\re}{\ve}\|^2_{L^2} +D_2\|\lap \psi\|_{L^2} \|\re\|_{L^2} +D_2|z_2|\bigl(\|\na\psi\|_{L^\infty}+\|\na\Phi_W\|_{L^\infty}\bigr)\|\gpe-\grad \psi\|_{L^2}
     \|\re\|_{L^2} \\
     &\leq  -D_2\|\frac{\re}{\ve}\|^2_{L^2} + \delta D_2\|\frac{\re}{\ve}\|^2_{L^2} +\delta D_2z_2^2\int_\Omega c_2^\ve |\gpe-\grad \psi|^2\,dx \\
&\quad +\f{M}\delta\bigl(\|\na\psi\|_{L^\infty}^2+\|\na\Phi_W\|_{L^\infty}^2\bigr) H^\ve+\frac{M}{\delta}\ve^2\|\lap \psi\|_{L^2}^2.\end{align*}

Finally by virtue of \eqref{mass conservtion of divergence form}, we find
\begin{align*}
    &- \int_\Omega (\grad \psi +\grad \pw) \Big(\zl \dl \cle \grad \mu^\ve_1 + \zr \dr \cre \grad \mu^\ve_2 \Big)\,dx +I_{33} \\
    & = -\int_\Omega \grad \psi \Big(\zl \dl \cle \grad \mu^\ve_1 + \zr \dr \cre \grad \mu^\ve_2 \Big)\,dx  \\
    & = \int_\Omega \psi \dive\Big(\zl \dl \cle \grad \mu^\ve_1 + \zr \dr \cre \grad \mu^\ve_2 \Big)\,dx \\
    & = \int_\Omega \psi(\pt \re + \ue \cdot \grad \re ) \,dx.
\end{align*}

By substituting the above estimates into \eqref{S3eq4}, we achieve
\beq \label{S3eq5}
\begin{split}
I_3
\leq & -\sum ^2_{i=1} D_i \int_\Omega \frac{|\grad \cei - \grad c_i|^2}{\cei}\,dx -\sum ^2_{i=1}  z_i^2 D_i\int_\Omega \cei |\gpe - \grad \psi |^2 \,dx
-D_2\|\frac{\re}{\ve}\|^2_{L^2}\\
&+\bigl(1+4\delta-\f{D_2}{2D_1}\bigr)\Theta^\ve +\f{M}\delta \bigl(1+\sum_{i=1}^2\|\grad c_i\|_{L^\infty}^2+\|\na\psi\|_{L^\infty}^2+\|\na\Phi_W\|_{L^\infty}^2\bigr) H^\ve\\
&+ M\ve^2\|\lap \psi\|_{L^2}^2
+\int_\Omega \psi(\pt \re + \ue \cdot \grad \re ) \,dx.
\end{split}
\eeq

By inserting the estimates \eqref{estimate of i1}, \eqref{estimate of i2} and \eqref{S3eq5} into \eqref{S3eq2} and then
taking $\delta=\f{D_2}{16D_1},$ we arrive at \eqref{main goal}. This completes the proof of Proposition \ref{S3prop1}.
\end{proof}

Now we are in a position to complete the proof of Theorem \ref{thm1}.

\begin{proof}[Proof of Theorem \ref{thm1}]Firstly, together with (\ref{assumption 1}-\ref{assumtion 2}) and $\pe(0)=0$, Theorem \ref{coexistence of npns} implies the existence of $T$ so that the systems \eqref{S1eq1} and \eqref{S1eq2} have unique strong solutions on $[0,T].$ Since the initial data is ``well-prepared", by virtue of (\ref{e3}), we get, by using integration by parts, that
\beq\label{S3eq6}
\begin{split}
    \int_0^t &\int_\Omega \psi(\pt \re + \ue \cdot\grad \re ) \,dxdt \\
     = & \int_\Omega \psi \re \,dx \big|^t_0  -\int_0^t \int_\Omega \pt \psi \re \,dxdt - \int_0^t \int_\Omega u^\ve \cdot \grad \psi \re \,dxdt \\
    =& \ve^2 \int_\Omega \grad \psi (t) \cdot \grad \psi^\ve (t) \,dx  -\int_0^t \int_\Omega (\pt \psi +u^\ve \cdot \grad \psi) \re \,dxdt   \\
    \leq & \f{\ve^2}4\|\gpe(t)\|_{L^2}^2+\ve^2\|\na\psi(t)\|_{L^2}^2+\f{D_2}{8D_1}D^*\bigl\|\f{\rho^\ve}{\ve}\bigr\|_{L^2_t(L^2)}^2\\
    &+M\ve^2\bigl(\|\pt\psi\|_{L^2_t(L^2)}^2+\|\na\psi\|_{L^2_t(L^\infty)}^2\|u^\ve\|_{L^\infty_t(L^2)}^2\bigr).
\end{split}\eeq
Observing that $\|u^\ve\|_{L^\infty_t(L^2)}^2\lesssim 1,$
by integrating \eqref{main goal} over $[0,t]$ and then inserting \eqref{S3eq6} to the resulting inequality, we find
\begin{align*}
    \f12& H^\ve(t) + \f{D_2}{8D_1} \int_0^t\Theta^\ve(t')\,dt'\\
    \leq & H^\ve(0)+M\ve^2\bigl(\|\na\psi\|_{L^\infty_T(L^2)}^2+\|\pt\psi\|_{L^2_T(L^2)}^2+\|\na\psi\|_{L^2_T(L^\infty)}^2+\|\lap\psi\|_{L^2_T(L^2)}^2\bigr)\\
    &
    +M\int_0^t\bigl(1+\|\na u\|_{L^\infty}+\sum_{i=1}^2\|\na c_i\|_{L^\infty}^2
    +\|\na\psi\|_{L^\infty}^2
    +\|\na\Phi_W\|_{L^\infty}^2\bigr)H^\ve(t')\,dt'.
   \end{align*}
By  using Gronwall's inequality, we achieve
\beq \label{S3eq7}
\begin{split}
H^\ve(t)& + \f{D_2}{16D_1} \int_0^t\Theta^\ve(t')\,dt'\\
\leq &  \Bigl(H^\ve(0)+M\ve^2\bigl(\|\na\psi\|_{L^\infty_T(L^2)}^2+\|\pt\psi\|_{L^2_T(L^2)}^2+\|\na\psi\|_{L^2_T(L^\infty)}^2 +\|\lap\psi\|_{L^2_T(L^2)}^2\bigr)\Bigr)\\
&\times\exp\Bigl(M\int_0^t\bigl(1+\|\na u\|_{L^\infty}+\sum_{i=1}^2\|\na c_i\|_{L^\infty}^2
    +\|\na\psi\|_{L^\infty}^2
    +\|\na\Phi_W\|_{L^\infty}^2\bigr)(t')\,dt'\Bigr),
    \end{split}\eeq
which together with \eqref{defintion of H} and \eqref{difinition of Theta} leads to (\ref{result 1 in thm 1}-\ref{result 3 in thm 1}). We thus complete the proof of Theorem \ref{thm1}.
\end{proof}

\begin{rmk} We remark that it is crucial to deal with the term $ \int_0^t\int_\Omega \psi(\pt \re + \ue \cdot\grad \re ) \,dxdt$ as that in \eqref{S3eq7}.
Indeed in view of (\ref{e1}) and \eqref{e2}, we have
\begin{align*}
\pt \re &+ \ue \cdot\grad \re=D_2\dive\bigl(\na\re+z_2\re(\na\psi^\ve+\na\Phi_W)\bigr)\\
&+\dive\bigl((D_1-D_2)z_1\na c_1^\ve+(D_1z_1-D_2z_2)z_1 c_1^\ve(\na\psi^\ve+\na\Phi_W)\bigr),
\end{align*}
which together with \eqref{equivalent psi} ensures that
\begin{align*}
\pt \re &+ \ue \cdot\grad \re=D_2\dive\bigl(\na\re+z_2\re(\na\psi^\ve+\na\Phi_W)\bigr)\\
&+\dive\bigl((D_1-D_2)z_1\na (c_1^\ve-c_1)+(D_1z_1-D_2z_2)z_1 \bigl(c_1^\ve(\na\psi^\ve+\na\Phi_W)-c_1(\na\psi+\na\Phi_W)\bigr).
\end{align*}
Yet it is impossible for us to gain $\ve^2$ for the following term:
\begin{align*}
\int_0^t\int_\Omega& \psi\dive\bigl((D_1-D_2)z_1\na (c_1^\ve-c_1)\\
&+(D_1z_1-D_2z_2)z_1 \bigl(c_1^\ve(\na\psi^\ve+\na\Phi_W)-c_1(\na\psi+\na\Phi_W)\bigr)\,dx\,dt'.
\end{align*}
\end{rmk}

\renewcommand{\theequation}{\thesection.\arabic{equation}}
\setcounter{equation}{0}
\section{Analysis of layers}\label{analysis of layers}

In order to estimate the difference between the solutions of the system \eqref{S1eq1} and that of \eqref{S1eq2} in stronger norm than $L^\infty_T(L^2),$
we need to analysis the initial layer and weak boundary layer. For simplicity, here and in the rest of this paper we always take  $\Omega=\mathbb{T}^{d-1} \times [0,1]$ for $ d=2$ or $3$,
and denote the velocity field $u^\ve=(v^\ve, w^\ve)$, where $v^\ve$ is the first $d-1$ components of $u^\ve$.  Motivated by \cite{WXM06,WW12}, we introduce the smooth cut-off functions $f(y)$, $g(y)$ as
\begin{equation}\label{f}
    \quad f(y)= \left\{\begin{array}{l}
    1 ,\quad 0\leq y \leq \frac{1}{4}, \\
     0, \quad \frac{1}{2}\leq y \leq 1,
    \end{array}
    \right.
\andf
    \quad g(y)= \left\{\begin{array}{l}
    0, \quad  0\leq y \leq \frac{1}{2}, \\
     1, \quad \frac{3}{4}\leq y \leq 1.
    \end{array}
    \right.\end{equation}
    And in what follows, we denote
 $\tau \eqdefa \frac{t}{\ve^2}$, $\xi\eqdefa\frac{y}{\ve}$, $\eta\eqdefa \frac{1-y}{\ve}.$ Then formally, we expand $(\cei,\Phi^\ve,v^\ve,w^\ve,p^\ve)$ as follows:
\begin{equation}\label{expansion}
    \quad \left\{\begin{array}{l}
     \cei(x',y,t)= \sum_{k=0}^{+\oo} \ve^k  \Big( \cik(x',y,t)+\ciik(x',y,\tau)+f(y)\cilbk(x',\xi,t) \\
    \qquad\qquad\qquad  +g(y)\cirbk(x',\eta,t) +f(y)\cilmk(x',\xi,\tau)+g(y)\cirmk(x',\eta,\tau) \Big),\\
     \Phi^\ve(x',y,t)=\sum_{k=0}^{+\oo} \ve^k  \Big( \pk(x',y,t)+\pik(x',y,\tau)+f(y)\plbk(x',\xi,t)\\
     \qquad\qquad\qquad  +g(y)\prbk(x',\eta,t)  +f(y) \plmk (x',\xi,\tau)+g(y)\prmk(x',\eta,\tau) \Big),\\
    v^\ve(x',y,t)=\sum_{k=0}^{+\oo} \ve ^k  \Big( \vk(x',y,t)+\vik(x',y,\tau)+f(y)\vlbk(x',\xi,t) \\
     \qquad\qquad\qquad +g(y)\vrbk(x',\eta,t)+f(y)\vlmk(x',\xi,\tau)+g(y)\vrmk(x',\eta,\tau) \Big) ,\\
     w^\ve(x',y,t)=\sum_{k=0}^{+\oo} \ve ^k  \Big( \wk(x',y,t)+\wik(x',y,\tau)+f(y)\wlbk(x',\xi,t) \\
    \qquad\qquad\qquad   +g(y)\wrbk(x',\eta,t)+f(y)\wlmk(x',\xi,\tau)+g(y)\wrmk(x',\eta,\tau) \Big),\\
    p^\ve(x',y,t)=\sum_{k=-2}^{+\oo} \ve ^k  \Big( \prek(x',y,t)+\preik(x',y,\tau)+f(y)\prelbk(x',\xi,t)\\
     \qquad\qquad\qquad +g(y)\prerbk(x',\eta,t) +f(y)\prelmk(x',\xi,\tau)+g(y)\prermk(x',\eta,\tau) \Big) ,
    \end{array}
    \right.
\end{equation}
where we set the subscript ``I", ``LB", ``RB", ``LM", ``RM" to represent the parts of the solutions in ``initial layer", ``boundary layer near $y=0$", ``boundary layer near $y=1$", ``mixed layer near $y=0$", ``mixed layer near $y=1$", respectively.

By plugging the expansions (\ref{expansion}) into \eqref{S1eq1} and comparing the coefficients  of $\ve^k$, we formally derive the equations
for $\bigl(c_i^{(k)},\Phi^{(k)},v^{(k)},w^{(k)},p^{(k)}\bigr)$ and the initial, boundary and mixed layers.

\begin{rmk}
    Expansion (\ref{expansion}) is only valid under the assumption of ``well-prepared" initial condition and the boundary condition $(\mathbf{EN})$. The expansion of pressure function starts from order of $\ve^{-2}$ because we have to balance the term $\lap \ue$, $\p _t \ue$ in the boundary and initial layers, respectively.
\end{rmk}

{\Large $\bullet$}{\bf Analysis of inner approximate solutions.}

For inner solutions $(\cik,\pk,\vk,\wk)$, it's easy to deduce that $\prefr= \prefl=\textit{Constant}$ (without loss of generality, we take $\prefr= \prefl=0$) and $(\cio,\po,\uo,\preo)=(c_i,\psi+\pw,u,p)$ with boundary conditions (\ref{limit-bdry}). Thus we have
\begin{prop}
  {\sl  Let $d=2, 3,$ $\lam \leq \cio(0) \in H^5$ and $\uo(0) \in H^5$, then system \eqref{S1eq2} admits a unique solution $(\cio,\uo)$ in $C([0,T_0]; H^5)\cap L^2(0,T_0; \Dot{H}^1\cap\Dot{H}^6)$}, where $T_0$ depends only on $\uo(0)$ and equals to $+\infty$ in the case $d=2$.
\end{prop}

$(\cil,\pl,\ul,\prel)$ satisfy the following system

\begin{equation}\label{equation of outer solution of order 1 }
\quad \left\{ \begin{array}{l}
         \pt \cil +\uo \cdot \grad \cil + \ul \cdot \grad \cio =D_i \dive (\grad \cil +z_i \cil \grad \po  +z_i \cio \grad\pl ), \\
          \rl =z_1 c_1^{\scriptscriptstyle(1)} +z_2 c_2^{\scriptscriptstyle(1)} =0, \\
         \pt \ul +\ul \cdot \grad \uo + \uo \cdot \grad \ul -\nu \lap \ul +\grad \prel =0, \\
         \dive \ul =0.
    \end{array}\right.
\end{equation}
We implement the system (\ref{equation of outer solution of order 1 }) with Dirichlet boundary conditions:
\begin{equation}\label{bdry condition of outer solution of order 1}
    \cil |_{\p \Omega} =0,\quad  \ul |_{\p \Omega} =0 \andf  \pl |_{\p \Omega} =0.
\end{equation}

Similar with system (\ref{limit1})-(\ref{limit-bdry}), by using $\rl=0$, we can rewrite the equations of $(\cil,\pl)$ as
\begin{align*}
\pt \cil +\uo \cdot \grad \cil + \ul \cdot \grad \cio = \frac{(z_1-z_2)D_1D_2}{z_1D_1-z_2D_2} \lap \cil, \\
       \dive ( \sum_{i=1}^2  z_i D_i \grad \cil +\sum_{i=1}^2 z_i^2 D_i \cil \grad \po  +\sum_{i=1}^2  z_i^2 D_i \cio \grad\pl )  =0.
\end{align*}

We have the following Proposition concerning the existence and uniqueness of solution to the system (\ref{equation of outer solution of order 1 }), the proof of which  will be omitted.

\begin{prop}
   {\sl Let $d=2, 3,$ $\lam \leq \cio(0) $ and $\bigl(\cio(0), \cil(0),\uo(0),\ul(0)\bigr) \in H^5$, then system (\ref{equation of outer solution of order 1 })-(\ref{bdry condition of outer solution of order 1}) admits a unique solution $(\cil,\ul)$ in $C([0,T_0]; H^5)\cap L^2(0,T_0; H^6)$.}
\end{prop}

$(\cir,\pr,\ur,\prer)$ satisfy the following system
\begin{equation}\label{equation of outer solution of order 2 }
\quad \left\{ \begin{array}{l}
         \pt \cir +\uo \cdot \grad \cir + \ul \cdot \grad \cil + \ur\cdot \grad \cio  =D_i \dive (\grad \cir +z_i \cir \grad \po)\\
          \qquad \qquad \qquad \qquad \qquad \qquad \qquad \qquad \qquad \qquad +z_i \cil \grad\pl + z_i \cio \grad \pr ), \\
          \rr=-\lap  \po,\\
         \pt \ur +\ur \cdot \grad \uo + \ul \cdot \grad \ul +\uo \cdot \grad \ur -\nu \lap \ur +\grad \prer =\lap \po \grad \po, \\
         \dive \ur =0 .
    \end{array}\right.
\end{equation}
To match boundary conditions (\ref{ebdry}), We implement the system (\ref{equation of outer solution of order 2 }) with the boundary conditions:
\begin{equation}\label{bdry condition of outer solution of order 2}
\begin{split}
    &\qquad \cir(x',0,t) =-\cilbr(x',0,t), \quad \cir(x',1,t) =-\cirbr(x',1,t), \\
    &\pr(x',0,t)=-\plbr(x',0,t), \quad \pr(x',1,t)=-\prbr(x',1,t) \andf \ur |_{\p \Omega} =0,
\end{split}
\end{equation}
where $(\cilbr,\cirbr,\plbr,\prbr)$ are defined in (\ref{solution of left boundary layer of order 2}), (\ref{solution of right boundary layer of order 2}). Similarly, we have
\begin{align*}
& \pt \cir +\uo \cdot \grad \cir + \ul \cdot \grad \cil + \ur\cdot \grad \cio -\frac{(z_1-z_2)D_1D_2}{z_1D_1-z_2D_2} \lap \cir \\
&\qquad = - \frac{D_i}{z_1D_1-z_2D_2}(\pt + \uo\cdot \grad) \lap\po \nonumber  + \frac{D_1D_2}{z_1D_1-z_2D_2}\lap\lap\po \nonumber \\
&\qquad\qquad  +\frac{z_1z_2D_1D_2}{z_i(z_1D_1-z_2D_2)}\dive(\lap\po \grad \po) ,\nonumber
\end{align*}

\begin{align*}
&  \dive \Big( \sum_{i=1}^2  z_i D_i \grad \cir +\sum_{i=1}^2( z_i^2 D_i \cir) \grad \po +\sum_{i=1}^2  (z_i^2 D_i \cil) \grad\pl +\sum_{i=1}^2 (z_i^2 D_i \cio )\grad \pr \Big)  \\
& \qquad \qquad \qquad\qquad \qquad \qquad =-(\pt + \uo\cdot \grad) \lap\po ,
\end{align*}
and
\begin{prop}
  {\sl  Let $d=2, 3,$ $\lam \leq \cio(0) $ and $\bigl(\cio(0), \cil(0),\cir(0),\uo(0),\ul(0),\ur(0)\bigr)$ belong to $ H^5$, then system (\ref{equation of outer solution of order 2 })-(\ref{bdry condition of outer solution of order 2}) admits a unique solution $(\cir,\ur)$ in $C([0,T_0]; H^5)\cap L^2(0,T_0; H^6)$.}
\end{prop}

{\Large $\bullet$}{\bf Analysis of boundary layer solutions.}

For boundary layers, we only deal with the boundary layers near $y=0$,  the same results hold for boundary layers near $y=1.$
 By comparing the leading order of Navier-Stokes equations, we have $\p_\xi \prelbfr =0 =\prelbfr(x',\xi \rightarrow +\oo) $, which leads to $\prelbfr = 0 $. The next order $(\cilbo,\plbo,\vlbo,\wlbo,\prelbfl)$ satisfy the following system:
\begin{equation}\label{equation of left boundary layer of order 0}
    \quad \left\{ \begin{array}{l}
         \p ^2_\xi \cilbo +z_i \p_\xi \cilbo \p_\xi \plbo +z_i (\gamma_i(x',0)+\cilbo) \p_\xi ^2 \plbo = 0, \quad 0<\xi,t<+\oo, \\
         -\p_\xi^2\plbo = \rlbo, \\
        \p_\xi^2 \vlbo =0, \\
      -\nu  \p_\xi^2 \wlbo + \p_\xi \prelbfl =0, \\
        \p_\xi \wlbo=0.
    \end{array}
    \right.
\end{equation}
We implement the system (\ref{equation of left boundary layer of order 0}) with the boundary conditions:
\begin{equation*}\label{bdry condition of left boundary layer of order 0}
     (\cilbo,\vlbo,\wlbo)(x',\xi=0) = 0 \andf (\cilbo,\plbo,\p_\xi \plbo,\vlbo,\wlbo,\prelbfl)(x',\xi \rightarrow +\oo)=0.
\end{equation*}

Thus we can take $(\cilbo,\plbo,\vlbo,\wlbo,\prelbfl)=0$, and by inserting them into the equations of $(\cilbl,\plbl,\vlbl,\wlbl,\prelbo)$, we obtain
\begin{equation}\label{equation of left boundary layer of order 1}
    \quad \left\{ \begin{array}{l}
         \p ^2_\xi \cilbl +z_i \gamma_i(x',0) \p_\xi ^2 \plbl = 0,  \quad 0<\xi,t<+\oo , \\
         -\p_\xi^2\plbl = \rlbl, \\
        \p_\xi^2 \vlbl =0, \\
         -\nu \p_\xi^2 \wlbl + \p_\xi \prelbo =0 ,\\
        \p_\xi \wlbl=0.
    \end{array}
    \right.
\end{equation}
We implement the system (\ref{equation of left boundary layer of order 1}) with  the boundary conditions:
\begin{equation*}\label{bdry condition of left boundary layer of order 1}
     (\cilbl,\vlbl,\wlbl)(x',\xi=0) = 0 \andf (\cilbl,\plbl,\p_\xi \plbl, \vlbl,\wlbl,\prelbo)(x',\xi \rightarrow +\oo)=0.
\end{equation*}

Thus we have $(\cilbl,\plbl,\vlbl,\wlbl,\prelbo)=0$ and by inserting them into the equations of \\$(\cilbr,\plbr,\vlbr,\wlbr,\prelbl)$, we obtain
\begin{equation}\label{equation of left boundary layer of order 2}
    \quad \left\{ \begin{array}{l}
         \p ^2_\xi \cilbr +z_i \gamma_i(x',0) \p_\xi ^2 \plbr = 0 , \quad 0<\xi,t<+\oo, \\
         -\p_\xi^2\plbr = \rlbr, \\
        \p_\xi^2 \vlbr =0, \\
         -\nu \p_\xi^2 \wlbr + \p_\xi \prelbl =0, \\
       \p_\xi \wlbr=0.
    \end{array}
    \right.
\end{equation}
To match the boundary condition $\re |_{\p \Omega}=0$, we implement the system (\ref{equation of left boundary layer of order 2})
with the boundary conditions:
\begin{gather*}
     \sum_{i=1}^2 z_i \cilbr(x,\xi=0) = \rlbr(x',\xi=0)=- \rr(x',y=0)=\lap \po (x',y=0), \\
    (\vlbr,\wlbr)(x',\xi=0) =0 \andf (\cilbr,\plbr,\p_\xi \plbr,\vlbr,\wlbr,\prelbl)(x',\xi \rightarrow +\oo)=0.
\end{gather*}

In view of $z_1 \gamma_1(x',0) + z_2 \gamma_2(x',0)=0$, one has
\begin{equation*}
    \p ^2_\xi (c_{\scriptscriptstyle 1,LB}^{\scriptscriptstyle (2)} + c_{\scriptscriptstyle 2,LB}^{\scriptscriptstyle (2)}) = 0,
\end{equation*}
 which implies $c_{\scriptscriptstyle 1,LB}^{\scriptscriptstyle (2)} = -c_{\scriptscriptstyle 2,LB}^{\scriptscriptstyle (2)}$ and $\rlbr = (z_1 - z_2 ) c_{\scriptscriptstyle 1,LB}^{\scriptscriptstyle (2)}$. By inserting them and  $ -\p_\xi^2\plbr = \rlbr$ into the equation of $c_{\scriptscriptstyle 1,LB}^{\scriptscriptstyle (2)}$, we achieve
\begin{equation*}
    \p _\xi^2 \rlbr -z_1(z_1-z_2)\gamma_1(x',0)\rlbr =0.
\end{equation*}

As a result, we deduce that
\begin{equation}\label{solution of left boundary layer of order 2}
\quad \left\{\begin{array}{l}
    \rlbr(x',\xi,t)= \lap \po (x',0,t) \exp{-\sqrt{z_1(z_1-z_2)\gamma_1(x',0)}\xi},\\
     c_{\scriptscriptstyle 1,LB}^{\scriptscriptstyle (2)} (x',\xi,t)= \frac{1}{z_1-z_2} \lap \po (x',0,t) \exp{-\sqrt{z_1(z_1-z_2)\gamma_1(x',0)}\xi},\\
     c_{\scriptscriptstyle 2,LB}^{\scriptscriptstyle (2)}(x',\xi,t) = \frac{-1}{z_1-z_2} \lap \po (x',0,t) \exp{-\sqrt{z_1(z_1-z_2)\gamma_1(x',0)}\xi}, \\
     \plbr(x',\xi,t)=\frac{-\lap \po (x',0,t)}{z_1(z_1-z_2)\gamma_1(x',0)} \exp{-\sqrt{z_1(z_1-z_2)\gamma_1(x',0)}\xi}, \\
     \vlbr=\wlbr=\prelbl =0 .
\end{array} \right.
\end{equation}

For $(\cirbk,\prbk,\vrbk,\wrbk,\prerbk)$, similarly one has
\begin{equation*}
    (\cirbo,\prbo,\vrbo,\wrbo,\prerbfl)=0,
\end{equation*}
\begin{equation*}
    (\cirbl,\prbl,\vrbl,\wrbl,\prerbo)=0,
\end{equation*}
and
\begin{equation}\label{solution of right boundary layer of order 2}
\quad \left\{\begin{array}{l}
    \rrbr(x',\eta,t)= \lap \po (x',1,t) \exp{-\sqrt{z_1(z_1-z_2)\gamma_1(x',1)}\eta},\\
     c_{\scriptscriptstyle 1,RB}^{\scriptscriptstyle (2)} (x',\eta,t)= \frac{1}{z_1-z_2} \lap \po (x',1,t) \exp{-\sqrt{z_1(z_1-z_2)\gamma_1(x',1)}\eta},\\
     c_{\scriptscriptstyle 2,RB}^{\scriptscriptstyle (2)}(x',\eta,t) = \frac{-1}{z_1-z_2} \lap \po (x',1,t) \exp{-\sqrt{z_1(z_1-z_2)\gamma_1(x',1)}\eta} ,\\
     \prbr(x',\eta,t)=\frac{-\lap \po (x',1,t)}{z_1(z_1-z_2)\gamma_1(x',1)} \exp{-\sqrt{z_1(z_1-z_2)\gamma_1(x',1)}\eta}, \\
     \vrbr=\wrbr=\prerbl =0.
\end{array} \right.
\end{equation}

{\Large $\bullet$}{\bf Analysis of initial layer solutions.}

For initial layers, the leading order $(\ciio,\uio,\preifr)$ satisfy the following system
\begin{equation}\label{equation of initial layer of order 0}
    \quad \left\{\begin{array}{l}
    \p _\tau \ciio = 0,  \quad x\in \Omega ,\tau >0,\\
    \rio=0,\\
    \p_\tau \uio +\grad \preifr =0, \\
    \dive \uio = 0.
    \end{array} \right.
\end{equation}
We implement the system (\ref{equation of initial layer of order 0}) with the boundary conditions:
\begin{equation*}\label{initial condition of initial layer of order 0}
     (\ciio,\uio,\preifr)(x',y,\tau \rightarrow +\oo) =0.
\end{equation*}
So that we have $(\ciio,\uio,\preifr)=0,$ from which,  we infer
\begin{equation}\label{equation of initial layer of order 1}
    \quad \left\{\begin{array}{l}
    \p _\tau \ciil = 0, \quad x\in \Omega ,\tau >0, \\
    \ril=0,\\
    \p_\tau \uil +\grad \preifl =0,\\
    \dive \uil = 0.
    \end{array} \right.
\end{equation}
We implement the system (\ref{equation of initial layer of order 1}) with the boundary conditions:
\begin{equation*}\label{initial condition of initial layer of order 1}
     (\ciil,\uil,\preifl)(x',y,\tau \rightarrow +\oo) =0.
\end{equation*}

Thus we have $(\ciil,\uil,\preifl)=0$. By plugging them into the next order, we obtain
\begin{equation}\label{equation of initial layer of order 2}
    \quad \left\{\begin{array}{l}
    \p _\tau \ciir = z_iD_i\dive(\cio(0)\grad \pio),  \quad x\in \Omega ,\tau >0 ,\\
    -\lap \pio=\rir, \\
    \p_\tau \uir +\grad \preio =0, \\
    \dive \uir = 0.
    \end{array} \right.
\end{equation}
To match the initial boundary condition ($\mathbf{EN}$), we implement the system (\ref{equation of initial layer of order 2}) with
the conditions:
\begin{equation}\label{initial condition of initial layer of order 2}
\begin{split}
 \rir(x',y,\tau=0)=-\rr(x',y,t=0)=\lap \po(t=0) ,\\
 (\ciir,\pio,\uir,\preio)(x',y,\tau \rightarrow +\oo) =0 \andf \pio |_{\p \Omega}=0.
    \end{split}
\end{equation}
It's easy to see that $\uir=\preio=0$. Now observing that $z_1 c_1^{\scriptscriptstyle(0)} +z_2 c_2^{\scriptscriptstyle(0)} = 0 $, one has
    $$\p_\tau (\frac{c_{\scriptscriptstyle 1,I}^{\scriptscriptstyle(2)}}{D_1}+\frac{c_{\scriptscriptstyle 2,I}^{\scriptscriptstyle(2)}}{D_2}) =0 .$$
Thus we get $\frac{c_{\scriptscriptstyle 1,I}^{\scriptscriptstyle(2)}}{D_1}+\frac{c_{\scriptscriptstyle 2,I}^{\scriptscriptstyle(2)}}{D_2}=0$, $\rir = \frac{z_1D_1-z_2D_2}{D_1}c_{\scriptscriptstyle 1,I}^{\scriptscriptstyle(2)} $ and then rewrite the equation of $\rir$ as
\begin{equation}\label{equation of rir}
    \p_\tau \rir = z_1(z_1D_1-z_2D_2)\dive\bigl(c_{\scriptscriptstyle 1}^{\scriptscriptstyle(0)}(0)\grad \pio \bigr).
\end{equation}

The next proposition is concerned with the exponential decay of $(\ciir,\pir)$.

\begin{Proposition}\label{prop ciir}
{\sl Let $\cio(0) \in H^5$ and $\ciir(0) \in H^2.$ Then the system (\ref{equation of initial layer of order 2} -\ref{initial condition of initial layer of order 2}) admits a unique solution in $C^\oo([0,+ \infty); H^2)$.
Moreover, there exist positive constants $M_1$, $M_2$ depending only on $\lam,z_i$ and $D_i$ for $i=1,2$, such that  for $0 \leq l \leq 3$, we have
\begin{equation*}
    \sum_{i=1}^2\|\p_\tau^l \ciir(\tau)\|^2_{H^2} + \|\p_\tau^l \pio(\tau)\|^2_{H^4} \leq M_1 (\sum_{i=1}^2 \|\cio(0)\|_{H^5}^{12}+1) \sum_{i=1}^2 \|\ciir(0)\|^2_{H^2} e^{- M_2 \tau}.
\end{equation*}}
\end{Proposition}
\begin{proof} For simplicity, we only deal with the \textit{a priori} estimate for the system  (\ref{equation of initial layer of order 2}).
  We   get,  by taking $L^2$ inner product of \eqref{equation of rir} with $\pio$,  that
\begin{equation*}
    \frac{1}{2}\frac{d}{d\tau}  \|\grad \pio\|^2_{L^2}  + z_1(z_1D_1-z_2D_2) \int_\Omega c_{\scriptscriptstyle 1}^{\scriptscriptstyle(0)}(0)|\grad \pio|^2 \,dx =0 .
\end{equation*}
Due to $c_{\scriptscriptstyle 1}^{\scriptscriptstyle(0)}(0) \geq \lam >0$, we achieve
\begin{equation*}
    \frac{1}{2}\frac{d}{d\tau}  \|\grad \pio\|^2_{L^2}  + z_1(z_1D_1-z_2D_2) \lambda \|\grad \pio\|^2_{L^2} \leq 0,
\end{equation*}
which implies
\begin{equation}\label{pio exponentially decay}
    \|\grad \pio(\tau)\|^2_{L^2} \leq \|\grad \pio(0)\|^2_{L^2} \exp\bigl(-2z_1(z_1D_1-z_2D_2) \lambda \tau \bigr).
\end{equation}

While we get, by  inserting $-\lap \pio = \rir$ into equation (\ref{equation of rir}), that
\begin{equation}\label{equation of rir,rewrite}
    \frac{1}{z_1(z_1D_1-z_2D_2)}\p_\tau \rir+ c_{\scriptscriptstyle 1}^{\scriptscriptstyle(0)}(0)\rir =  \grad c_{\scriptscriptstyle 1}^{\scriptscriptstyle(0)}(0) \cdot \grad \pio.
\end{equation}
Then by virtue of $c_{\scriptscriptstyle 1,I}^{\scriptscriptstyle(0)}(0) \geq \lam >0$ and (\ref{pio exponentially decay}), we get, by multiplying equation (\ref{equation of rir,rewrite}) by $\rir$ and integrating over $\Omega$, that
\begin{align*}
     \frac{1}{2z_1(z_1D_1-z_2D_2)}&\frac{d}{d\tau} \|\rir\|^2_{L^2} +\lam \|\rir\|^2_{L^2} \leq  \frac{1}{2}\lam \|\rir\|^2_{L^2} \\
     &+  \frac{1}{2\lam} \|\grad c_{\scriptscriptstyle 1}^{\scriptscriptstyle(0)}(0) \|^2_{L^\oo}\|\grad \pio(0)\|^2_{L^2} \exp\bigl(-2z_1(z_1D_1-z_2D_2) \lambda \tau \bigr).
\end{align*}
By using  Gronwall's inequality, one has
\begin{equation*}
     \|\rir\|^2_{L^2} \leq  \bigl(\|\rir(0)\|^2_{L^2} + \frac{1}{2\lam^2} \|\grad c_{\scriptscriptstyle 1}^{\scriptscriptstyle(0)}(0) \|^2_{L^\oo}\|\grad \pio(0)\|^2_{L^2} \bigr) \exp\bigl( -z_1(z_1D_1-z_2D_2) \lambda \tau \bigr).
\end{equation*}

While by taking spatial or time derivative of (\ref{equation of initial layer of order 2}) and (\ref{equation of rir,rewrite}), we get  along the same line, for $0\leq l \leq 3$ and for some $M_1, M_2$, that
\begin{align*}
    \|\p_\tau^l \rir\|^2_{H^2}
    &\leq M_1 \Bigl( \|\p_\tau^l \rir(0)\|_{H^2}^2 + \sum_{j=1}^3\|\grad^j \cio(0)\|_{L^\oo}^{8-2j} \|\grad \p_\tau^l \pio(0)\|_{L^2}^2 \\
    &\quad + \sum_{j=1}^2\|\grad^j \cio(0)\|_{L^\oo}^{6-2j} \|\p_\tau^l\rir(0)\|_{L^2}^2 + \|\grad \cio(0)\|_{L^\oo}^2\|\grad \p_\tau^l\rir(0)\|_{L^2}^2    \Bigr) e^{-M_2\tau}\\
    & \leq M_1 \sum_{j=0}^3 \|\cio(0)\|_{H^5}^{2j} \|\p_\tau^l \rir(0)\|^2_{H^2} e^{- M_2 \tau}.
\end{align*}
Observing that $$\rir = \frac{z_1D_1-z_2D_2}{D_1}c_{\scriptscriptstyle 1,I}^{\scriptscriptstyle(2)} =\frac{z_2D_2-z_1D_1}{D_2}c_{\scriptscriptstyle 2,I}^{\scriptscriptstyle(2)} $$ and $$\|\p_\tau^{l+1} \rir(0) \|_{H^2} \leq M \|\cio(0)\|_{H^5} \|\p_\tau^l \rir(0)\|_{H^2}, $$ we complete the proof of Proposition \ref{prop ciir}.
\end{proof}

The next order of initial layer $(\ciis,\pil,\uis)$ satisfy
\begin{equation}\label{equation of initial layer of order 3}
    \quad \left\{\begin{array}{l}
    \p _\tau \ciis = z_iD_i\dive(\cio(0)\grad \pil + \cil(0)\grad \pio),  \quad x\in \Omega ,\tau >0, \\
    -\lap \pil=\ris, \\
    \p_\tau \uis +\grad \preil =0, \\
    \dive \uis = 0.
    \end{array} \right.
\end{equation}
We implement the system (\ref{equation of initial layer of order 3}) with the boundary conditions:
\begin{equation}\label{initial condition of initial layer of order 3}
\begin{split}
     \ris(x',y,\tau=0)=-\ris(x',y,t=0)=\lap \pl(t=0),\\
     (\ciis,\pil,\uis,\preil)(x',y,\tau \rightarrow +\oo) =0 \andf \pil |_{\p \Omega}=0.
\end{split}
\end{equation}

Similarly one has $\uis=\preil=0$ and the following Proposition.
\begin{Proposition}\label{prop cirs}
{\sl Let $\cio(0),\cil(0) \in H^5$ and $\ciis(0) \in H^2.$ Then the system (\ref{equation of initial layer of order 3})-(\ref{initial condition of initial layer of order 3}) admits a unique solution in $C^\oo([0,+\infty); H^2)$.
Moreover, there exist positive constants $M_1$, $M_2$ depending only on $\lambda,z_i$ and $D_i$ for $i=1,2$, such that  for $0 \leq l \leq 3$, we have
\begin{equation*}
\begin{split}
    &\sum_{i=1}^2\|\p_\tau^l \ciis(\tau)\|_{H^2}^2 + \|\p_\tau^l \pil(\tau)\|_{H^4}^2  \\
    &\qquad \qquad \leq M_1 \bigl(\sum_{i=1}^2 \|\cio(0)\|_{H^5}^{12} + \sum_{i=1}^2\|\cil(0)\|_{H^5}^{12} +1\bigr) \sum_{i=1}^2 \|\ciis(0)\|_{H^2}^2 e^{- M_2 \tau}.
\end{split}
\end{equation*}
If we assume in addition that $\pil\equiv 0$ and $\pir(0)=0$, then we have $\pir\equiv 0$.
}
\end{Proposition}

{\Large $\bullet$}{\bf Analysis of mixed layer solutions.}

For mixed boundary layer, we also only deal with the one near $y=0$. Firstly by comparing the leading order, one has $\p_\xi \prelmfr =0 =\prelmfr(x',\xi \rightarrow +\oo) $, which leads to $\prelmfr = 0 $. Thus the next order $(\cilmo,\plmo,\vlmo,\wlmo,\prelmfl)$ satisfy
\begin{equation}\label{equation of left mixed layer of order 0}
    \quad \left\{\begin{array}{l}
    \p _\tau \cilmo  =D_i \p_\xi^2 \cilmo +z_iD_i \p_\xi \cilmo \p_\xi\plmo \\
    \qquad \qquad \qquad +z_iD_i \gamma_i(x',0,0)\p_\xi^2 \plmo , \qquad 0< \xi,\tau < +\oo ,\\
    -\p_\xi^2 \plmo=\rlmo, \\
    \p_\tau \vlmo -\nu \p_\xi^2 \vlmo =0, \\
    \p_\tau \wlmo -\nu \p_\xi^2 \wlmo + \p_\xi \prelmfl=0, \\
    \p_\xi \wlmo= 0.
    \end{array} \right.
\end{equation}
We implement the system (\ref{equation of left mixed layer of order 0}) with the boundary conditions:
\begin{align*}\label{bdry condition of left mixed layer of order 0}
  &(\cilmo,\plmo,\vlmo,\wlmo)(x',\xi=0,\tau) = -(\ciio,\pio,\vio,\wio)(x',y=0,\tau)=0 ,\\
  &(\cilmo,\plmo,\vlmo,\wlmo,\prelmfl)(x',\xi\rightarrow +\oo,\tau) = 0,\\
  &(\cilmo,\plmo,\vlmo,\wlmo,\prelmfl)(x',\xi,\tau \rightarrow +\oo) = 0.
\end{align*}
Thus one has $(\cilmo,\plmo,\vlmo,\wlmo,\prelmfl)=0$. By plugging them into the equation of next order, we obtain
\begin{equation}\label{equation of left mixed layer of order 1}
    \quad \left\{\begin{array}{l}
    \p _\tau \cilml  =D_i \p_\xi^2 \cilml  +z_iD_i \gamma_i(x',0,0)\p_\xi^2 \plml,  \quad 0< \xi,\tau < +\oo ,\\
    -\p_\xi^2 \plml=\rlml ,\\
    \p_\tau \vlml -\nu \p_\xi^2 \vlml =0 ,\\
    \p_\tau \wlml -\nu \p_\xi^2 \wlml + \p_\xi \prelmo=0 ,\\
    \p_\xi \wlml= 0.
    \end{array} \right.
\end{equation}
We implement the system (\ref{equation of left mixed layer of order 1}) with the boundary conditions:
\begin{align*}\label{bdry condition of left mixed layer of order 1}
  &(\cilml,\plml,\vlml,\wlml)(x',\xi=0,\tau) = -(\ciil,\pil,\vil,\wil)(x',y=0,\tau)=0 ,\\
  &(\cilml,\plml,\vlml,\wlml,\prelmo)(x',\xi\rightarrow +\oo,\tau) = 0,\\
  &(\cilml,\plml,\vlml,\wlml,\prelmo)(x',\xi,\tau \rightarrow +\oo) = 0.
\end{align*}
Thus one has $(\cilml,\plml,\vlml,\wlml,\prelmo)=0$. By plugging them into the equation of next order, we obtain
\begin{equation}\label{equation of left mixed layer of order 2}
    \quad \left\{\begin{array}{l}
    \p _\tau \cilmr  =D_i \p_\xi^2 \cilmr  +z_iD_i \gamma_i(x',0,0)\p_\xi^2 \plmr,  \quad 0< \xi,\tau < +\oo, \\
    -\p_\xi^2 \plmr=\rlmr, \\
    \p_\tau \vlmr -\nu \p_\xi^2 \vlmr =0,\\
    \p_\tau \wlmr -\nu \p_\xi^2 \wlmr + \p_\xi \prelml=0, \\
    \p_\xi \wlmr= 0.
    \end{array} \right.
\end{equation}
We implement the system (\ref{equation of left mixed layer of order 2}) with
\begin{equation}\label{bdry condition of left mixed layer of order 2}
\begin{split}
&\cilmr (x',\xi=0,\tau) = -\ciir(x',y=0,\tau),\\
  &(\vlmr,\wlmr)(x',\xi=0,\tau) = -(\vir,\wir)(x',y=0,\tau)=0, \\
  &(\cilmr,\plmr,\p_\xi\plmr,\vlmr,\wlmr,\prelml)(x',\xi\rightarrow +\oo,\tau) = 0,\\
  &(\cilmr,\plmr,\vlmr,\wlmr,\prelml)(x',\xi,\tau \rightarrow +\oo) = 0.
\end{split}
\end{equation}

Similarly for $(\cirmk,\prmk,\vrmk,\wrmk,\prermk)$, one has
\begin{equation*}
    (\cirmo,\prmo,\vrmo,\wrmo,\prermfl)=0,
\end{equation*}
\begin{equation*}
    (\cirml,\prml,\vrml,\wrml,\prermo)=0,
\end{equation*}
and
\begin{equation}\label{equation of right mixed layer of order 2}
    \quad \left\{\begin{array}{l}
    \p _\tau \cirmr  =D_i \p_\eta^2 \cirmr  +z_iD_i \gamma_i(x',1,0)\p_\eta^2 \prmr,  \quad 0< \eta,\tau < +\oo, \\
    -\p_\eta^2 \prmr=\rrmr, \\
    \p_\tau \vrmr -\nu \p_\eta^2 \vrmr =0 ,\\
    \p_\tau \wrmr -\nu \p_\eta^2 \wrmr - \p_\eta \prerml=0, \\
    -\p_\eta \wrmr= 0.
    \end{array} \right.
\end{equation}
with boundary conditions:
\begin{equation}
\begin{split}\label{bdry condition of right mixed layer of order 2}
&\cirmr (x',\eta=0,\tau) = -\ciir(x',y=1,\tau),\\
  &(\vrmr,\wrmr)(x',\eta=0,\tau) = -(\vir,\wir)(x',y=1,\tau)=0, \\
  &(\cirmr,\prmr,\p_\eta\prmr,\vrmr,\wrmr,\prerml)(x',\eta\rightarrow +\oo,\tau) = 0,\\
  &(\cirmr,\prmr,\vrmr,\wrmr,\prerml)(x',\eta,\tau \rightarrow +\oo) = 0.
\end{split}
\end{equation}

The next Proposition is concerned with the time and spatial decay of $\cilmr$,  the same results also hold for $\cirmr$.

\begin{Proposition}\label{S4prop3}
  {\sl  Let $\cio(0) \in H^5$, $\ciir(0) \in H^2$ and $\grad^k_{x'}\cilmr(x',\cdot,0) $ belong to $ H^4(\mathbb{R}^+_\xi, \langle \xi \rangle^4 \, d\xi)$ for any $x' \in \T^{d-1}$ and $0 \leq k \leq 2$, then system (\ref{equation of left mixed layer of order 2})-(\ref{bdry condition of left mixed layer of order 2}) admits a unique solution in $C([0,+\infty); H^4(\mathbb{R}^+_\xi, \langle \xi \rangle^4 \, d\xi) )$. Moreover, there exists $M$ depending only on $\lambda, \Lambda, \ga_i(x), z_i$ and $D_i$ for $i=1,2$, such that for $0 \leq k \leq 2$ and $1\leq l \leq 2$, we have
    \begin{equation}\label{estimate of cilmr}
        \begin{split}
    &\sum_{i=1}^2\|\grad^k_{x'} \cilmr(x',\cdot,\tau)\|^2_{H^4(\mathbb{R}^+_\xi, \langle \xi \rangle^4 \, d\xi)} + \sum_{i=1}^2  \int_0^{+\oo} \|\grad^k_{x'} \p_\xi \cilmr(x',\cdot,\tau)\|^2_{H^4(\mathbb{R}^+_\xi, \langle \xi \rangle^4 \, d\xi)} \,d\tau   \\
    & \quad + \int_0^{+\oo} \|\grad^k_{x'} \rlmr(x',\cdot,\tau)\|^2_{H^4(\mathbb{R}^+_\xi, \langle \xi \rangle^4 \, d\xi)} \, d\tau
    \leq
    \sum_{i=1}^2\|\grad^k_{x'} \cilmr(x',\cdot,0)\|^2_{H^4(\mathbb{R}^+_\xi, \langle \xi \rangle^4 \, d\xi)} \\
    &\qquad + M \sum_{i=1}^2\sum_{j=0}^3 \int_0^{+\oo}  |\grad^k_{x'} \p_\tau^j \ciir(x',0,\tau)|^2 \, d\tau  + M \sum_{i=1}^2\sum_{j=0}^2 \sup_\tau |\grad^k_{x'} \p_\tau^j\ciir(x',0,\tau)|^2
        \end{split}
    \end{equation}
and
\begin{equation}\label{rlmr to plmr}
\begin{split}
  &\|\grad^k_{x'}\p_\xi \p_\tau^l \plmr(x',\cdot,\tau)\|_{L^\oo_\xi} +\|\grad^k_{x'}\p_\xi \p_\tau^l \plmr(x',\cdot,\tau)\|_{L^2_\xi} \\
  &\qquad \lesssim \|\grad^k_{x'}\rlmr(x',\cdot,\tau)\|_{H^4(\mathbb{R}^+_\xi, \langle \xi \rangle^4 \, d\xi)} +\sum_{i=1}^2\|\grad^k_{x'}\p_\xi\cilmr(x',\cdot,\tau)\|_{H^4(\mathbb{R}^+_\xi, \langle \xi \rangle^4 \, d\xi)}.\end{split}
\end{equation}}
\end{Proposition}

\begin{proof} Once again, we shall only deal with the \textit{a priori} estimate in weighted Sobolev space $H^4(\mathbb{R}^+_\xi, \langle \xi \rangle^4 \, d\xi)$. For simplicity, we omit $x'$, and denote $\gamma_i(x',0,0)$ by $\gamma_i$ throughout this proof. By introducing
\begin{equation}\label{change of variable}
\quad \left\{\begin{array}{l}
    a_i(\tau) \eqdefa -\ciir(y=0,\tau), \quad r(\tau)\eqdefa z_1a_1+z_2a_2, \\
     \alpha_i \eqdefa \cilmr - a_i(\tau)e^{-\xi}, \quad \omega  \eqdefa \rlmr - r(\tau)e^{-\xi} =z_1\al_1+z_2\al_2,
\end{array}\right.
\end{equation}
we consider an equivalent system:
\begin{equation}\label{equation of alpha_i}
    \p_\tau \alpha_i - D_i \p_\xi^2\alpha_i + z_iD_i\gamma_i\omega = -(\frac{d a_i}{d\tau}-D_ia_i+z_iD_i\gamma_ir)e^{-\xi},\quad 0< \xi,\tau<+\oo,
\end{equation}
implemented with boundary conditions
    $$\alpha_i(0,\tau) = \omega(0,\tau)=0 .$$

Firstly, we get, by taking $L^2(\mathbb{R}^+_\xi)$ inner product of the  equation (\ref{equation of alpha_i}) with $\frac{\al_i}{\ga_iD_i}$,
and then summing up the resulting equalities for $i=1,2$, that
\begin{align*}
     \sum_{i=1}^2 &\frac{1}{2\ga_iD_i}\frac{d}{d\tau} \|\al_i\|_{L^2_\xi}^2+ \sum_{i=1}^2 \frac{1}{\ga_i}\|\p_\xi \al_i\|_{L^2_\xi}^2  + \|\omega\|_{L^2_\xi}^2  \\
    & \leq M\sum_{i=1}^2\bigl( |a_i| + |\frac{d a_i}{d\tau}|\bigr)\int_0^{+\oo} e^{-\xi} \al_i d\xi  \\
    & \leq M\sum_{i=1}^2\bigl( |a_i|^2 + |\frac{d a_i}{d\tau}|^2 \bigr) + \frac{1}{2 \Lambda }\sum_{i=1}^2\|\p_\xi \al_i\|_{L^2_\xi}^2 ,
\end{align*}
where in the last step we used the following  equality:
\begin{equation*}
    \int_0^{+\oo} e^{-\xi} f d\xi = -\int_0^{+\oo} \frac{d}{d\xi}e^{-\xi} f d\xi = \int_0^{+\oo} e^{-\xi} f' d\xi, \quad \textit{for any} \ f\in C^1_\xi\cap L^\oo_\xi \ \textit{with} \ f(0)=0 .
\end{equation*}

As a consequence, we obtain
\begin{equation}\label{estimate of alpha}
\begin{split}
\sum_{i=1}^2 \frac{1}{2\ga_iD_i} &\|\al_i(\tau)\|_{L^2_\xi}^2+ \sum_{i=1}^2 \frac{1}{2\ga_i}\int_0^{+\oo}  \|\p_\xi \al_i\|_{L^2_\xi}^2 \,d\tau + \int_0^{+\oo} \|\omega\|_{L^2_\xi}^2 \, d\tau \\
& \leq \sum_{i=1}^2 \frac{1}{2\ga_iD_i} \|\al_i(0)\|_{L^2_\xi}^2 + M \sum_{i=1}^2\int_0^{+\oo}\bigl(|a_i|^2 + |\frac{d a_i}{d\tau}|^2\bigr) \, d\tau .
\end{split}
\end{equation}

While by taking $L^2(\mathbb{R}^+_\xi)$ inner product  of the  equation (\ref{equation of alpha_i})  with $-\frac{\p_\xi^2 \al_i}{\ga_iD_i}$ and
then summarizing the resulting equalities for $i=1,2$, one has
\begin{align*}
     \sum_{i=1}^2 &\frac{1}{2\ga_iD_i}\frac{d}{d\tau} \|\p_\xi \al_i\|_{L^2_\xi}^2+ \sum_{i=1}^2 \frac{1}{\ga_i}\|\p_\xi^2 \al_i\|_{L^2_\xi}^2  + \|\p_\xi \omega\|_{L^2_\xi}^2  \\
    & \leq M\sum_{i=1}^2\bigl( |a_i|^2 + |\frac{d a_i}{d\tau}|^2\bigr) + \frac{1}{2 \Lambda }\sum_{i=1}^2\|\p_\xi^2 \al_i\|_{L^2_\xi}^2 ,
\end{align*}
which implies
\begin{equation}\label{estimate of pxi alpha}
    \begin{split}
\sum_{i=1}^2 \frac{1}{2\ga_iD_i} &\|\p_\xi\al_i(\tau)\|_{L^2_\xi}^2+ \sum_{i=1}^2 \frac{1}{2\ga_i}\int_0^{+\oo}  \|\p_\xi^2 \al_i\|_{L^2_\xi}^2 \,d\tau + \int_0^{+\oo} \|\p_\xi \omega\|_{L^2_\xi}^2 \, d\tau \\
& \leq \sum_{i=1}^2 \frac{1}{2\ga_iD_i} \|\p_\xi \al_i(0)\|_{L^2_\xi}^2 + M \sum_{i=1}^2\int_0^{+\oo}\bigl(|a_i|^2 + |\frac{d a_i}{d\tau}|^2\bigr) \, d\tau .
\end{split}
\end{equation}

Next we consider higher derivative estimate of $\alpha_i.$ We first get, by taking time derivative of (\ref{equation of alpha_i}), that
\begin{equation}\label{equation of ptau alpha_i}
     \p_\tau^2 \alpha_i - D_i \p_\xi^2\p_\tau\alpha_i + z_iD_i\gamma_i\p_\tau\omega = -(\frac{d^2a_i}{d\tau^2}-D_i\frac{da_i}{d\tau}+z_iD_i\gamma_i\frac{dr}{d\tau})e^{-\xi} ,
\end{equation}
implemented with boundary conditions
    $$\p_\tau\alpha_i(0,\tau) = \p_\tau\omega(0,\tau)=0 .$$

Along the same line to the derivation of \eqref{estimate of alpha} and \eqref{estimate of pxi alpha}, one has, for $k=0,1$
\begin{equation}\label{estimate of pxi ptau alpha}
\begin{split}
\sum_{i=1}^2 \frac{1}{2\ga_iD_i} &\|\p_\xi^k \p_\tau \al_i(\tau)\|_{L^2_\xi}^2+ \sum_{i=1}^2 \frac{1}{2\ga_i}\int_0^{+\oo}  \|\p_\xi^{k+1}\p_\tau \al_i\|_{L^2_\xi}^2 \,d\tau + \int_0^{+\oo} \|\p_\xi^k \p_\tau\omega\|_{L^2_\xi}^2 \, d\tau \\
& \leq \sum_{i=1}^2 \frac{1}{2\ga_iD_i} \|\p_\xi^k \p_\tau \al_i(0)\|_{L^2_\xi}^2
+ M \sum_{i=1}^2\int_0^{+\oo}\bigl(|\frac{da_i}{d\tau}|^2 + |\frac{d^2 a_i}{d\tau^2}|^2\bigr)\, d\tau ,
\end{split}
\end{equation}

Now in view of (\ref{change of variable}), (\ref{estimate of alpha}), (\ref{estimate of pxi alpha}), (\ref{estimate of pxi ptau alpha}) and  $$\|\p_\xi^k \p_\tau \al_i(0)\|_{L^2_\xi} \lesssim  \| \al_i(0) \|_{H^{k+2}_\xi} +M(|a_i(0)| + |\frac{da_i}{d\tau}(0) |), $$
we get, for $0\leq k, l \leq 1$,
\begin{equation*}
    \begin{split}
        \sum_{i=1}^2\Bigl( \frac{1}{2\ga_iD_i} &\|\p_\xi^k \p_\tau^l \cilmr(\tau)\|_{L^2_\xi}^2+\frac{1}{2\ga_i}\int_0^{+\oo}  \|\p_\xi^{k+1}\p_\tau^l \cilmr\|_{L^2_\xi}^2 \,d\tau\Bigr)+ \int_0^{+\oo} \|\p_\xi^k \p_\tau^l \rlmr\|_{L^2_\xi}^2 \, d\tau \\
& \leq M\sum_{i=1}^2 \|\cilmr(0)\|_{H^3_\xi}^2 + M \sum_{i=1}^2 \sum_{j=0}^2 \int_0^{+\oo}  |\frac{d^j a_i}{d\tau^j}|^2 \, d\tau + M\sum_{i=1}^2\sum_{j=0}^1 \sup_\tau |\frac{d^ja_i}{d\tau^j}(\tau)|^2.
    \end{split}
\end{equation*}

 On the other hand, by taking $L^2(\mathbb{R}^+_\xi)$ inner product  of
 the equation (\ref{equation of alpha_i})  with $\xi^2 \frac{\p_\tau \al_i}{\gamma_iD_i}$ and integrating by parts, and then summarizing the equalities for $i=1,2$, we find
\begin{equation*}
\begin{split}
 & \sum_{i=1}^2 \frac{1}{2\ga_i}\frac{d}{d\tau}\|\xi\p_\xi \al_i\|_{L^2_\xi}^2 +  \frac{1}{2}\frac{d}{d\tau}  \|\xi\omega\|_{L^2_\xi}^2 +\sum_{i=1}^2  \frac{1}{\ga_iD_i} \|\xi \p_\tau \al_i\|_{L^2_\xi}^2 \\
    & \leq  M\sum_{i=1}^2\bigl(  |a_i|+|\frac{d a_i}{d\tau}|
 \bigr)\sum_{i=1}^2\int_0^{+\oo} e^{-\xi} \xi^2 |\p_\tau \al_i| d\xi +  M\int_0^{+\oo}|\xi\p_\xi \al_i  \p_\tau \al_i | \, d\xi\\
    &\leq  M\sum_{i=1}^2\bigl(  |a_i|^2+|\frac{d a_i}{d\tau}|^2 +
 \|\p_\tau\al_i\|_{L^2_\xi}^2 +\|\p_\xi\al_i\|_{L^2_\xi}^2\bigr) + \sum_{i=1}^2\frac{1}{2\max_i \ga_i D_i}\|\xi \p_\tau\al_i\|_{L^2_\xi}^2  ,
\end{split}
\end{equation*}
which together with (\ref{equation of alpha_i}) implies
\begin{equation}\label{estimate of xi pxi plmr}
    \begin{split}
 \sum_{i=1}^2 \frac{1}{2\ga_i} & \|\xi\p_\xi \al_i(\tau)\|_{L^2_\xi}^2 +  \frac{1}{2}  \|\xi\omega(\tau)\|_{L^2_\xi}^2 +\sum_{i=1}^2  \frac{1}{2\ga_iD_i} \int_0^\oo \|\xi \p_\tau \al_i\|_{L^2_\xi}^2 \, d\tau\\
& \leq  \sum_{i=1}^2 \frac{1}{2\ga_i}\|\xi\p_\xi \al_i(0)\|_{L^2_\xi}^2 +  \frac{1}{2}  \|\xi\omega(0)\|_{L^2_\xi}^2 \\
&\quad + M \sum_{i=1}^2\int_0^{+\oo}  \bigl( |a_i|^2+|\frac{d a_i}{d\tau}|^2 +
 \|\p_\xi \al_i\|_{L^2_\xi}^2 +\|\p_\xi^2\al_i\|_{L^2_\xi}^2 + \|\omega\|_{L^2_\xi}^2\bigr) \, d\tau .
    \end{split}
\end{equation}

Similarly, by taking $L^2(\mathbb{R}^+_\xi)$ inner product  of the equation (\ref{equation of ptau alpha_i}) with $\xi^4 \frac{\p_\tau \al_i}{\gamma_iD_i}$ and integrating by parts, and then summarizing the equalities for $i=1,2$, we obtain
\begin{equation*}
\begin{split}
   & \sum_{i=1}^2  \frac{1}{2\ga_iD_i}  \frac{d}{d\tau} \|\xi^2 \p_\tau \al_i\|_{L^2_\xi}^2+ \sum_{i=1}^2 \frac{1}{\ga_i}\|\xi^2 \p_\xi \p_\tau \al_i\|_{L^2_\xi}^2  +  \|\xi^2 \p_\tau\omega\|_{L^2_\xi}^2  \\
    & \leq  M\sum_{i=1}^2\bigl(  |\frac{d a_i}{d\tau}| + |\frac{d^2 a_i}{d\tau^2}|
 \bigr)\sum_{i=1}^2\int_0^{+\oo} e^{-\xi} \xi^4 |\p_\tau \al_i| d\xi +  M\int_0^{+\oo}|\xi^2\p_\xi\p_\tau\al_i \xi\p_\tau \al_i | \, d\xi\\
    &\leq  M\sum_{i=1}^2\bigl(  |\frac{d a_i}{d\tau}|^2 + |\frac{d^2 a_i}{d\tau^2}|^2+
 \|\xi \p_\tau\al_i\|_{L^2_\xi}^2 \bigr) + \frac{1}{2\Lambda}\sum_{i=1}^2\|\xi^2 \p_\xi \p_\tau\al_i\|_{L^2_\xi}^2  .
\end{split}
\end{equation*}

Thus in view of (\ref{estimate of alpha}), (\ref{estimate of pxi alpha}) and (\ref{estimate of xi pxi plmr}), to sum up, we get
\begin{equation*}\label{estimate of xi^2 ptau alpha}
\begin{split}
\sum_{i=1}^2& \frac{1}{2\ga_iD_i} \|\xi^2 \p_\tau \al_i(\tau)\|_{L^2_\xi}^2+ \sum_{i=1}^2 \frac{1}{2\ga_i}\int_0^{+\oo}  \|\xi^2 \p_\xi\p_\tau \al_i\|_{L^2_\xi}^2 \,d\tau + \int_0^{+\oo} \|\xi^2 \p_\tau\omega\|_{L^2_\xi}^2 \, d\tau \\
& \leq \sum_{i=1}^2 \frac{1}{2\ga_iD_i} \|\xi^2 \p_\tau \al_i(0)\|_{L^2_\xi}^2 + M \sum_{i=1}^2\int_0^{+\oo}\bigl(|\frac{da_i}{d\tau}|^2 + |\frac{d^2 a_i}{d\tau^2}|^2 +\|\xi \p_\tau \alpha_i \|_{L^2_\xi}^2\bigr) \, d\tau \\
&\leq M \sum_{i=1}^2 \|\al_i(0)\|^2_{H^2(\mathbb{R}^+_\xi, \langle \xi \rangle^4 \, d\xi)}  +M \sum_{i=1}^2\sum_{j=0}^2\int_0^{+\oo}  |\frac{d^ja_i}{d\tau^j}|^2  \, d\tau + M\sum_{i=1}^2\sum_{j=0}^1 \sup_\tau |\frac{d^ja_i}{d\tau^j}(\tau)|^2.
\end{split}
\end{equation*}

Along the same line, one has the estimate of derivative of the highest order
\begin{equation*}\label{estimate of xi^2 ptau^2 alpha}
\begin{split}
\sum_{i=1}^2& \frac{1}{2\ga_iD_i} \|\xi^2 \p_\tau^2 \al_i(\tau)\|_{L^2_\xi}^2+ \sum_{i=1}^2 \frac{1}{2\ga_i}\int_0^{+\oo}  \|\xi^2 \p_\xi\p_\tau^2 \al_i\|_{L^2_\xi}^2 \,d\tau + \int_0^{+\oo} \|\xi^2 \p_\tau^2 \omega\|_{L^2_\xi}^2 \, d\tau \\
&\leq M \sum_{i=1}^2 \|\al_i(0)\|^2_{H^4(\mathbb{R}^+_\xi, \langle \xi \rangle^4 \, d\xi)}  +M \sum_{i=1}^2\sum_{j=0}^3 \int_0^{+\oo}  |\frac{d^ja_i}{d\tau^j}|^2 \, d\tau + M\sum_{i=1}^2\sum_{j=0}^2 \sup_\tau |\frac{d^ja_i}{d\tau^j}(\tau)|^2.
\end{split}
\end{equation*}

By taking derivatives in direction $x'$, we are able to finish the proof of $(\ref{estimate of cilmr})$ by the same argument. Lastly, (\ref{rlmr to plmr}) is a direct consequence of the equation of $\p_\tau^l \cilmr$ and the following inequality
\begin{align*}
    &|\grad^k_{x'}\p_\xi \p_\tau^l \plmr| =|\int_\xi^{+\oo} \langle z\rangle ^{-2} \langle z\rangle^2 \grad^k_{x'}\p_\tau^l \rlmr dz| \lesssim  \langle \xi\rangle^{-\frac{3}{2}} \|\langle \xi \rangle^2 \grad^k_{x'}\p_\tau^l\rlmr\|_{L^2_\xi} .
\end{align*}
This competes the proof of Proposition \ref{S4prop3}.
\end{proof}

\renewcommand{\theequation}{\thesection.\arabic{equation}}
\setcounter{equation}{0}
\section{The Convergence in $L^\oo_T(H^2)$}\label{convergence of Loo}

In this section, we focus on estimating the difference between the solutions of \eqref{S1eq1} and the expansions in $L^\oo_T(H^2)$. To this end, we
shall construct approximate solutions ($c_i^{app},\Phi^{app},u^{app}$)  with the same boundary conditions as $(\cei,\Phi^\ve,\ue)$, and denote
$(\cisum,\psum,\usum) \eqdefa (\cei-c_i^{app},\Phi^\ve-\Phi^{app},\ue-u^{app})$ so that
\begin{equation}\label{S5eq1}
    \quad \left\{\begin{array}{l}
         \cisum(x',y,t) = \cei(x',y,t) - \cio(x',y,t)  -\ve\cil(x',y,t)- \ve^2 \cir(x',y,t)  \\
         \qquad \qquad \qquad- \ve^2 \ciir(x',y,\tau) - \ve^2 f(y)\cilbr(x',\xi,t) - \ve^2g(y)\cirbr(x',\eta,t) \\
         \qquad \qquad \qquad - \ve^2 f(y)\cilmr(x',\xi,t) - \ve^2g(y)\cirmr(x',\eta,t), \\
         \psum(x',y,t)= \Phi^\ve(x',y,t) -\po(x',y,t) -\ve\pl(x',y,t) - \ve^2 \pr(x',y,t)\\
         \qquad \qquad \qquad -\pio(x',y,\tau) - \ve\pil(x',y',\tau) \\
         \qquad \qquad \qquad -\ve^2f(y)\plbr(x',\xi,t)-\ve^2 g(y)\prbr(x',\eta,t), \\
         \usum(x',y,t)= \ue(x',y,t)-\uo(x',y,t) - \ve\ul(x',y,t) ,
    \end{array}
    \right.
\end{equation}
implemented with the boundary conditions:
\begin{equation}\label{bdry condition of eqsum}
    \cisum |_{\p \Omega}=0, \quad \psum |_{\p \Omega}=0 \andf  \usum |_{\p \Omega}=0.
\end{equation}

Due to (\ref{bdry condition of eqsum}), we have $\|\cisum\|_{L^2} \lesssim \|\grad \cisum\|_{L^2} \lesssim \|\lap \cisum\|_{L^2}$, and similar results hold for $\psum$ and $\usum$. We shall first handle the estimates of $\|\lap \psum\|_{L^\oo_T(L^2)}$  and $\|\pt \cisum\|_{L^\oo_T(L^2)}$, and then use the $c_i^S$ equation to derive the estimate of $\|\lap \cisum\|_{L^\oo_T(L^2)}$.

\subsection{\bf Estimate of $\|(\grad\cisum, \ve\lap \psum,\grad \usum)\|_{L^\oo_T(L^2)}$.}

The goal of this subsection is to prove the following type of inequality:
\begin{align}\label{goal estimate in Loo H1}
    \frac{d}{dt}\bigl(\int_\Omega \frac{|\grad \cisum|^2}{\cio} \,dx + \ve^2\|\lap \psum\|^2_{L^2} + \|\usum\|^2_V \bigr)
     + \int_\Omega \bigl(\frac{|\lap \cisum|^2}{\cio} + \cio|\lap\psum|^2\bigr) \,dx \\
    + \ve^{-2} \|\grad \rsum\|^2_{L^2} +\nu \|A \usum\|_{L^2}^2    \leq O(\ve^3)  \nonumber.
\end{align}
To this end, we write the equations of $(\cisum, \psum, \usum)$ as
\begin{subequations}\label{S5eq8}
\begin{gather}
\label{equation of cisum}
    \pt \cisum +\usum\cdot \grad \cisum = D_i\dive\bigl( \cio(\frac{\grad\cisum}{\cio}+ z_i\grad \psum)\bigr)+z_iD_i \dive(\cisum\grad \psum) +K_i+L_i+M_i,\\
    \rsum = -\ve^2 \lap \psum + N, \label{5.4b}\\
\label{equation of usum}
    \pt \usum +\usum\cdot \grad \usum - \nu\lap\usum - \grad p^S = -\rsum\grad\psum +O+P,\\
    \dive\usum = 0,
\end{gather}
\end{subequations}
where
\begin{align*}
    K_i &\eqdefa -(\uo+\ve\ul)\cdot \grad \cisum  \nonumber\\
    &- \usum\cdot \grad (\cio+\ve\cil+\ve^2 \cir+\ve^2\ciir + \ve^2 f \cilbr +\ve^2 g\cirbr+ \ve^2 f \cilmr +\ve^2 g\cirmr) \nonumber \\
     & +z_iD_i \dive\bigl(\cisum\grad (\po+\ve\pl+\ve^2 \pr+ \pio+\ve\pil+\ve^2f\plbr+\ve^2g\prbr)\bigr)   \nonumber\\
    &+z_iD_i \dive\bigl((\ve\cil+\ve^2 \cir+\ve^2\ciir+ \ve^2 f \cilbr +\ve^2 g\cirbr+ \ve^2 f \cilmr +\ve^2 g\cirmr)\grad \psum\bigr) \nonumber ,
\end{align*}
\begin{align*}
    L_i &\eqdefa z_iD_i\bigl(\cio-\gamma_i(x',0,t)\bigr)f\p_\xi^2\plbr +z_iD_i\bigl(\cio-\gamma_i(x',1,t)\bigr)f\p_\eta^2\prbr \nonumber \\
    &\quad +z_iD_i\dive\bigl((\cio-\cio(0))\grad\pio\bigr),
\end{align*}
\begin{align*}
   & M_i \eqdefa -\ve^3\ul\cdot\grad\cir + \ve^2 \ur\cdot \grad \cio  \\
    & -(\uo+\ve\ul)\cdot \grad( \ve^2\ciir + \ve^2 f \cilbr +\ve^2 g\cirbr+ \ve^2 f \cilmr +\ve^2 g\cirmr)  \\
    & + \ve^2D_i\lap\ciir +\ve^2 D_i (f \lap_{x'} \cilbr + g \lap_{x'}\cirbr +f \lap_{x'} \cilmr + g \lap_{x'}\cirmr )  \\
    &+\ve^2 D_i ( f'' \cilbr +g''\cirbr+f'' \cilmr +g''\cirmr)  \\
    &+2\ve D_i ( f' \p_\xi \cilbr -  g'\p_\eta \cirbr+f' \p_\xi \cilmr -  g'\p_\eta \cirmr)  \\
    &+\ve^2 z_iD_i\grad \cio\grad (f\plbr+g\prbr)  +\ve^2z_iD_i\cio (f''\plbr+g''\prbr+f \lap_{x'} \plbr + g \lap_{x'}\prbr)  \\
    &+2\ve z_iD_i\cio(f' \p_\xi \plbr - g'\p_\eta \prbr) \\
     &+ \ve^3 z_iD_i\dive(\cil\grad \pr +\cir\grad \pl) +\ve^4 z_iD_i\dive(\cir\grad\pr) + \ve z_iD_i\dive(\cio\grad\pil) \\
     &+z_iD_i \dive\bigl((\ve\cil+\ve^2 \cir+\ve^2\ciir)\grad (\pio+\ve\pil+\ve^2f\plbr+\ve^2g\prbr)\bigr) \\
    &+z_iD_i \dive\bigl((\ve^2 f \cilbr +\ve^2 g\cirbr+\ve^2 f \cilmr +\ve^2 g\cirmr)\grad (\pio+\ve\pil+\ve^2f\plbr+\ve^2g\prbr)\bigr) \\
    &+z_iD_i \dive\bigl((\ve^2\ciir+\ve^2 f \cilbr +\ve^2 g\cirbr+\ve^2 f \cilmr +\ve^2 g\cirmr)\grad (\po+\ve\pl+\ve^2 \pr)\bigr) \\
    & -\ve^2 f \pt\cilbr -\ve^2 g \pt\cirbr-z_iD_i\gamma_i(x',0,0)f\p_\xi^2\plmr-z_iD_i\gamma_i(x',1,0)g\p_\eta^2\prmr ,
\end{align*}
\begin{align*}
    N \eqdefa& -\ve^3 \lap \pl -\ve^4 \lap \pr -\ve^3 \lap \pil - 2\ve^3(f'\p_\xi\plbr - g'\p_\eta \prbr)\\
    &-\ve^4(f\lap_{x'}\plbr+g\lap_{x'}\prbr +f''\plbr+g''\prbr )+\ve^2(f\p_\xi \plmr+g\p_\eta \prmr)  ,
\end{align*}
\begin{align}
   O \eqdefa& -\usum\cdot \grad (\uo+\ve\ul)-(\uo+\ve\ul)\cdot \grad \usum \nonumber \\
    & - \ve^2(\rr+\rir+f\rlbr+g\rrbr+f\rlmr+g\rrmr)\grad\psum \nonumber \\
    &- \rsum \grad (\po+\ve\pl+\ve^2 \pr+ \pio+\ve\pil+\ve^2f\plbr+\ve^2g\prbr) \nonumber,
\end{align}
and
\begin{align}
    P \eqdefa & -\ve^2\ul\cdot\grad\ul-\ve^2(\rr+\rir+f\rlbr+g\rrbr+f\rlmr+g\rrmr)\nonumber \\
    & \qquad \cdot \grad (\po+\ve\pl+\ve^2 \pr+ \pio+\ve\pil+\ve^2f\plbr+\ve^2g\prbr) \nonumber .
\end{align}

\begin{thm}\label{thm4}
    {\sl
Let $d=2$ or $3$,  we assume that the initial data is ``well-prepared" i.e. $\re(0)=0$, $\lam \leq \cio(0) $, $\bigl(\cio(0), \cil(0),\cir(0),\uo(0),\ul(0),\ur(0)\bigr) \in  H^5$, $\bigl(\ciir(0), \ciis(0)\bigr) \in H^2$ and $\grad^k_{x'}\cilmr(x',\cdot,0)$ belong to  $H^4(\mathbb{R}^+_\xi, \langle \xi \rangle^4 \, d\xi)$ for any $x'\in \mathbb{T}^{d-1}$ and $0 \leq k \leq 2$. Let $(\cei, \ue,\pe )\in C([0,T];H^5)$ be the strong solution of system (\ref{S1eq1})-(\ref{ebdry}) with initial data $\bigl(\cei(0),\ue(0)\bigr) $, where $T \leq T_0$ for $T_0$ 
 being determined by Proposition \ref{prop ciir}. If there exists a positive constant $M_1>0 $ so that
\begin{align}\label{initial condition of thm4}
     &\|\grad \cisum(0)\|_{L^2} + \|\grad \usum(0)\|_ {L^2} +\ve\|\lap \psum(0)\|_ {L^2} \leq M_1 \ve^\frac{3}{2} .
\end{align}
Then for $\ve$ being sufficiently small, there exists a positive constant $M_2$ depending only on initial data, $\nu,M_1, T,\lam, \Lam, W(x), \ga_i(x), z_i$ and $D_i$ for $i=1,2$, so that
\begin{subequations}\label{estimate of lap psum}
\begin{gather}
\label{estimate of lap psum1}
     \|\grad \cisum\|_{L^\oo_T(L^2)} + \|\grad \usum\|_ {L^\oo_T(L^2)} + \ve \|\lap \psum\|_ {L^\oo_T(L^2)}  \leq M_2 \ve^\frac{3}{2} ,\\
     \|\lap  \cisum\|_{L^2_T(L^2)} + \|A \usum\|_ {L^2_T(L^2)}  + \|\lap \psum\|_ {L^2_T(L^2)}+\ve^{-1} \|\grad \rsum\|_ {L^2_T(L^2)}  \leq M_2 \ve^\frac{3}{2}. \label{estimate of lap psum2}
\end{gather}
\end{subequations}}
\end{thm}

\begin{lem}\label{estimate of K_i-P_i}
    Under the assumptions in Theorem \ref{thm4}, for $\ve \leq 1$, one has
    \beq\label{S5eq9}
    \begin{split}
       & \|K_i\|_{L^2 } +\|O_i\|_{L^2}\lesssim  \Bigl(\int_\Omega \frac{|\grad \cisum|^2}{\cio} \,dx\Bigr)^\frac{1}{2} +\|\usum\|_V + \ve \|\lap \psum\|_{L^2} , \\
       &\|L_i\|_{L^\oo(0,T;L^2)} +\|M_i\|_{L^\oo(0,T;L^2)}  \lesssim \ve^\frac{1}{2} , \\
       &\|L_i\|_{L^2(0,T;L^2)} +\|M_i\|_{L^2(0,T;L^2)} \lesssim \ve^{\frac{3}{2}} , \\
       &\|N\|_{L^2(0,T;L^2)} \lesssim \ve^3,\quad \|\grad N\|_{L^2(0,T;L^2)} \lesssim \ve^\frac{5}{2}, \quad \|\pt N\|_{L^2(0,T;L^2)} \lesssim \ve^\frac{3}{2},\\
       & \|P\|_{L^\oo(0,T;L^2)} +\|P\|_{L^2(0,T;L^2)} \lesssim \ve^2  .\\.
    \end{split}\eeq
\end{lem}

The proof of Lemma \ref{estimate of K_i-P_i} involves tedious calculation, we
 leave it in the Appendix \ref{appA}.

 With Lemma \ref{estimate of K_i-P_i}, we are able to refine the convergence rate of $\|(\grad\cisum, \ve\lap \psum,\grad \usum)\|_{L^\oo_T(L^2)}$, which will be crucial for the proof of Theorem \ref{thm4}. In order to do so, let us define
 \begin{equation*}
 T^\star\eqdefa\sup\bigl\{\ t\leq T; \quad \|\grad \cisum\|_{L^\infty_t(L^2)}\leq 1,
 \ \|\Delta\Phi^S\|_{L^\infty_t(L^2)}\leq 1, \
 \|\grad \usum\|_{L^\infty_t(L^2)}\leq 1 \ \bigr\}.
 \end{equation*}

\begin{lem}\label{refined convergence rate of L2}
    Under the assumptions in Theorem \ref{thm4}, for $\ve$ being sufficiently small, there exists a constant $M$ depending only on initial data, $\nu,M_1, T,\lam, \Lam, z_i$ and $D_i$  for $i=1,2$, so that
\begin{subequations}
    \begin{gather}\label{S5eq10}
         \|\cisum\|_{L^\oo_{T^\star}(L^2)}^2 +\ve^2 \|\grad \psum\|_{L^\oo_{T^\star}(L^2)}^2 + \|\usum\|_{L^\oo_{T^\star}(L^2)}^2 \leq M \ve^3, \\
         \|\grad \cisum\|_{L^2_{T^\star}(L^2)}^2  + \|\grad \psum\|_{L^2_{T^\star}(L^2)}^2 + \ve^2 \|\lap \psum\|_{L^2_{T^\star}(L^2)}^2+ \|\grad \usum\|_{L^2_{T^\star}(L^2)}^2 \leq M \ve^3.
    \end{gather}
\end{subequations}
\end{lem}
\begin{proof} Thanks to \eqref{5.4b}, we get, by taking $L^2$ inner product of $\pt\cisum + \usum\cdot \grad \cisum$ with $\frac{\cisum}{\cio}+ z_i\psum$  and  using integration by parts and then summarizing the resulting equalities for $i=1,2$, that
\begin{align*}
     &\sum_{i=1}^2  \int_{\Omega} (\pt\cisum + \usum\cdot \grad \cisum)(\frac{\cisum}{\cio}+ z_i\psum) \,dx\\
     &\quad=\frac{1}{2}\frac{d}{dt} \int_\Omega \frac{(\cisum)^2}{\cio}\, dx - \frac{1}{2}\int_\Omega (\cisum)^2\frac{d}{dt}(\frac{1}{\cio})\, dx + \int_\Omega \frac{\usum}{\cio} \grad\cisum \cisum\, dx \\
     &\qquad + \frac{\ve^2}{2}\frac{d}{dt}\|\grad \psum\|_{L^2}^2 +  \int_{\Omega} \pt N \psum \,dx - \int_{\Omega} \usum \cdot \grad \psum \rsum \,dx \\
     & \quad \geq \frac{1}{2}\frac{d}{dt} \int_\Omega \frac{(\cisum)^2}{\cio}\, dx+ \frac{\ve^2}{2}\frac{d}{dt}\|\grad \psum\|_{L^2}^2 - \frac{D^*}{16} \sum_{i=1}^2  \Bigl( \int_{\Omega} \frac{|\grad \cisum|^2}{\cio} \,dx + z_i^2 \int_{\Omega} \cio|\grad \psum|^2 \,dx \Bigr) \\
     & \qquad -\frac{\nu}{2}\|u\|_V^2- M\|\pt \cio\|_{L^\oo} \int_\Omega \frac{(\cisum)^2}{\cio}\, dx -M\|\grad \cisum \|_{L^2}^4\|\usum\|_{L^2}^2\\
     &\qquad  - \int_{\Omega} \usum \cdot \grad \psum \rsum \,dx -M  \|\pt N\|_{L^2}^2 .
 \end{align*}

Observing that $\Lambda \geq\cio \geq \lambda >0$, we find
   \begin{align*}
        &\sum_{i=1}^2 D_i\int_{\Omega} \dive\bigl(\cio(\frac{\grad\cisum}{\cio}+z_i\grad\psum)\bigr) (\frac{\cisum}{\cio} + z_i \psum) \, dx \\
        &\quad =- \sum_{i=1}^2 D_i\int_{\Omega} \cio(\frac{\grad\cisum}{\cio}+z_i\grad\psum) (\frac{\grad \cisum}{\cio} + z_i \grad \psum + \cisum\grad(\frac{1}{\cio})) \, dx \\
        &\quad \leq  - \frac{D^*}{2} \sum_{i=1}^2  \Bigl( \int_{\Omega} \frac{|\grad \cisum|^2}{\cio} \,dx + z_i^2 \int_{\Omega} \cio|\grad \psum|^2 \,dx \Bigr) - \ve^2D^*  \|\lap \psum\|_{L^2}^2 \\
        & \qquad + M\|\grad\cio\|_{L^\oo}^2 \int_\Omega \frac{(\cisum)^2}{\cio}\, dx + \ve^{-2}\|N\|_{L^2}^2 .
    \end{align*}

Next we estimate the quadratic term. Indeed  by using integration by parts and Sobolev's embedding, one has
\begin{equation*}
    \begin{split}
        &\sum_{i=1}^2 D_i\int_{\Omega} \dive(\cisum\grad \psum) (\frac{\cisum}{\cio} + z_i \psum) \, dx \\
         &\quad \leq M(\|\grad\cio\|_{L^\oo}+ 1) \|\cisum\|_{L^2}^\frac{1}{2} \|\grad\cisum\|_{L^2}^\frac{1}{2}\|\lap\psum\|_{L^2} (\|\grad \cisum\|_{L^2} + \|\grad \psum\|_{L^2}) \\
        &\quad \leq M(\|\grad\cio\|_{L^\oo}^4+ 1)\|\lap\psum\|_{L^2}^4\|\cisum\|_{L^2}^2  \\
        &\qquad  + \frac{D^*}{16} \sum_{i=1}^2  \Bigl( \int_{\Omega} \frac{|\grad \cisum|^2}{\cio} \,dx + z_i^2 \int_{\Omega} \cio|\grad \psum|^2 \,dx \Bigr) .
    \end{split}
\end{equation*}

By using integration by parts and H\"older's inequality, one has
\begin{align*}
    &\sum_{i=1}^2 \int_{\Omega} K_i(\frac{\cisum}{\cio}+ z_i \psum) \, dx \leq \frac{D^*}{32} \sum_{i=1}^2  \Bigl( \int_{\Omega} \frac{|\grad \cisum|^2}{\cio} \,dx + z_i^2 \int_{\Omega} \cio|\grad \psum|^2 \,dx \Bigr) \\
      &\quad +  M\sum_{i=1}^2 (1+\|\grad\cio\|_{L^\oo}^2)\Bigl(\|\uo+\ve\ul\|_{L^\oo}^2  \int_\Omega \frac{(\cisum)^2}{\cio}\, dx  \\
      &\quad  +\|\cio+\ve\cil+\ve^2 \cir+\ve^2\ciir + \ve^2 f \cilbr +\ve^2 g\cirbr+ \ve^2 f \cilmr +\ve^2 g\cirmr\|^2_{L^\oo} \|\usum\|_{L^2}^2 \\
       & \quad  +\|\grad (\po+\ve\pl+\ve^2 \pr+ \pio+\ve\pil+\ve^2f\plbr+\ve^2g\prbr)\|_{L^\oo}^2   \int_\Omega \frac{(\cisum)^2}{\cio}\, dx  \\
    &\quad +  \ve^2 \|\cil+\ve \cir+\ve\ciir + \ve f \cilbr +\ve g\cirbr+ \ve f \cilmr +\ve g\cirmr\|_{L^\oo}^2 \|\grad \psum\|_{L^2}^2   \Bigr).
\end{align*}

The estimates of rest terms are straightforward
\begin{align*}
    \sum_{i=1}^2 \int_{\Omega} (L_i +M_i)(\frac{\cisum}{\cio}+z_i \psum) \,dx &\leq \frac{D^*}{32} \sum_{i=1}^2  \Bigl( \int_{\Omega} \frac{|\grad \cisum|^2}{\cio} \,dx + z_i^2 \int_{\Omega} \cio|\grad \psum|^2 \,dx \Bigr) \\
    &\quad +M\sum_{i=1}^2(\|L_i\|_{L^2}^2 + \|M_i\|_{L^2}^2).
\end{align*}

Whereas by multiplying (\ref{equation of usum}) by $\usum$ and using integration by parts, we get
\begin{equation*}
    \begin{split}
        \frac{1}{2}&\frac{d}{dt} \|\usum\|_{L^2}^2 +\nu \|\grad \usum\|_{L^2}^2 \\
        &\leq  \bigl(1+\|\grad(\uo+\ve\ul)\|_{L^\oo}+ \ve^4\|\rr+\rir+f\rlbr+g\rrbr+f\rlmr+g\rrmr\|_{L^\oo}^2 \\
        &\qquad +\|(\po+\ve\pl+\ve^2 \pr+ \pio+\ve\pil+\ve^2f\plbr+\ve^2g\prbr)\|_{L^\oo}^2 \bigr)\|\usum\|_{L^2}^2\\
        &\quad  - \int_{\Omega} \usum \cdot \grad \psum \rsum \,dx  +  \|P\|_{L^2}^2 + \frac{D^*}{16} \sum_{i=1}^2  \bigl( \int_{\Omega} \frac{|\grad \cisum|^2}{\cio} \,dx + z_i^2 \int_{\Omega} \cio|\grad \psum|^2 \,dx \bigr).\\
    \end{split}
\end{equation*}

By summing up all the above inequalities and using (\ref{result 1 in thm 1}-\ref{result 3 in thm 1}) and  the estimates of the layer solutions obtained  in Section \ref{analysis of layers}, we achieve
\begin{equation}\label{S5eq11}
\begin{split}
   \frac{1}{2} &\sum_{i=1}^2 \frac{d}{dt}\int_\Omega \frac{(\cisum)^2}{\cio}\,dx + \frac{\ve^2}{2} \frac{d}{dt}\|\grad \psum\|^2_{L^2}  +\frac{1}{2}\frac{d}{dt}\|\usum\|_{L^2}^2 +\frac{\nu}{2} \|\grad \usum\|_{L^2}^2 +\frac{D^*\ve^2}{2}\|\lap \psum\|^2_{L^2}\\
    &+ \frac{D^*}{4}   \sum_{i=1}^2\bigl(\int_\Omega \frac{|\grad \cisum |^2}{\cio}\,dx+z^2_i \int_\Omega \cio|\grad \psum|^2 \,dx\bigr)  +D^*\ve^2\|\lap \psum\|^2_{L^2} \\
    &\leq \kappa_1(t)\Bigl(\int_\Omega \frac{(\cisum)^2}{\cio}\,dx+ \ve^2\|\grad \psum\|_{L^2}^2  + \|\usum\|_{L^2}^2 \Bigr) +\ka_2(t) \qquad  \text{on} \ [0,T^\star],
\end{split}
\end{equation} provided that
$\ve<1$ is so  small that $M_1 \ve^\frac{1}{2} \leq \frac{1}{2} $ and consequently $0< T^\star$,
where $\ka_1$ and $\ka_2$ are positive functions satisfying: for some $M$ depending only on initial data, $\nu,M_1, T,\lam, \Lam,W(x),$ $\ga_i(x), z_i$ and $D_i$  for $i=1,2$, there hold
\begin{equation*}
    \int_0^{T^\star} \kappa_1(t')\, dt' \leq M \andf \int_0^{T^\star} \kappa_2(t') \, dt' \leq M\ve^3.
\end{equation*}

By using Gronwall's inequality to \eqref{S5eq11} with initial conditions (\ref{initial condition of thm4}), we complete the proof of Lemma \ref{refined convergence rate of L2}.
\end{proof}

\begin{prop}\label{prop5.1}
  {\sl   Under the assumptions in Theorem \ref{thm4}, for $\ve$ being sufficiently small and for any $t \in [0,T^\star],$ one has
    \begin{equation}\label{estimate of thm4}
    \begin{split}
    \frac{d}{dt}&\Bigl(\frac{1}{2} \sum_{i=1}^2 \int_\Omega \frac{|\grad\cisum|^2}{\cio}\,dx + \frac{\ve^2}{2} \|\lap \psum\|^2_{L^2}  +\frac{1}{2}\|\usum\|_{V}^2\Bigr) +\frac{\nu}{2} \|A\usum\|_{L^2}^2\\
    &+ \frac{D^*}{4}   \sum_{i=1}^2\bigl(\int_\Omega \frac{|\lap \cisum |^2}{\cio}\,dx+z^2_i \int_\Omega \cio|\lap \psum|^2 \,dx\bigr) +\frac{D^*}{\ve^2}\|\grad \rsum\|^2_{L^2}    \\
    &\quad \leq \kappa_3(t) \bigl(\sum_{i=1}^2\int_\Omega \frac{|\grad\cisum|^2}{\cio}\,dx + \ve^2\|\lap \psum\|^2_{L^2} + \|\usum\|_{V}^2 \bigr) +\kappa_4(t) ,
\end{split}\end{equation}
where $\kappa_3$ and $\kappa_{4}$ are positive functions satisfying: for some $M$ depending only on initial data, $\nu,M_1, T,\lam, \Lam, W(x), \ga_i(x), z_i$ and $D_i$ for $i=1,2$, there hold
\begin{equation*}
    \int_0^{T^\star} \kappa_3(t') \,dt' \leq M \andf \int_0^{T^\star} \kappa_{4}(t') \, dt' \leq M\ve^3 .
\end{equation*}}
\end{prop}

\begin{proof} The main idea of the proof is to take $L^2$ inner product of
 the equation (\ref{equation of cisum}) with  $-\dive(\frac{\grad \cisum}{\cio}+z_i\grad\psum).$  Indeed we
first get, by using integration by parts, that
    \begin{align*}
       - &\sum_{i=1}^2\int_\Omega \pt\cisum \dive(\frac{\grad \cisum}{\cio}+z_i\grad \psum) \, dx \\
        &=\frac{1}{2}\frac{d}{dt}\sum_{i=1}^2\int_\Omega \frac{|\grad \cisum|^2}{\cio} \, dx-\sum_{i=1}^2 \frac{1}{2}\int_\Omega |\grad\cisum|^2 \pt \Bigl(\frac{1}{\cio}\Bigr) \, dx \ -\int_\Omega \pt\rsum \lap\psum \, dx\\
        & \geq \frac{1}{2} \frac{d}{dt}\sum_{i=1}^2\int_\Omega \frac{|\grad \cisum|^2}{\cio} \, dx  + \frac{\ve^2}{2} \frac{d}{dt}\|\lap \psum\|_{L^2}^2  -M \sum_{i=1}^2\|\pt \cio\|_{L^\oo} \int_\Omega \frac{|\grad \cisum|^2}{\cio} \, dt \\
        &\quad -\frac{D^*}{16} \sum_{i=1}^2 \bigl(\int_\Omega \frac{|\lap \cisum |^2}{\cio}\,dx+z^2_i \int_\Omega \cio|\lap \psum|^2 \,dx \bigr)   - M\|\pt N\|^2_{L^2},
    \end{align*}

and
\begin{align*}
        - &\sum_{i=1}^2\int_\Omega \usum\cdot\grad \cisum \dive(\frac{\grad \cisum}{\cio}+z_i\grad \psum) \, dx \\
        &\geq -M\sum_{i=1}^2 (1+\|\grad\cio\|_{L^\oo}^4) \|\grad \cisum\|_{L^2}^2\|\usum\|_V^4 - \frac{D^*}{16}  \sum_{i=1}^2\bigl(\int_\Omega \frac{|\lap \cisum |^2}{\cio}\,dx+z^2_i \int_\Omega \cio|\lap \psum|^2 \,dx \bigr) .
    \end{align*}

While for the dissipation term, we find
    \begin{align*}
         -& \sum_{i=1}^2\int_\Omega D_i \dive\bigl(\cio (\frac{\grad \cisum}{\cio}+z_i\grad \psum)\bigr) \dive(\frac{\grad \cisum}{\cio}+z_i\grad \psum) \, dx \\
         &=  - \sum_{i=1}^2\int_\Omega D_i\cio \bigl(\frac{\lap\cisum}{\cio} + z_i\frac{\grad \cio}{\cio}\grad \psum +z_i\lap\psum \bigr)\bigl(\frac{\lap\cisum}{\cio} + \grad\cisum \grad (\frac{1}{\cio}) + z_i \lap\psum\bigr)  \, dx \\
         &\leq  -\frac{D^*}{2} \sum_{i=1}^2\int_\Omega \cio|\frac{\lap\cisum}{\cio}+z_i\lap\psum|^2 \, dx  +M\sum_{i=1}^2 (\|\grad \cio\|_{L^\oo}^2+1)\int_\Omega \frac{|\grad \cisum|^2}{\cio} \, dx \\
         &\quad + M\sum_{i=1}^2(\|\grad \cio\|_{L^\oo}^2+1) \|\grad \psum\|_{L^2}^2 \\
         &= -\frac{D^*}{2}  \sum_{i=1}^2\bigl(\int_\Omega \frac{|\lap \cisum |^2}{\cio}\,dx+z^2_i \int_\Omega \cio|\lap \psum|^2 \,dx \bigr) + \frac{D^*}{\ve^2} \|\grad \rsum\|_{L^2}^2  \\
         & \quad +M\sum_{i=1}^2(\|\grad \cio\|_{L^\oo}^2+1)\int_\Omega \frac{|\grad \cisum|^2}{\cio} \, dx \\
         &\quad+ M\sum_{i=1}^2(\|\grad \cio\|_{L^\oo}^2+1)\|\grad \psum\|_{L^2}^2+ M\ve^{-2} \|\grad N\|_{L^2}^2.
    \end{align*}

For the quadratic term, we have
    \begin{align*}
        -& \sum_{i=1}^2z_i D_i\int_\Omega \dive(\cisum\grad \psum) \dive(\frac{\grad \cisum}{\cio}+z_i\grad \psum) \, dx \\
        &\leq  M\sum_{i=1}^2  \|\grad\cisum\|_{L^2}^\frac{1}{2} \|\lap\cisum\|_{L^2}^\frac{1}{2} \|\lap\psum\|_{L^2}(\|\grad \cio\|_{L^\oo}+1)(\|\lap\cisum\|_{L^2}+\|\lap\psum\|_{L^2})\\
        &\leq M\sum_{i=1}^2(\|\grad \cio\|_{L^\oo}^4+1) \|\lap\psum\|_{L^2}^4\|\grad \cisum\|_{L^2}^2 + \frac{D^*}{32}\sum_{i=1}^2\bigl(\int_\Omega \frac{|\lap \cisum |^2}{\cio}\,dx+z^2_i \int_\Omega \cio|\lap \psum|^2 \,dx \bigr).
    \end{align*}

The estimates of rest terms are straightforward
\begin{align*}
    -& \sum_{i=1}^2z_i D_i\int_\Omega (K_i+L_i+M_i) \dive(\frac{\grad \cisum}{\cio}+z_i\grad \psum) \, dx \\
    &\quad \leq M\sum_{i=1}^2(\|\grad \cio\|_{L^\oo}^2+1)(\|K_i\|_{L^2}^2 + \|L_i\|_{L^2}^2+ \|M_i\|_{L^2}^2)\\
    &\qquad +\frac{D^*}{32}\sum_{i=1}^2\bigl(\int_\Omega \frac{|\lap \cisum |^2}{\cio}\,dx+z^2_i \int_\Omega \cio|\lap \psum|^2 \,dx \bigr).
\end{align*}

By summarizing the above estimates, we achieve
\begin{equation}\label{estimate of cisum,sum}
    \begin{split}
         \frac{1}{2} &\frac{d}{dt}\sum_{i=1}^2 \int_\Omega \frac{|\grad\cisum|^2}{\cio}\,dx + \frac{5D^*}{16}   \sum_{i=1}^2\bigl(\int_\Omega \frac{|\lap \cisum |^2}{\cio}\,dx+z^2_i \int_\Omega \cio|\lap \psum|^2 \,dx \bigr)+\frac{D^*}{\ve^2}\|\grad \rsum\|^2_{L^2}    \\
         &\leq M\bigl(\sum_{i=1}^2\|\pt\cio\|_{L^\oo}  + \sum_{i=1}^2(\|\grad \cio\|_{L^\oo}^4+1)(\|\lap\psum\|_{L^2}^4+1) \bigr) \sum_{i=1}^2 \int_\Omega \frac{|\grad\cisum|^2}{\cio}\,dx  \\
         &\quad + M \sum_{i=1}^2(\|\grad \cio\|_{L^\oo}^4+1)\bigl(\|\grad \cisum\|_{L^2}^2 \|\usum\|_V^4 + \|K_i\|_{L^2}^2 +\|L_i\|_{L^2}^2 +\|M_i\|_{L^2}^2\bigr)\\
         &\quad +M\sum_{i=1}^2(\|\grad \cio\|_{L^\oo}^2+1)\|\grad \psum\|_{L^2}^2 +M\|\pt N\|_{L^2}^2 +M\ve^{-2} \|\grad N\|_{L^2}^2.
    \end{split}
\end{equation}

Similarly, we get, taking $L^2$ inner product of the  equation (\ref{equation of usum}) with $A\usum,$ that
\begin{equation}\label{estimate of usum,sum}
    \begin{split}
        \frac{1}{2}\frac{d}{dt}\|\usum\|_{V}^2  +\frac{\nu}{2} \|A\usum\|_{L^2}^2 &\leq M \|\lap \psum\|_{L^2}^4 \sum_{i=1}^2 \int_\Omega \frac{|\grad\cisum|^2}{\cio}\,dx+ M\|\usum\|_V^6+ \|O\|_{L^2}^2 \\
        &\quad + \frac{D^*}{16}\sum_{i=1}^2\bigl(\int_\Omega \frac{|\lap \cisum |^2}{\cio}\,dx+z^2_i \int_\Omega \cio|\lap \psum|^2 \,dx \bigr) +\|P\|_{L^2}^2 .
    \end{split}
\end{equation}

Notice that $\|\grad \cisum\|_{L^2}+\|\lap\psum\|_{L^2}+\|\usum\|_{V}  \leq 3$ on $[0,T^\star]$, by virtue of (\ref{result 1 in thm 1}-\ref{result 3 in thm 1}), Lemma \ref{estimate of K_i-P_i}, Lemma \ref{refined convergence rate of L2} and all the estimates of the layer solutions obtained in Section \ref{analysis of layers},
 we deduce (\ref{estimate of thm4}) from (\ref{estimate of cisum,sum}) and (\ref{estimate of usum,sum}).
\end{proof}

With Proposition \ref{prop5.1}, we are able to complete the proof of Theorem \ref{thm4}.

\begin{proof}[Proof of Theorem \ref{thm4}]
By using Gronwall's inequality to (\ref{estimate of thm4}) with initial conditions (\ref{initial condition of thm4}), one has (\ref{estimate of lap psum1}) and (\ref{estimate of lap psum2}) for $t\leq T^\star$. In particular, we have $\|\grad \cisum\|_{L^2}+\ve\|\lap \psum\|_{L^2} + \|\usum\|_{V}\leq M\ve^\frac{3}{2}$ for some $M$ irrelevant with $T^\star$. Thus by taking $\ve<1$ being sufficiently small so that $M\ve^\frac{1}{2} \leq \frac{1}{2}$, we conclude that $T^\star = T$ by the standard continuous argument, which completes the proof of Theorem \ref{thm4}.
\end{proof}

\subsection{\bf Estimate of $\|(\pt\cisum,\ve \grad \pt \psum,\pt \usum)\|_{L^\oo_T(L^2)}$.}
The goal of this subsection is to prove the following type of inequality:
\begin{align}\label{goal estimate in Loo H2}
    \frac{d}{dt}\int_\Omega \Bigl(\frac{|\p_t \cisum|^2}{\cio} + \ve^2|\grad \pt \psum|^2 + |\pt \usum|^2\Bigr)\,dx  + \int_\Omega \Bigl(\frac{|\grad \pt \cisum|^2}{\cio} + \cio|\grad \pt \psum|^2\Bigr) \,dx \\
    + \ve^2 \|\pt \lap \psum \|^2_{L^2} +\nu \|\pt \usum\|_V^2    \leq O(\ve) . \nonumber
\end{align}

For simplicity, we only n consider the difference functions as follows:
\begin{equation}\label{S5eq2}
    \quad \left\{\begin{array}{l}
         \cirs(x',y,t) = \cei(x',y,t) - \cio(x',y,t)  -\ve\cil(x',y,t)- \ve^2 \ciir(x',y,\tau)  \\
         \qquad \qquad \qquad - \ve^2 f(y)\cilmr(x',\xi,\tau) - \ve^2g(y)\cirmr(x',\eta,\tau) ,\\
        \prs(x',y,t)= \Phi^\ve(x',y,t) -\po(x',y,t)-\pio(x',y,\tau)-\ve\pil(x',y,\tau), \\
         \urs(x',y,t)= \ue(x',y,t)-\uo(x',y,t),
    \end{array}
    \right.
\end{equation}
which satisfy the following boundary conditions:
\begin{equation}
    \cirs |_{\p \Omega}=0, \quad \prs |_{\p \Omega}=0 \andf \urs |_{\p \Omega}=0.
\end{equation}
\begin{rmk}
    Based on (\ref{goal estimate in Loo H2}), the construction of $(\cirs,\prs,\urs)$ guarantees homogeneous Dirichlet boundary conditions and collects all the layer solutions appearing in the left-hand side of (\ref{goal estimate in Loo H2}). Though the function $\ve^2 f(y)\cilmr(x',\xi,\tau)$, $\ve^2g\cirmr(x',\eta,\tau) $ are too ``small" to appear in the energy estimate (\ref{goal estimate in Loo H2}), they compensate the boundary value of $\ve^2 \ciir(x',y,\tau)$.
\end{rmk}
Under the assumption of Theorem \ref{thm4},  we have
\begin{subequations}
\begin{gather}
\label{estimate of lap prs1}
     \|\grad \cirs\|_{L^\oo_T(L^2)} + \ve^\frac{1}{2}\|\urs\|_ {L^\oo_T(V)} + \ve \|\lap \prs\|_ {L^\oo_T(L^2)}  \leq M_2 \ve^\frac{3}{2} ,\\
     \ve^\frac{1}{2}\|\lap  \cirs\|_{L^2_T(L^2)} + \|A \urs\|_ {L^2_T(L^2)}  + \ve^\frac{1}{2}\|\lap \prs\|_ {L^2_T(L^2)}  \leq M_2 \ve. \label{estimate of lap prs2}
\end{gather}
\end{subequations}

It is easy to observe that  $(c_i^R,\Phi^R,u^R)$ verifies
\begin{subequations}\label{S5eq3}
\begin{gather}
\label{equation of cirs}
    \pt \cirs +\urs\cdot \grad \cirs -\ve z_iD_i\dive(\cio \grad \pil+\cil\grad \pio) = D_i \dive(\cio(\frac{\grad \cirs}{\cio} + z_i \grad \prs)) \\
   \qquad \qquad +z_iD_i\dive(\cirs\grad\prs)+E_i+F_i+G_i+H_i, \nonumber\\
\label{equation of rrs}
    \qquad\rrs=-\ve^2\lap \prs -\ve^2 \lap \po -\ve^3 \lap \pil +\ve^2 f\p_\xi^2 \plmr+\ve^2g\p_\eta^2 \prmr,\\
\label{equation of urs}
    \pt \urs +(\urs+\uo)\cdot \grad \urs - \nu \lap \urs -\grad p^R= -\urs \cdot \grad \uo \\
    \qquad \qquad -(\rrs + \ve^2\rir+\ve^2f\rlmr+\ve^2g\rrmr)\grad(\prs+\po+\pio+\ve \pil), \nonumber \\
    \dive\urs =0 ,
\end{gather}
\end{subequations}
where
\begin{align}
   & E_i \eqdefa -\uo\cdot\grad \cirs -\urs\cdot \grad (\cio+\ve\cil+\ve^2 \ciir+\ve^2 f\cilmr+\ve^2 g \cirmr ) \nonumber \\
    &\qquad + z_iD_i\dive\big((\ve\cil +\ve^2\ciir+\ve^2f\cilmr+\ve^2g\cirmr)\grad\prs\big) \nonumber \\
    &\qquad + z_iD_i\dive\big(\cirs \grad (\po+\pio+\ve\pil) \big)\nonumber, \\
    &F_i \eqdefa z_iD_i\dive\bigl((\cio-\cio(0))\grad \pio\bigr) - z_iD_i\gamma_i(x',0,0)f\p_\xi^2 \plmr -   z_iD_i\gamma_i(x',1,0)g\p_\eta^2 \prmr ,\nonumber \\
    &G_i \eqdefa \ve \ul\cdot \grad \cio- \ve z_iD_i \dive(\cio \grad \pl) +2\ve D_i(f'\p_\xi \cilmr - g'\p_\eta \cirmr), \nonumber\\
    &H_i \eqdefa -\ve^2\uo \cdot \grad (\ciir+f\cilmr+g\cirmr) + \ve^2 D_i \lap\ciir \nonumber \\
    &\qquad + \ve^2D_i(f\lap_{x'}\cilmr + g\lap_{x'} \cirmr) +\ve^2D_i (f''\cilmr+ g'' \cirmr) \nonumber \\
    &\qquad  + \ve^2 z_iD_i\dive(\cil\grad\pil)+\ve^2z_iD_i\dive\bigl((\ciir+f\cilmr+g\cirmr)\grad(\po+\pio+\ve\pil )\bigr) . \nonumber
\end{align}

By taking time derivative of equation (\ref{equation of cirs}) and (\ref{equation of urs}), we get
\beq\label{equation of pt cirs}
\begin{split}
    \pt^2 \cirs +& \pt \urs\cdot \grad \cirs +\urs\cdot \grad \pt \cirs-\ve^{-1} z_iD_i\dive(\cio \grad \p_\tau \pil+\cil\grad \p_\tau \pio) \\
    &= D_i \dive(\cio(\frac{\grad \pt \cirs}{\cio} + z_i \grad \pt\prs)) +z_iD_i\pt \dive(\cirs\grad\prs)  \\
    &\quad +z_iD_i\dive(\pt \cio \grad \prs) +\ve z_iD_i\dive(\pt\cio \grad \pil+\pt\cil\grad \pio) \\
  &\quad  +\pt E_i+\pt F_i+\pt G_i+\pt H_i ,
\end{split} \eeq
and \beq\label{equation of pt urs}
\begin{split}
    &\pt^2 \urs +(\urs+\uo)\cdot \grad \pt \urs - \nu \lap \pt \urs -\grad \pt p^R \\
    & \quad = - \pt(\urs+\uo)\cdot \grad  \urs-\pt \urs \cdot \grad \uo - \urs \cdot \grad \pt\uo  \\
    & \qquad -(\pt\rrs + \p_\tau \rir+f\p_\tau \rlmr+g\p_\tau \rrmr)\grad(\prs+\po+\pio+\ve \pil)  \\
    & \qquad -(\rrs + \ve^2\rir+\ve^2f\rlmr+\ve^2g\rrmr)\grad(\pt\prs+\pt\po+\ve^{-2}\p_\tau\pio+ \ve^{-1}\p_\tau\pil).
\end{split}\eeq

\begin{thm}\label{thm3}
{\sl
Under the assumptions in Theorem \ref{thm4}, if we assume in addition, that
\begin{subequations}\label{assumption in thm3}
    \begin{gather}
         \|\cirs(0)\|_{L^2}   \leq M_1 \ve^2, \label{assumption 1 in thm3} \\
     \|\pt \cirs(0)\|_{L^2} + \|\pt \urs(0)\|_ {L^2} + \ve \|\grad \pt \prs(0)\|_ {L^2}  \leq M_1 \ve^\frac{1}{2} .\label{assumption 2 in thm3}
    \end{gather}
\end{subequations}
Then for $\ve$ being sufficiently small, there exists a positive constant $M_2$ depending only on initial data, $\nu,M_1, T,\lam, \Lam, W(x), \ga_i(x), z_i$ and $D_i$ for $i=1,2$, so that
\begin{subequations}
    \begin{gather}
     \|\pt \cirs\|_{L^\oo_T(L^2)} + \|\pt \urs\|_ {L^\oo_T(L^2)} + \ve \|\grad \pt \prs\|_ {L^\oo_T(L^2)}  \leq M_2 \ve^\frac{1}{2} ,\label{result1 of thm3}\\
     \|\grad \pt \cirs\|_{L^2_T(L^2)} + \|\grad \pt \urs\|_ {L^2_T(L^2)}  + \|\grad \pt \prs\|_ {L^2_T(L^2)}+\ve \|\lap \pt \prs\|_ {L^2_T(L^2)}  \leq M_2 \ve^\frac{1}{2}\label{result2 of thm3} .
\end{gather}\end{subequations}}
\end{thm}
Since Theorem \ref{thm3} is a direct consequence of  the following estimate (\ref{estimate of cirs and urs, sum}) and Gronwall's inequality, it suffices to prove (\ref{estimate of cirs and urs, sum}).
\begin{prop}\label{estimate of thm3}
    Under the assumptions in Theorem \ref{thm3}, for any $t\in [0,T]$: one has
    \begin{equation}\label{estimate of cirs and urs, sum}
        \begin{split}
     \frac{d}{dt} & \Bigl(\sum_{i=1}^2 \int_{\Omega } \frac{|\pt\cirs|^2}{\cio} \,dx +\ve^2\|\grad \pt \prs\|^2_{L^2}
     + \frac{1}{2}\|\pt \urs\|_{L^2}^2\Bigr) +\frac{\nu}{4} \|\grad \pt \urs\|^2_{L^2}  \\
    & +\frac{D^*}{8}\Bigl( \sum_{i=1}^2  \int_\Omega \frac{|\grad \pt \cirs|^2}{\cio}\,dx +\sum_{i=1}^2 z_i^2 \int_\Omega \cio |\grad \pt \prs|^2 \,dx\Bigr) +\frac{D^*\ve^2}{2} \|\lap \pt \prs \|^2_{L^2}  \\
    \leq& \kappa_5(t) \sup_{[0,t]}(\sum_{i=1}^2 \int_{\Omega } \frac{|\pt\cirs|^2}{\cio} \,dx +\ve^2\|\grad \pt \prs\|^2_{L^2} + \|\pt \urs\|_{L^2}^2) + \kappa_6(t),
    \end{split}
    \end{equation}
where $\kappa_5,\kappa_6$ are positive functions satisfying
\begin{equation*}
   \int_0^T \kappa_5(t') \, dt' \sim O(1) \andf  \int_0^T \kappa_6(t') \, dt' \sim O(\ve), \quad \textit{as} \ \ve\rto 0.
\end{equation*}
\end{prop}

\begin{proof} By taking $L^2$ inner product  of the equation (\ref{equation of pt urs}) with
$\pt \urs$ and  using integration by parts,  we get
\begin{align*}
    \frac{1}{2}&\frac{d}{dt}  \|\pt  \urs\|_{L^2}^2  + \nu \|\grad \pt \urs\|^2_{L^2}
        \leq
      \|A\urs\|_{L^2}\|\pt\urs\|_{L^2}\|\grad \pt\urs\|_{L^2} + \|\pt \uo\|_{L^\oo}\|A \urs\|_{L^2} \|\pt \urs\|_{L^2} \\
    & \quad +\|\grad \uo\|_{L^\oo} \|\pt \urs\|_{L^2}^2 + \|\grad\pt \uo\|_{L^\oo}  \| \urs\|_{L^2} \| \pt \urs\|_{L^2} \\
    &\quad + \|\lap \prs\|_{L^2} \|\pt \rrs\|_{L^2} \|\grad \pt \urs\|_{L^2} +  \|\grad \prs\|_{L^2} \|\p_\tau \rir+f \p_\tau \rlmr+ g \p_\tau \rrmr\|_{L^\oo} \|\pt \urs\|_{L^2} \\
    & \quad +  \| \grad(\po+\pio+\ve \pil) \|_{L^\oo} \|\pt \rrs\|_{L^2} \|\pt \urs\|_{L^2}  \\
    & \quad + \| \grad(\po+\pio+\ve \pil)\|_{L^\oo}\|\p_\tau \rir+f \p_\tau \rlmr+ g \p_\tau \rrmr\|_{L^2} \|\pt \urs\|_{L^2}  \\
    &\quad +\|\rrs\|_{L^\oo} \|\grad \pt \prs\|_{L^2} \|\pt \urs\|_{L^2}  \\
    &\quad +\| \ve^2 \grad \pt \po + \grad \p_\tau \pio + \ve\grad\p_\tau \pil \|_{L^\oo} \|\ve^{-2} \rrs \|_{L^2}  \|\pt \urs\|_{L^2}  \\
    &\quad +\ve^2\|\rir+ f\rlmr+g\rrmr\|_{L^\oo} \|\grad \pt \prs\|_{L^2}  \|\pt \urs\|_{L^2}   \\
    &\quad + \|\rir+ f\rlmr+g\rrmr\|_{L^2}  \| \ve^2 \grad \pt \po + \grad \p_\tau \pio + \ve\grad\p_\tau \pil \|_{L^\oo}  \|\pt \urs\|_{L^2}.
    \end{align*}
    Applying H\"older's inequality yields
    \begin{align*}
    \frac{1}{2}&\frac{d}{dt}  \|\pt  \urs\|_{L^2}^2  + \nu \|\grad \pt \urs\|^2_{L^2}\\
    & \leq M\Bigl( 1+\frac{1}{\nu}+ \frac{1}{\nu} \|A\urs\|_{L^2}^2 + \| \pt \uo\|_{H^3}^2 + \|\grad \uo\|_{L^\oo} + \|\rrs \|_{L^\oo}^2+ \|\ve^{-2} \rrs \|_{L^2}^2 \\
    & \qquad \quad + \ve^4\|\rir+ f\rlmr+g\rrmr\|_{L^\oo}^2 + \|\rir+ f\rlmr+g\rrmr\|_{L^2}^2  \Bigr) \|\pt\urs\|_{L^2}^2 \\
    &\quad + M(\frac{1}{\nu}\|\lap \prs\|_{L^2}^2 + \| \grad(\po+\pio+\ve \pil) \|_{L^\oo}^2 ) \sum_{i=1}^2 \int_{\Omega } \frac{|\pt\cirs|^2}{\cio} \,dx \\
    &\quad + \frac{3\nu}{4}\|\grad \pt \urs\|_{L^2}^2 +   \frac{D^*}{16}\sum_{i=1}^2 z_i^2 \int_\Omega \cio |\grad \pt \prs|^2 \,dx \\
    &\quad + M\bigl(\|A \urs\|_{L^2}^2  +\|\grad \prs \|_{L^2}^2\|\p_\tau\rir+ f\p_\tau\rlmr+g\p_\tau\rrmr\|_{L^\oo}^2 \bigr) \\
    &\quad +M\| \grad(\po+\pio+\ve \pil)\|_{L^\oo}^2 \|\p_\tau \rir+f \p_\tau \rlmr+ g \p_\tau \rrmr\|_{L^2}^2 \\
    & \quad + M\| \ve^2 \grad \pt \po + \grad \p_\tau \pio + \ve\grad\p_\tau \pil \|_{L^\oo}^2.
\end{align*}

By inserting (\ref{result 1 in thm 1}-\ref{result 3 in thm 1}), (\ref{estimate of lap prs1}-\ref{estimate of lap prs2}) and all the estimates of the layer solutions obtained in Section \ref{analysis of layers} into the above inequality, we achieve
\begin{equation}\label{estimate of urs,sum}
    \begin{split}
         \frac{1}{2}\frac{d}{dt}\|\pt \urs\|_{L^2}^2 + \frac{\nu}{4} \|\grad \pt \urs\|^2_{L^2} \leq &\kappa_7(t) \bigl(\sum_{i=1}^2\int_{\Omega } \frac{|\pt\cirs|^2}{\cio} \,dx + \|\pt u^R\|^2_{L^2} \bigr) \\
    &  +   \frac{D^*}{16} \sum_{i=1}^2 z_i^2 \int_\Omega \cio |\grad \pt \prs|^2 \,dx   + \kappa_8(t),
    \end{split}
\end{equation}

where  $\kappa_7$, $\kappa_8$ are positive functions satisfying
\begin{equation*}
    \int_0^T  \kappa_7(t') \,dt' \sim O(1) \andf \int_0^T  \kappa_8(t') \,dt' \sim O(\ve^2), \quad \textit{as $\ve \rto 0$} .
\end{equation*}

Below let us handle the estimate  of $(\pt\cirs, \pt\prs).$ We shall first take $L^2$ inner product
of the equation (\ref{equation of pt cirs}) with $\frac{\pt \cirs }{\cio} + z_i \pt \prs$ 
 and then summarizing the resulting equalities for $i=1,2.$  
Indeed  by virtue of $\pt \prs|_{\p\Omega}=0$ and Poincar{\'e}'s inequality, one has
\begin{align*}
  \sum_{i=1}^2 \int_{\Omega } & \pt^2 \cirs (\frac{\pt \cirs }{\cio} + z_i \pt \prs) \,dx = \sum_{i=1}^2 \int_{ \Omega } \frac{\pt^2 \cirs \pt \cirs }{\cio} \,dx + \int_{\Omega }\pt^2 \rrs \pt \prs \,dx \\
  &=\frac{1}{2} \frac{d}{dt} \sum_{i=1}^2 \int_{\Omega } \frac{|\pt\cirs|^2}{\cio} \,dx - \frac{1}{2}\sum_{i=1}^2 \int_{ \Omega } |\pt\cirs|^2 \pt(\frac{1}{\cio}) \,dx \\
  & \quad +\frac{\ve^2}{2} \frac{d}{dt}\|\grad \pt \prs\|^2_{L^2} + \ve^2 \int_{ \Omega } \grad \pt^2\po \grad\pt\prs \,dx   - \ve^{-1} \int_{ \Omega } \lap \p_\tau^2 \pil \pt \prs \,dx \\
  &\quad - \ve^{-1}\int_{ \Omega } \p_\xi \p_\tau^2\plmr (f\p_y\pt\prs+f'\pt \prs ) \,dx'dy  \\
  &\quad + \ve^{-1}\int_{ \Omega } \p_\eta \p_\tau^2 \prmr (g\p_y\pt\prs+g'\pt \prs ) \,dx'dy\\
  &\geq \frac{1}{2} \frac{d}{dt} \sum_{i=1}^2 \int_{\Omega } \frac{|\pt\cirs|^2}{\cio} \,dx +\frac{\ve^2}{2} \frac{d}{dt}\|\grad \pt \prs\|^2_{L^2} \\
  & \quad - M\|\pt \cio\|_{L^\oo} \sum_{i=1}^2 \int_{\Omega } \frac{|\pt\cirs|^2}{\cio} \,dx -\ve^2 \|\grad \pt \prs\|^2_{L^2} -\ve^2 \|\grad \pt^2 \po\|^2_{L^2}\\
  &  \quad - M\ve^{-2}\|\p_\xi \p_\tau^2 \plmr\|^2_{L^2} - M\ve^{-2}\| \p_\eta \p_\tau^2  \prmr \|^2_{L^2}  \\
  &\quad  - \frac{D^*}{16} \sum_{i=1}^2 z_i^2 \int_\Omega \cio |\grad \pt \prs|^2 \,dx - \ve^{-1} \int_{ \Omega } \lap \p_\tau^2 \pil \pt \prs \,dx .
\end{align*}

Recalling that $\p_\tau^2 \lap \pil = -\sum_{i=1}^2z_i^2D_i\dive(\cio(0)\grad \p_\tau\pil + \cil(0)\grad \p_\tau \pio).$
We get, by using integration by parts, that
\begin{align*}
    -&\ve^{-1} \sum_{i=1}^2z_iD_i \int_\Omega \dive(\cio\grad \p_\tau \pil+\cil\grad\p_\tau \pio)(\frac{\pt \cirs }{\cio}+z_i \pt \prs)\,dx  \\
    &= \ve^{-1} \sum_{i=1}^2z_iD_i \int_\Omega (\cio\grad \p_\tau \pil+\cil\grad\p_\tau \pio) \grad (\frac{\pt \cirs }{\cio}) \,dx \\
    &\quad +\ve^{-1} \sum_{i=1}^2z_i^2 D_i \int_\Omega \Big( (\cio-\cio(0)) \grad \p_\tau \pil+(\cil-\cil(0)) \grad\p_\tau \pio\Big)  \grad\pt \prs \,dx  \\
    &\quad + \ve^{-1} \int_{ \Omega } \lap \p_\tau^2 \pil \pt \prs \,dx  \\
     &\geq -M\sum_{i=1}^2\bigl(\|\pt\cio\|_{L^\oo}^2 + \ve^{-2}\|\cio\|_{L^\oo}^2\|\grad\p_\tau\pil\|_{L^2}^2+\ve^{-2}\|\cil\|_{L^\oo}^2\|\grad\p_\tau\pio\|_{L^2}^2 \bigr) \int_{\Omega } \frac{|\pt\cirs|^2}{\cio} \,dx \\
     &\quad -M \ve^2\sum_{i=1}^2\bigl(\|\pt \cio\|_{L^\oo}^2 + \|\pt \cil\|_{L^\oo}^2\bigr) \|\grad \pt \prs\|^2_{L^2}   \\
     &\quad -M\sum_{i=1}^2\bigl(\|\pt\cio\|_{L^\oo}^2 + \ve^{-2}\|\cio\|_{L^\oo}^2\|\grad\p_\tau\pil\|_{L^2}^2+\ve^{-2}\|\cil\|_{L^\oo}^2\|\grad\p_\tau\pio\|_{L^2}^2 \bigr) \\
     &\quad -\|\tau \grad \p_\tau \pio\|_{L^2}^2- \|\tau \grad \p_\tau \pil\|_{L^2}^2+ \ve^{-1} \int_{ \Omega } \lap \p_\tau^2 \pil \pt \prs \,dx .
\end{align*}

While for the dissipation term, one has
    \begin{align*}
         \sum_{i=1}^2 & D_i  \int_\Omega\dive(\cio(\frac{\grad \pt \cirs}{\cio} + z_i \grad \pt\prs))(\frac{\pt \cirs }{\cio} + z_i \pt \prs) \,dx \\
    & = -\sum_{i=1}^2  D_i  \int_\Omega \cio|\frac{\grad \pt \cirs}{\cio} + z_i \grad \pt\prs|^2 \,dx \\
    &\quad-\sum_{i=1}^2  D_i  \int_\Omega \cio(\frac{\grad \pt \cirs}{\cio} + z_i \grad \pt\prs)\pt\cirs \grad (\frac{1}{\cio}) \,dx  \\
    & \leq -\frac{D^*}{2}\sum_{i=1}^2  \int_\Omega \cio|\frac{\grad \pt \cirs}{\cio} + z_i \grad \pt\prs|^2 \,dx +M\sum_{i=1}^2\|\grad\cio\|_{L^\oo}^2 \|\pt \cirs\|^2_{L^2}\\
    &\leq -\frac{D^*}{2}\Bigl( \sum_{i=1}^2  \int_\Omega \frac{|\grad \pt \cirs|^2}{\cio}\,dx +\sum_{i=1}^2 z_i^2 \int_\Omega \cio |\grad \pt \prs|^2 \,dx +2\ve^2 \|\lap \pt \prs \|^2_{L^2} \Bigr)  \\
    & \quad -\ve^2 D^*\int_\Omega \lap \pt \po \lap \pt \prs \,dx  - \ve D^* \int_\Omega \lap \p_\tau \pil \lap \pt \prs \,dx   \\
    &\quad - \int_\Omega (f \p_\tau \rlmr +g \p_\tau \rrmr ) \lap \pt \prs \,dx +M\sum_{i=1}^2\bigl(\|\grad\cio\|_{L^\oo}^2 \|\pt \cirs\|^2_{L^2}\bigr)  \\
    & \leq -\frac{D^*}{2}\Bigl( \sum_{i=1}^2  \int_\Omega \frac{|\grad \pt \cirs|^2}{\cio}\,dx +\sum_{i=1}^2 z_i^2 \int_\Omega \cio |\grad \pt \prs|^2 \,dx \Bigr) -\frac{D^*\ve^2}{2} \|\lap \pt \prs \|^2_{L^2}   \\
    & \quad +M\sum_{i=1}^2\|\grad\cio\|_{L^\oo}^2\int_\Omega \frac{|\grad \pt \cirs|^2}{\cio}\,dx + M \ve^2 \|\lap \pt \po \|_{L^2}^2 + M\|\lap \p_\tau \pil \|_{L^2}^2   \\
    &\quad +M\ve^{-2} \bigl( \| \p_\tau \rlmr\|_{L^2}^2   + \|\p_\tau \rrmr\|_{L^2}^2 \bigr)    .
    \end{align*}

For the quadratic terms, noting that $\|\cirs\|_{L^\oo}\lesssim 1$, we get, by using integration by parts and H{\" o}lder's inequality, that
    \begin{align*}
        \sum_{i=1}^2 & \int_\Omega \pt \urs \cdot \grad \cirs (\frac{\pt \cirs }{\cio} + z_i \pt \prs)\,dx  \\
    & = - \sum_{i=1}^2 \int_\Omega \pt \urs \cdot   \frac{\grad \pt \cirs }{\cio} \cirs \,dx - \sum_{i=1}^2 \int_\Omega \pt \urs \cdot \grad (\frac{1}{\cio})  \pt \cirs  \cirs \,dx  \\
    &  \quad - \int_\Omega\pt \urs \cdot \grad \pt \prs \rrs \,dx  \\
    & \geq - M\|\pt \urs\|^2_{L^2} - M\sum_{i=1}^2\|\grad \cio\|_{L^\oo}^2 \int_{\Omega } \frac{|\pt\cirs|^2}{\cio} \,dx \\
    &\quad - \frac{D^*}{32} \Big( \sum_{i=1}^2  \int_\Omega \frac{|\grad \pt \cirs|^2}{\cio}\,dx +\sum_{i=1}^2 z_i^2 \int_\Omega \cio |\grad \pt \prs|^2 \,dx \Big) .
    \end{align*}
Thanks to $\dive \urs =0$, one has
\begin{align*}
    \sum_{i=1}^2 &\int_\Omega  \urs \cdot \grad \pt \cirs (\frac{\pt \cirs }{\cio} + z_i \pt \prs)\,dx  \\
    & = - \frac{1}{2}\sum_{i=1}^2 \int_\Omega \urs\cdot\grad(\frac{1}{\cio}) |\pt\cirs|^2 \,dx   - \int_\Omega \urs \cdot \grad \pt \prs \pt \rrs \,dx \\
    & \geq -M\|\grad \cio\|_{L^\oo}\|\urs\|_V\|\pt\cirs\|_{L^2}^{\frac{1}{2}}\|\grad\pt\cirs\|_{L^2}^\frac{3}{2} -M\|\urs\|_V\|\pt\rrs\|_{L^2}^\frac{1}{2}\|\grad\pt\rrs\|_{L^2}^\frac{1}{2}\|\grad \pt\prs\|_{L^2}\\
    & \geq -M( \sum_{i=1}^2\|\grad \cio\|_{L^\oo}^4+1)\|\urs\|^4_{V} \sum_{i=1}^2 \int_{\Omega } \frac{|\pt\cirs|^2}{\cio} \,dx \\
    & \quad -  \frac{D^*}{32} \Big( \sum_{i=1}^2  \int_\Omega \frac{|\grad \pt \cirs|^2}{\cio}\,dx +\sum_{i=1}^2 z_i^2 \int_\Omega \cio |\grad \pt \prs|^2 \,dx \Big) .
\end{align*}

Whereas for the term $\pt \dive(\cirs\grad \prs)$, one has
\begin{equation*}
\begin{split}
    \sum_{i=1}^2 &z_iD_i\int_\Omega \pt \dive(\cirs\grad\prs)   (\frac{\pt \cirs }{\cio} + z_i \pt \prs) \,dx  \\
   & =   -\sum_{i=1}^2 z_iD_i \int_\Omega \pt\cirs\grad \prs \grad(\frac{\pt \cirs }{\cio} + z_i \pt \prs) \,dx    \\
    &  \quad  -\sum_{i=1}^2z_iD_i \int_\Omega \cirs\grad \pt\prs \grad(\frac{\pt \cirs }{\cio} + z_i \pt \prs) \,dx  , \\
\end{split} \end{equation*}
from which, we infer
\begin{equation*}
\begin{split}
 \sum_{i=1}^2 &z_iD_i\int_\Omega \pt \dive(\cirs\grad\prs)   (\frac{\pt \cirs }{\cio} + z_i \pt \prs) \,dx   \\
    &\leq  M\sum_{i=1}^2 \bigl(\|\pt\cirs\|_{L^2}^\frac{1}{2}\|\|\grad\pt \cirs\|_{L^2}^\frac{1}{2} \|\lap\prs\|_{L^2} + \|\grad \cirs\|_{L^2}^\frac{1}{2}\|\lap\cirs\|_{L^2}^\frac{1}{2}\|\grad \pt\prs\|_{L^2} \bigr) \\
    &\qquad \times \bigl(1+\|\grad \cio\|_{L^\oo} \bigr) \bigl(\| \grad\pt \cirs\|_{L^2} + \|\grad \pt \prs \|_{L^2} \bigr)  \\
    &\leq M \sum_{i=1}^2\bigl(1+\|\grad \cio\|_{L^\oo}^4 \bigr)\|\lap\prs\|_{L^2}^4 \int_{\Omega } \frac{|\pt\cirs|^2}{\cio} \,dx\\
    &\quad + M\sum_{i=1}^2 \bigl(1+\|\grad \cio\|_{L^\oo}^4 \bigr)\|\grad\cirs\|_{L^2}^2\|\lap\cirs\|_{L^2}^2 \|\grad \pt \prs\|_{L^2}^2  \\
    & \quad + \frac{D^*}{32} \Big( \sum_{i=1}^2  \int_\Omega \frac{|\grad \pt \cirs|^2}{\cio}\,dx +\sum_{i=1}^2 z_i^2 \int_\Omega \cio |\grad \pt \prs|^2 \,dx  \Big).
\end{split} \end{equation*}

The estimates of the rest terms are much more straightforward. Precisely, by using integration by parts and H\"older's inequality, one has
\begin{align*}
   \sum_{i=1}^2 &\int_\Omega \bigl(z_iD_i\dive(\pt \cio \grad \prs) +\ve z_iD_i\dive(\pt\cio \grad \pil \\
   &\qquad \qquad \qquad \qquad \qquad \qquad +\pt\cil\grad \pio) \bigr)\bigl(\frac{\pt \cirs }{\cio} + z_i \pt \prs\bigr) \,dx   \\
    \leq &M\sum_{i=1}^2\bigl(\|\pt\cio\|_{L^\oo} +\|\pt\cil\|_{L^\oo} \bigr) \bigl(\| \grad\prs\|_{L^2}+ \ve\|\grad \pio\|_{L^2} + \ve\|\grad \pil\|_{L^2} \bigr) \\
   & \qquad \times \bigl(1+\|\grad \cio\|_{L^\oo} \bigr)\bigl(\|\grad\pt \cirs\|_{L^2}+ \|\grad \pt \prs \|_{L^2}\bigr) \\
   \leq &  \frac{D^*}{32} \bigl( \sum_{i=1}^2  \int_\Omega \frac{|\grad \pt \cirs|^2}{\cio}\,dx +\sum_{i=1}^2 z_i^2 \int_\Omega \cio |\grad \pt \prs|^2 \,dx  \bigr)  \\
   & + M\sum_{i=1}^2\bigl(\|\pt\cio\|_{L^\oo}^2 +\|\pt\cil\|_{L^\oo}^2 \bigr) \bigl(1+\|\grad \cio\|_{L^\oo}^2 \bigr) \\
   &\qquad \times \bigl(\| \grad\prs\|_{L^2}^2+ \ve\|\grad \pio\|_{L^2}^2 + \ve\|\grad \pil\|_{L^2}^2 \bigr).
\end{align*}
Similarly, one has
\begin{align*}
    &\sum_{i=1}^2 \int_\Omega \pt E_i (\frac{\pt \cirs }{\cio} + z_i \pt \prs) \,dx  \\
    &= \sum_{i=1}^2 \int_\Omega  (\pt\uo \cirs +\uo\pt\cirs) \cdot \grad(\frac{\pt \cirs }{\cio} + z_i \pt \prs) \,dx  \\
    & \quad+ \sum_{i=1}^2\int_\Omega  (\cio+\ve\cil+\ve^2 \ciir+\ve^2 f\cilmr+\ve^2 g \cirmr ) \pt \urs  \cdot \grad(\frac{\pt \cirs }{\cio} + z_i \pt \prs) \,dx  \\
    & \quad + \sum_{i=1}^2 \int_\Omega  (\pt \cio+\ve \pt \cil+\p_\tau \ciir+f\p_\tau \cilmr+g \p_\tau \cirmr ) \urs \cdot \grad (\frac{\pt \cirs }{\cio} + z_i \pt \prs) \,dx  \\
    & \quad -\sum_{i=1}^2 z_iD_i \int_\Omega  (\ve \cil+\ve^2\ciir+\ve^2 f\cilmr+\ve^2 g\cirmr )\grad \pt \prs \grad (\frac{\pt \cirs }{\cio} + z_i \pt \prs) \,dx  \\
    & \quad -\sum_{i=1}^2 z_iD_i \int_\Omega  (\ve \pt \cil+\p_\tau \ciir+f \p_\tau\cilmr+ g\p_\tau \cirmr )\grad \prs \grad (\frac{\pt \cirs }{\cio} + z_i \pt \prs) \,dx  \\
    & \quad - \sum_{i=1}^2z_iD_i \int_\Omega \pt\cirs \grad(\po+\pio+\ve\pil)\grad (\frac{\pt \cirs }{\cio} + z_i \pt \prs) \,dx \\
    &\quad - \sum_{i=1}^2z_iD_i \int_\Omega (\int_0^1 \pt\cirs(\theta t)d \theta)\tau \grad( \ve^2 \pt\po+\p_\tau \pio+\ve\p_\tau\pil )\grad (\frac{\pt \cirs }{\cio} + z_i \pt \prs) \,dx  \\
    &\quad -\sum_{i=1}^2 z_iD_i \int_\Omega \cirs(0)\grad(\pt\po+\ve^{-2}\p_\tau \pio+\ve^{-1}\p_\tau\pil )\grad (\frac{\pt \cirs }{\cio} + z_i \pt \prs) \,dx, \end{align*}
    from which, we infer
    \begin{align*}
     \sum_{i=1}^2 &\int_\Omega \pt E_i (\frac{\pt \cirs }{\cio} + z_i \pt \prs) \,dx  \\
    & \leq  M\sum_{i=1}^2\bigl(1+\|\grad \cio\|_{L^\oo}^2 \bigr) \bigl(\|\uo\|_{L^\oo}^2 +\|\grad(\po+\pio+\ve\pil)\|^2_{L^\oo} \\
    &\qquad \qquad + \|\tau \grad( \ve^2 \pt\po+\p_\tau \pio+\ve\p_\tau\pil )\|_{L^\oo}^2 \bigr) \sup_{[0,t]} \int_{\Omega } \frac{|\pt\cirs|^2}{\cio} \,dx  \\
    &\quad +M\sum_{i=1}^2\bigl(1+\|\grad \cio\|_{L^\oo}^2 \bigr) \| \cio+\ve\cil+\ve^2 \ciir+\ve^2 f\cilmr+\ve^2 g \cirmr \|_{L^\oo}^2 \|\pt\urs\|^2_{L^2} \\
    &\quad +M\ve^2 \sum_{i=1}^2\bigl(1+\|\grad \cio\|_{L^\oo}^2 \bigr) \|\cil+\ve\ciir+\ve f\cilmr+\ve g\cirmr\|_{L^\oo}^2 \|\grad \pt \prs\|_{L^2}^2\\
    &\quad + \frac{D^*}{32} \Bigl( \sum_{i=1}^2  \int_\Omega \frac{|\grad \pt \cirs|^2}{\cio}\,dx +\sum_{i=1}^2 z_i^2 \int_\Omega \cio |\grad \pt \prs|^2 \,dx  \Bigr)  \\
    & \quad +M\sum_{i=1}^2\bigl(1+\|\grad \cio\|_{L^\oo}^2 \bigr) \bigl( \|\pt \cio+\ve \pt \cil+\p_\tau \ciir+f\p_\tau \cilmr+g \p_\tau \cirmr\|_{L^\oo}^2\|\urs\|_{L^2}^2 \\
    &\qquad \qquad  +\|\pt\uo\|_{L^\oo}^2\|\cirs\|_{L^2}^2 +\|\ve \pt \cil+\p_\tau \ciir+f \p_\tau\cilmr+ g\p_\tau \cirmr \|_{L^\oo}^2\|\grad \prs\|_{L^2}^2 \\
    &\qquad \qquad + \|\cirs(0)\|_{L^2}^2 \|\grad(\pt\po+\ve^{-2}\p_\tau \pio+\ve^{-1}\p_\tau\pil )\|_{L^\oo}^2 \bigr).
\end{align*}

Along the same line, we obtain
\begin{align*}
    &\sum_{i=1}^2\int_\Omega \pt F_i (\frac{\pt \cirs }{\cio} + z_i \pt \prs) \,dx  \\
    &= -\sum_{i=1}^2z_i D_i \int_\Omega \pt \cio \grad \pio  \grad(\frac{\pt \cirs }{\cio} + z_i \pt \prs) \,dx  \\
    &\quad  -\sum_{i=1}^2z_i D_i  \int_\Omega (\int_0^t \pt \cio(\theta s) ds ) \tau \grad \p_\tau \pio  \grad(\frac{\pt \cirs }{\cio} + z_i \pt \prs) \,dx  \\
    &\quad + \sum_{i=1}^2z_i D_i \gamma_i(x',0,0)\ve^{-1} \int_\Omega \p_\xi \p_\tau \plmr \p_y\bigl( f(\frac{\pt \cirs }{\cio} + z_i \pt \prs)\bigr) \,dx  \\
    &\quad - \sum_{i=1}^2z_i D_i \gamma_i(x',1,0)\ve^{-1} \int_\Omega \p_\eta \p_\tau\prmr \p_y \bigl( g(\frac{\pt \cirs }{\cio} + z_i \pt \prs)\bigr) \,dx  \\
    & \leq  \frac{D^*}{32} \Bigl( \sum_{i=1}^2  \int_\Omega \frac{|\grad \pt \cirs|^2}{\cio}\,dx +\sum_{i=1}^2 z_i^2 \int_\Omega \cio |\grad \pt \prs|^2 \,dx \Bigr) \\
    &\quad +M\sum_{i=1}^2\bigl(1+\|\grad \cio\|_{L^\oo}^2 \bigr) \bigl(\|\grad \pio\|^2_{L^2}\|\pt\cio\|_{L^\oo}^2 +\|\tau \grad \p_\tau \pio\|^2_{L^2}\|\pt\cio\|_{L^\oo}^2 \\
    &\qquad \qquad \quad + \ve^{-2}\|\p_\xi \p_\tau \plmr \|^2_{L^2} + \ve^{-2}\|\p_\eta \p_\tau \prmr \|^2_{L^2}\bigr) ,
\end{align*}
and
\begin{align*}
    &\sum_{i=1}^2  \int_\Omega \pt G_i (\frac{\pt \cirs }{\cio} + z_i \pt \prs) \,dx \\
     & \leq   \frac{D^*}{32} \Bigl( \sum_{i=1}^2  \int_\Omega \frac{|\grad \pt \cirs|^2}{\cio}\,dx +\sum_{i=1}^2 z_i^2 \int_\Omega \cio |\grad \pt \prs|^2 \,dx \Bigr) \\
     & \quad + M \sum_{i=1}^2\bigl(1+\|\grad \cio\|_{L^\oo}^2 \bigr)\bigl( \ve^2 \|\pt(\ul \cio)\|_{L^2}^2 + \ve^2 \|\pt(\cio \grad \pl) \|_{L^2}^2 \\
     &\qquad \qquad  + \ve^{-2}\|\p_\xi \p_\tau \cilmr\|_{L^2}^2 +\ve^{-2}\|\p_\eta \p_\tau \cilmr\|_{L^2}^2 \bigr),
\end{align*}
and
\begin{align*}
   & \sum_{i=1}^2 \int_\Omega \pt H_i (\frac{\pt \cirs }{\cio} + z_i \pt \prs) \,dx \\
    & \leq  \frac{D^*}{32} \Bigl( \sum_{i=1}^2  \int_\Omega \frac{|\grad \pt \cirs|^2}{\cio}\,dx +\sum_{i=1}^2 z_i^2 \int_\Omega \cio |\grad \pt \prs|^2 \,dx \Bigr) \\
    &\quad + M \sum_{i=1}^2\bigl(1+\|\grad \cio\|_{L^\oo}^2 \bigr) \bigl( \ve^2\|\pt \uo\|_{L^\oo}^2 \|\ciir+ f\cilmr +g\cirmr\|_{L^2}^2  \\
    &\qquad +\|\uo\|_{L^\oo}^2 \|\p_\tau(\ciir+ f\cilmr +g\cirmr)\|_{L^2}^2 + \|\grad \p_\tau \ciir\|_{L^2}^2 + \|\grad_{x'}\p_\tau(f\cilmr+g\cirmr)\|_{L^2}^2 \\
    &\qquad +\|\p_\tau(f''\cilmr+g''\cirmr) \|_{L^2}^2 + \ve^2 \|\pt\cil\|_{L^\oo}^2\|\grad \pil\|_{L^2}^2 +\|\cil\|_{L^\oo}^2\|\grad \p_\tau \pil\|_{L^2}^2 \\
    &\qquad +\|\p_\tau (\ciir+ f\cilmr +g\cirmr) \|_{L^2}^2 \|\grad(\po+\pio+\ve\pil)\|_{L^\oo}^2\\
    &\qquad +\|\ciir+ f\cilmr +g\cirmr\|_{L^2}^2 \|\grad(\ve^2 \pt\po+\p_\tau \pio+\ve\p_\tau \pil)\|_{L^\oo}^2 \bigr).
\end{align*}

By summing up the above inequalities and using (\ref{result 1 in thm 1}-\ref{result 3 in thm 1}), (\ref{estimate of lap prs1}-\ref{estimate of lap prs2}) and all the estimates concerning the layer solutions, which have been obtained in Section \ref{analysis of layers},  we achieve
\beq\label{estimate of cirs,sum}
\begin{split}
    & \frac{d}{dt} \sum_{i=1}^2 \int_{\Omega } \frac{|\pt\cirs|^2}{\cio} \,dx +\ve^2 \frac{d}{dt}\|\grad \pt \prs\|^2_{L^2}   \\
    &\quad +\frac{3D^*}{16}\Bigl( \sum_{i=1}^2  \int_\Omega \frac{|\grad \pt \cirs|^2}{\cio}\,dx +\sum_{i=1}^2 z_i^2 \int_\Omega \cio |\grad \pt \prs|^2 \,dx \Bigr)+\frac{D^*\ve^2}{2} \|\lap \pt \prs \|^2_{L^2}   \\
    &\leq \kappa_9(t) \sup_{[0,t]}(\sum_{i=1}^2 \int_{\Omega } \frac{|\pt\cirs|^2}{\cio} \,dx +\ve^2\|\grad \pt \prs\|^2_{L^2} + \|\pt \urs\|_{L^2}^2) + \kappa_{10}(t) ,
\end{split}\eeq
where $\kappa_9,\kappa_{10} $ are positive functions satisfying
\begin{equation*}
   \int_0^T \kappa_9(t') dt' \sim O(1) \andf  \int_0^T \kappa_{10}(t') dt' \sim O(\ve), \quad \textit{as} \ \ve \rto 0.
\end{equation*}

By combining (\ref{estimate of urs,sum}) and (\ref{estimate of cirs,sum}) together, we thus complete the proof of Proposition \ref{estimate of thm3}.
\end{proof}

\begin{rmk}
Recalling the definition of $(\cirs,\prs,\urs)$ and $(\cisum,\psum,\usum)$, we deduce that
\begin{equation}\label{estimate of pt cirs}
    \|\pt \cisum\|_{L^2} + \ve\|\pt \grad \psum\|_{L^2} + \|\pt \usum\|_{L^2} \lesssim \ve^{\frac{1}{2}}.
\end{equation}
\end{rmk}

\subsection{\bf Estimate of $\|\lap \cisum\|_{L^\oo_T(L^2)}$.}
Firstly we rewrite (\ref{equation of cisum}) as
\begin{align*}
    D_i \lap \cisum = \pt \cisum + \usum\cdot \grad \cisum -z_iD_i \dive(\cio\grad \psum) -z_iD_i \dive(\cisum\grad \psum) -K_i-L_i-M_i ,
\end{align*}
from which, we infer
\begin{align*}
    \|\lap \cisum\|_{L^2} &\lesssim \|\pt \cisum\|_{L^2}  +\|\lap \psum\|_{L^2}(1+\|\grad\cio\|_{L^\oo})  + \|K_i\|_{L^2}+\|L_i\|_{L^2}+\|M_i\|_{L^2}  \\
    &\quad   \|\grad\cisum\|_{L^2}^\frac{1}{2}\|\lap\cisum\|_{L^2}^\frac{1}{2}(\|\usum\|_V+\|\lap\psum\|_{L^2}) .
\end{align*}
Applying H\"older inequality gives rise to
\begin{align*}
    \|\lap \cisum\|_{L^2} &\lesssim \|\pt \cisum\|_{L^2}  +\|\lap \psum\|_{L^2}(1+\|\grad\cio\|_{L^\oo})  + \|K_i\|_{L^2}+\|L_i\|_{L^2}+\|M_i\|_{L^2}  \\
    &\quad   \|\grad\cisum\|_{L^2}(\|\usum\|_V^2+\|\lap\psum\|_{L^2}^2),
\end{align*}
which together with (\ref{result 1 in thm 1}-\ref{result 3 in thm 1}), (\ref{estimate of lap psum1}-\ref{estimate of lap psum2}), (\ref{result1 of thm3}-\ref{result2 of thm3}), Lemma \ref{estimate of K_i-P_i} and all the estimates of the layer solutions obtained in Section \ref{analysis of layers} implies the following Corollary:

\begin{cor}\label{cor5.1}
    {\sl
Under the assumptions in Theorem \ref{thm4} and Theorem \ref{thm3}, one has
\begin{equation*}
    \|\cisum\|_{L^\oo_T(H^2)} +  \|\rsum\|_{L^\oo_T(H^2)} \lesssim \ve^\frac{1}{2},
\end{equation*}
and consequently
\begin{equation*}
     \|\cei- \cio\|_{L^\oo_T(H^2)} +  \|\re\|_{L^\oo_T(H^2)} \lesssim \ve^\frac{1}{2}.
\end{equation*}
    }\end{cor}

Now we are in a position to complete the proof of Theorem \ref{m ain thm}.
\begin{proof}[Proof of Theorem \ref{main thm}]
It suffices to verify the assumptions in Corollary \ref{cor5.1}. First of all,  thanks to  (\ref{main thm,assumption2}), Theorem \ref{coexistence of npns} implies the existence of $T$, which equals $+\oo$ in the case $d=2$, so that the systems \eqref{S1eq1} and \eqref{S1eq2} have unique
strong
solution on $[0,T].$
Then in view of  (\ref{main thm,assumption4}), (\ref{solution of left boundary layer of order 2}) and (\ref{solution of right boundary layer of order 2}), one has $$\plbr(0)=\prbr(0)=0.$$
Furthermore, due to (\ref{main thm,assumption1}) and $\Phi^\ve(0)=\pw=\po(0)$, we have
\begin{gather*}
    \cil(0)=\cir(0)=\ciir(0)=\cilmr(0)=\cirmr(0)=0,\\
    \pl(0)=\pr(0)=\pio(0)=\pil(0)=0 \andf \ul(0)=\ur(0)=0,
\end{gather*}
which together with the analysis in Section \ref{analysis of layers} ensures that
\begin{align*}
    \cil=\cir=\ciir=\cilmr=\cirmr=0,\\
    \pl=\pr=\pio=\pil=0, \quad \ul=\ur=0,
\end{align*}
and
\begin{equation}
    \begin{gathered}\label{residue term t=0}
     \cisum(0)=\cirs(0)= \cei(0)-c_i(0),\quad  \psum(0)=\prs(0)=\pe(0)-\psi(0)=0,\\
     \usum(0)=\urs(0)=\ue(0)-u_0,\quad  L_i(0)=M_i(0)=N(0)=P(0)=0,\\
    O(0) = -\usum(0)\cdot\grad u_0-u_0\cdot\grad\usum(0),\\
     K_i(0) = -u_0\cdot\grad\cisum(0)-\usum(0)\cdot\grad c_i(0) + z_iD_i\dive\bigl(\cisum(0)\grad\pw\bigr),\\
    \pt N =- 2\ve^3\pt(f'\p_\xi\plbr - g'\p_\eta \prbr)-\ve^4\pt(f\lap_{x'}\plbr+g\lap_{x'}\prbr +f''\plbr+g''\prbr ).
 \end{gathered}
\end{equation}
Thus all the assumptions in Theorem \ref{thm3} and Theorem \ref{thm4} hold except (\ref{assumption 2 in thm3}).
 
 Next we verify (\ref{assumption 2 in thm3}). In view of (\ref{solution of left boundary layer of order 2}) and (\ref{solution of right boundary layer of order 2}), we infer
\begin{equation}\label{O,K,ptN}
 \begin{gathered}
        \|O(0)\|_{L^2} \lesssim \|u_0\|_{H^3} \|\usum(0)\|_{V},\quad  \|\pt N(0)\|_{L^2} \lesssim \ve^\frac{7}{2}\|\pt \psi(0)\|_{H^2},\\
        \|K_i(0)\|_{L^2} \lesssim \|u_0\|_{H^2}\|\grad \cisum(0)\|_{L^2}  + \|c_i(0)\|_{H^3}\|\usum(0)\|_{L^2} + \|\pw\|_{H^4}\|\grad\cisum(0)\|_{L^2}.
\end{gathered}
\end{equation}

By virtue  of the equation (\ref{equation of cisum}) and (\ref{residue term t=0}), one has
\begin{align*}
    \|\pt\cisum(0)\| &\lesssim \|\usum(0)\|_{V}\|\lap\cisum(0)\|_{L^2} + \|\lap\cisum(0)\|_{L^2} +\|K_i(0)\|_{L^2}.
\end{align*}

Notice that $\rsum(0)=\re(0)=0$, we deduce from  equation (\ref{equation of usum}) that
\begin{align*}
    \|\pt\usum(0)\|_{L^2} \lesssim \|A\usum(0)\|_{L^2} + \|A\usum(0)\|_{L^2}^2 + \|O(0)\|_{L^2}.
\end{align*}

To estimate $\|\grad \pt \psum\|_{L^2}$, we  get, by first multiplying the equation of $\cisum$ by $z_i$ and  summing up  for $i=1,2$, and then inserting the equation (\ref{5.4b}) into the resulting equation, that
\begin{align*}
    -\ve^2 \lap \pt \psum + \pt N + \usum\cdot\grad\rsum &=\sum_{i=1}^2z_iD_i\lap\cisum +\sum_{i=1}^2z_i^2D_i\dive(\cio\grad\psum) \\
    &\quad + \sum_{i=1}^2z_i^2D_i \dive(\cisum\grad\psum) +\sum_{i=1}^2z_i(K_i+L_i+M_i),
\end{align*}
from which, one has
\begin{align*}
    &\lap( \ve^2  \pt \psum +\sum_{i=1}^2z_iD_i\cisum +\sum_{i=1}^2z_i^2D_i\cio\psum) = \usum \cdot \grad \rsum + \sum_{i=1}^2z_i^2D_i\dive(\psum\grad \cio) \\
    &\qquad \qquad \qquad \qquad\qquad \qquad  -\sum_{i=1}^2z_i^2D_i \dive(\cisum\grad\psum) - \sum_{i=1}^2z_i(K_i+L_i+M_i) +\pt N.
\end{align*}
It follows from standard elliptic estimate and  $\psum(0)=0=\rsum(0)$ that
\begin{align*}
    &\ve^2 \|\grad \pt \psum(0) \|_{L^2} \lesssim \sum_{i=1}^2 \|\grad\cisum(0)\|_{L^2} +  \|\pt N(0)\|_{L^2}  +\sum_{i=1}^2\|K_i(0)\|_{L^2}.
\end{align*}

Finally, by inserting (\ref{main thm,assumption1}-\ref{main thm,assumption3}), Lemma \ref{estimate of K_i-P_i} and (\ref{O,K,ptN}) into the above inequalities, we achieve
$$\|\pt\cisum(0)\|_{L^2}+\|\pt\usum(0)\|_{L^2}+\ve\|\grad \pt \psum(0)\|_{L^2} \lesssim \ve^\frac{1}{2}.$$
This verifies (\ref{assumption 2 in thm3}) and we thus completes the proof of Theorem \ref{main thm}.
\end{proof}

\appendix
\section{The proof of Lemma \ref{estimate of K_i-P_i} }\label{appA}

The goal of this section is to present the proof of
 Lemma \ref{estimate of K_i-P_i}.

 \begin{proof}[Proof of
 Lemma \ref{estimate of K_i-P_i}] The estimates of the terms $K_i$ to $P_i$ are straight application of H\"older's inequality.
 We first observe that
\begin{equation*}
    \begin{split}
\| &K_i\|_{L^2}
     \leq M\|\uo + \ve \ul\|_{L^\oo} (\int_\Omega \frac{|\grad \cisum|^2}{\cio} \,dx)^\frac{1}{2} \\
     & + M \| \cio+\ve\cil+\ve^2 \cir+\ve^2\ciir + \ve^2 f \cilbr +\ve^2 g\cirbr+ \ve^2 f \cilmr +\ve^2 g\cirmr) \|_{H^2} \|\usum\|_{H^1}\\
     & +M \| \po+\ve\pl+\ve^2 \pr+ \pio+\ve\pil+\ve^2f\plbr+\ve^2g\prbr \|_{W^{2,\oo}} (\int_\Omega \frac{|\grad \cisum|^2}{\cio} \,dx)^\frac{1}{2} \\
     & +M\ve \| \cil+\ve \cir+\ve\ciir+ \ve f \cilbr +\ve g\cirbr+ \ve f \cilmr +\ve g\cirmr  \|_{W^{1,\oo}} \|\lap \psum\|_{L^2} .
    \end{split}
\end{equation*}
For  $L_i$, one has
\begin{equation*}
\|L_i\|_{L^2}
      \leq M\ve \|\grad \cio \|_{L^\oo} \bigl(\|\xi \p^2_\xi \plbr\|_{L^2} +\|\xi \p^2_\eta \prbr\|_{L^2} \bigr) + M\ve^2 \|\pt \cio\|_{H^2} \|\pio\|_{H^2} .
\end{equation*}

For the estimate of $M_i$, we find
\begin{align*}
&\|M_i\|_{L^2}\leq \ve^3 \|\ul \cdot \grad \cir\|_{L^2} + \ve^2 \|\ur\cdot \grad \cio\|_{L^2} + M\ve^2 \|\lap \ciir\|_{L^2} \\
      & \quad + \ve\|\uo+\ve\ul\|_{L^\oo} \| \grad( \ve\ciir + \ve f \cilbr +\ve g\cirbr+ \ve f \cilmr +\ve g\cirmr)\|_{L^2} \\
      &\quad + M\ve^2 \|f \lap_{x'} \cilbr + g \lap_{x'}\cirbr +f \lap_{x'} \cilmr + g \lap_{x'}\cirmr \|_{L^2}\\
      &\quad+ M\ve^2 \|f'' \cilbr +g''\cirbr+f'' \cilmr +g''\cirmr\|_{L^2} \\
      &\quad +M\ve \|f' \p_\xi \cilbr -  g'\p_\eta \cirbr+f' \p_\xi \cilmr - \ve g'\p_\eta \cirmr\|_{L^2}\\
      &\quad +M\ve^2 \|\grad \cio\|_{L^\oo} \|\grad (f\plbr+g\prbr) \|_{L^2}\\
      &\quad + M \ve^2 \|\cio\|_{L^\oo} \|f''\plbr+g''\prbr+f \lap_{x'} \plbr + g \lap_{x'}\prbr\|_{L^2}\\
      &\quad + M\ve  \|\cio\|_{L^\oo} \|f' \p_\xi \plbr - \ve g'\p_\eta \prbr \|_{L^2} +M\ve \|\cio\|_{H^2}\|\pil\|_{H^2} \\
      &\quad +M\ve^3 \bigl(\|\cil\|_{H^2} \|\pr\|_{H^2} +\|\cir\|_{H^2} \|\pl\|_{H^2} +\ve \|\cir\|_{H^2} \|\pr\|_{H^2}  \bigr) \\
      &\quad + M\ve \| \cil + \ve\cir + \ve \ciir\|_{H^1} \|\grad (\pio+\ve\pil+\ve^2f\plbr+\ve^2g\prbr)\|_{L^2} \\
      &\quad +M\ve \| \cil + \ve\cir + \ve \ciir\|_{L^\oo}\|\lap(\pio+\ve\pil+\ve^2f\plbr+\ve^2g\prbr) \|_{L^2}\\
      &\quad +M\ve \|\grad (\ve f \cilbr +\ve g\cirbr+\ve f \cilmr +\ve g\cirmr)  \|_{L^2}\| \grad (\pio+\ve\pil+\ve^2f\plbr+\ve^2g\prbr) \|_{L^\oo}\\
      &\quad + M\ve^2 \| f \cilbr + g\cirbr+ f \cilmr + g\cirmr)  \|_{L^2}\| \lap (\pio+\ve\pil+\ve^2f\plbr+\ve^2g\prbr) \|_{L^\oo}\\
      &\quad +M\ve \| \grad(\ve\ciir+\ve f \cilbr +\ve g\cirbr+\ve f \cilmr +\ve g\cirmr) \|_{L^2} \|\po+\ve\pl+\ve^2 \pr\|_{H^1} \\
      &\quad +M\ve^2  \| \ciir+ f \cilbr +g\cirbr+f \cilmr + g\cirmr) \|_{L^2}\|\po+\ve\pl+\ve^2 \pr\|_{H^2}\\
      &\quad + M\ve^2\bigl( \|\pt\cilbr\|_{L^2} + \|\pt\cirbr\|_{L^2} \bigr) + M\|\p_\xi^2\plmr\|_{L^2}+M\|\p_\eta^2\prmr\|_{L^2}.
\end{align*}

For  $N_i$, we have
\begin{equation*}
    \begin{split}
\|N_i\|_{L^2}
      &\leq \ve^3 \| \pl\|_{H^2 } + \ve^4 \|\pr\|_{H^2} + \ve^3 \|\pil\|_{H^2} + M\ve^3 \bigl(\|\p_\xi \plbr\|_{L^2} +\|\p_\eta \prbr\|_{L^2}\bigr) \\
      &\quad+ \ve^4 \| f\lap_{x'}\plbr+g\lap_{x'}\prbr +f''\plbr+g''\prbr \|_{L^2}+ \ve^2\|f\p_\xi^2\plmr+g\p_\eta^2\prmr\|_{L^2},
    \end{split}
\end{equation*}
and
\begin{equation*}
    \begin{split}
\|\grad N_i\|_{L^2}
      &\leq \ve^3 \| \pl\|_{H^3 } + \ve^4 \|\pr\|_{H^3} + \ve^3 \|\pil\|_{H^3} + M\ve^3 \bigl(\|\grad_{x'} \p_\xi \plbr\|_{L^2} +\|\grad_{x'}\p_\eta \prbr\|_{L^2}\bigr) \\
      &\quad+ \ve^4 \| f\grad_{x'}\lap_{x'}\plbr+g\grad_{x'}\lap_{x'}\prbr +f''\grad_{x'}\plbr+g''\grad_{x'}\prbr \|_{L^2} \\
      &\quad + M\ve^2 \bigl(\|\p_\xi^2 \plbr\|_{L^2} +\|\p_\eta^2 \prbr\|_{L^2}\bigr)\\
      &\quad + \ve^3 \| \p_\xi(f\lap_{x'}\plbr+f''\plbr)-\p_\eta(g\lap_{x'}\prbr +g''\prbr )\|_{L^2} \\
      &\quad+M\ve^2 \bigl(\|\grad_{x'}\p_\xi^2 \plmr\|_{L^2} +\|\grad_{x'}\p_\eta^2 \prmr\|_{L^2}\bigr) +\ve\| \p_\xi(f\p_\xi^2\plmr)-\p_\eta(g\p_\eta^2\prmr  )\|_{L^2}  ,
    \end{split}
\end{equation*}
and
\begin{equation*}
    \begin{split}
        \|\pt N_i\|_{L^2}
        &\leq \ve^3 \| \pt\pl\|_{H^2 } + \ve^4 \|\pt \pr\|_{H^2} + \ve \|\p_\tau \pil\|_{H^2} + M\ve^3 \bigl(\|\p_\xi \pt \plbr\|_{L^2} +\|\p_\eta \pt \prbr\|_{L^2}\bigr) \\
      &\quad+ \ve^4 \| f\lap_{x'}\pt\plbr+g\lap_{x'}\pt\prbr +f''\pt\plbr+g''\pt\prbr \|_{L^2}\\
      &\quad+\| f\p_\xi^2 \p_\tau\plmr + g\p_\eta^2 \p_\tau\prmr\|_{L^2}.
    \end{split}
\end{equation*}

For $O_i$, one has
\begin{equation*}
    \begin{split}
&\|O_i\|_{L^2}
     \leq \|\uo+\ve\ul\|_{H^2} \| \usum\|_{H^1} \\
     &\quad + M\ve^2 \|\rr+\rir+f\rlbr+g\rrbr+f\rlmr+g\rrmr\|_{L^\oo}\|\lap \psum\|_{L^2} \\
     &\quad + M \|\grad (\po+\ve\pl+\ve^2 \pr+ \pio+\ve\pil+\ve^2f\plbr+\ve^2g\prbr)\|_{L^\oo} \Bigl(\int_\Omega \frac{|\grad \cisum|^2}{\cio} \,dx\Bigr)^\frac{1}{2} .
    \end{split}
\end{equation*}

For  $P_i$, we have
\begin{equation*}
    \begin{split}
\|P_i\|_{L^2}
      &\leq \ve^2 \|\ul\|_{L^\oo}\|\grad \ul\|_{L^2}+\ve^2 \|\rr+\rir+f\rlbr+g\rrbr+f\rlmr+g\rrmr\|_{L^\oo} \\
      &\qquad\quad   \times \|\grad (\po+\ve\pl+\ve^2 \pr+ \pio+\ve\pil+\ve^2f\plbr+\ve^2g\prbr)\|_{L^2}.
    \end{split}
\end{equation*}

By summarizing the above estimates and using the estimates which we obtained in the previous sections, we arrive at
\eqref{S5eq9}. This completes the proof of Lemma \ref{estimate of K_i-P_i}.
\end{proof}

\subsection*{Acknowledgements} Ping Zhang is partially  supported by National Key R$\&$D Program of China under grant
  2021YFA1000800, K. C. Wong Education Foundation and
  by National Natural Science Foundation of China  under Grant 12031006.


\begin{thebibliography}{50}

\bibitem{BFA14} D. Bothe, A. Fischer and J. Saal, Global well-posedness and stability of electrokinetic flows,
 {\it SIAM J. Math. Anal.}, {\bf 46} (2014),  1263-1316.

 \bibitem{B00} Y. Brenier,
Convergence of the Vlasov-Poisson system to the incompressible Euler equations,
{\it Comm. Partial Differential Equations}, {\bf 25} (2000), 737-754.

\bibitem{Constantin & Foias} P. Constantin and C. Foias,
{\it Navier-Stokes equations}, The University of Chicago Press, Chicago, 1988.

\bibitem{CI On the npns system} P. Constantin and M. Ignatova, On the Nernst-Planck-Navier-Stokes system,
 {\it Arch. Ration. Mech. Anal.}, {\bf 232} (2019), 1379-1428.


\bibitem{CIL NPNS far from eq} P. Constantin, M. Ignatova and F.-N. Lee,  Nernst-Planck-Navier-Stokes systems far from equilibrium,
 {\it Arch. Ration. Mech. Anal.}, {\bf 240} (2021), 1147-1168.

\bibitem{CIL Interior electroneutrlity} P. Constantin, M. Ignatova and F.-N. Lee,  Interior electroneutrality in Nernst-Planck-Navier-Stokes systems,
 {\it Arch. Ration. Mech. Anal.}, {\bf 242} (2021), 1091-1118.

\bibitem{CIL Physica D} P. Constantin, M. Ignatova and F.-N. Lee,  Existence of stability of nonequilibrium steady states of Nernst-Planck-Navier-Stokes systems,
 {\it Phys. D}, {\bf 442} (2022), Paper No. 133536.

 \bibitem{FS17} A. Fischer and J. Saal, Global weak solutions in three space dimensions for electrokinetic flow processes,
  {\it J. Evol. Equ.}, {\bf 17} (2017),  309-333.

\bibitem{JS09} J.~W. Jerome and R. Sacco,  Global weak solutions for an incompressible charged fluid with multi-scale couplings: initial-boundary-value problem, {\it Nonlinear Anal.}, {\bf 71} (2009),  2487-2497.

\bibitem{LW20} J.-G.    Liu and J. Wang,  Global existence for Nernst-Planck-Navier-Stokes system in $R^n,$
 {\it Commun. Math. Sci.}, {\bf 18} (2020), 1743-1754.

 \bibitem{lz3} F. Lin and P. Zhang, On the semiclassical limit of Gross-Pitaevski
equations in an exterior domain, {\it Arch. Ration. Mech. Anal.}, {\bf 179} (2005), 79-107.

\bibitem{Ru90} I.  Rubinstein, {\it Electro-diffusion of ions}. SIAM Studies in Applied Mathematics, {\bf 11}. Society for Industrial and Applied Mathematics (SIAM), Philadelphia, PA, 1990.

\bibitem{S09} M. Schmuck, Analysis of the Navier-Stokes-Nernst-Planck-Poisson system, {\it Math. Models Methods Appl. Sci.},
 {\bf 19} (2009),  993-1015.

 \bibitem{WJ21} S. Wang and L. Jiang,  Quasi-neutral limit and the initial layer problem of the electro-diffusion model arising in electro-hydrodynamics, {\it  Nonlinear Anal. Real World Appl.}, {\bf 59} (2021), Paper No. 103266, 25 pp.

\bibitem{WXM06} S. Wang, Z. Xin, and P. A. Markowich, Quasi-neutral limit of the drift-diffusion models for semiconductors: The case of general sign-changing doping profile, {\it SIAM J. Math. Anal.}, {\bf 37} (2006), 1854-1889.

\bibitem{WW12} S. Wang and K. Wang, The mixed layer problem and quasi-neutral limit of the drift-diffusion model for semiconductors, {\it SIAM J. Math. Anal.}, {\bf 44} (2012), 699-717.


\bibitem{Z1} P. Zhang, Wigner measure and the semi-classical limit of
Schr\"odinger-Poisson equation, {\it SIAM J. Math. Analysis},
{\bf 34} (2002), 700-718.


\end{thebibliography}
\end{document}